\documentclass[11pt]{amsart}

\usepackage{hyperref}
\usepackage[utf8]{inputenc}
\usepackage{amsfonts}
\usepackage{amssymb}
\usepackage{amsmath}
\usepackage{amsthm}
\usepackage{enumerate}
\usepackage{mathrsfs}
\usepackage{tikz}

\newcommand{\mathscripty}{\mathscr}

\newcommand{\rs}{\mathord{\upharpoonright}}
\newcommand{\sm}{\setminus}

\newcommand{\e}{\varepsilon}

\newcommand{\GG}{\mathbb{G}}

\newcommand{\NN}{\mathbb{N}}

\newcommand{\RR}{\mathbb{R}}
\newcommand{\er}{\mathbb{R}}

\newcommand{\ce}{\mathbb{C}}

\newcommand{\PP}{\mathbb{P}}

\newcommand{\Cstar}{\mathrm{C}^*}
\newcommand{\Cmeas}{\mathrm{C}}

\newcommand{\CT}{\mathcal{T}}
\newcommand{\CI}{\mathcal{I}}
\newcommand{\CJ}{\mathcal{J}}
\newcommand{\CX}{\mathcal{X}}
\newcommand{\CY}{\mathcal{Y}}

\newcommand{\SA}{\mathscripty{A}}
\newcommand{\SB}{\mathscripty{B}}
\newcommand{\SC}{\mathscripty{C}}
\newcommand{\SD}{\mathscripty{D}}
\newcommand{\SE}{\mathscripty{E}}
\newcommand{\SF}{\mathscripty{F}}
\newcommand{\SG}{\mathscripty{G}}
\newcommand{\SH}{\mathscripty{H}}

\newcommand{\SK}{\mathscripty{K}}
\newcommand{\SL}{\mathscripty{L}}
\newcommand{\SM}{\mathscripty{M}}

\newcommand{\SP}{\mathscripty{P}}
\newcommand{\SQ}{\mathscripty{Q}}
\newcommand{\SR}{\mathscripty{R}}

\newcommand{\SU}{\mathscripty{U}}

\newcommand{\SX}{\mathscripty{X}}
\newcommand{\SY}{\mathscripty{Y}}

\newcommand{\ZFC}{\mathrm{ZFC}}

\newcommand{\CH}{\mathrm{CH}}
\newcommand{\MA}{\mathrm{MA}}

\newcommand{\OCA}{\mathrm{OCA}}
\newcommand{\PFA}{\mathrm{PFA}}

\newcommand{\U}{\SU}

\DeclareMathOperator{\Skel}{\bf{Skel}}

\newtheorem{theorem}{Theorem}[section]
\newtheorem*{theorem*}{Theorem}
\newtheorem{proposition}[theorem]{Proposition}
\newtheorem*{proposition*}{Proposition}
\newtheorem{lemma}[theorem]{Lemma}
\newtheorem*{lemma*}{Lemma}
\newtheorem{corollary}[theorem]{Corollary}
\newtheorem*{corollary*}{Corollary}

\newtheorem*{fact*}{Fact}
\theoremstyle{definition}
\newtheorem{definition}[theorem]{Definition}
\newtheorem*{definition*}{Definition}
\newtheorem{claim}[theorem]{Claim}
\newtheorem*{claim*}{Claim}
\newtheorem{conjecture}[theorem]{Conjecture}
\newtheorem*{conjecture*}{Conjecture}
\newtheorem{notation}[theorem]{Notation}
\newtheorem{theoremi}{Theorem}
\newtheorem{corollaryi}[theoremi]{Corollary}
\newtheorem{definitioni}[theoremi]{Definition}
\newtheorem{conjecturei}[theoremi]{Conjecture}

\theoremstyle{remark}

\newtheorem*{example*}{Example}
\newtheorem{remark}[theorem]{Remark}
\newtheorem*{remark*}{Remark}

\newtheorem*{note*}{Note}
\newtheorem{question}[theorem]{Question}
\newtheorem*{question*}{Question}

\newcommand{\set}[2]{\left\{#1\mathrel{}\middle|\mathrel{}#2\right\}}
\newcommand{\seq}[2]{\left\langle #1\mathrel{}\middle|\mathrel{}#2\right\rangle}

\newcommand{\norm}[1]{\left\lVert #1 \right\rVert}

\DeclareMathOperator{\spann}{span}
\DeclareMathOperator{\ran}{ran}
\DeclareMathOperator{\dom}{dom}
\DeclareMathOperator{\supp}{supp}

\DeclareMathOperator{\Fin}{Fin}

\DeclareMathOperator{\Aut}{Aut}

\usepackage{enumitem}

\newcommand{\iso}{\Lambda}

\title{Forcing axioms and coronas of $\Cstar$-algebras}%
\author[P. McKenney]{Paul McKenney}
\address[P. McKenney]{Department of Mathematics \\
Miami University \\
501 E. High Street \\
Oxford, Ohio 45056 USA
}
\email{mckennp2@miamioh.edu}
\urladdr{}

\author[A. Vignati]{Alessandro Vignati}
\address[A. Vignati]{
Institut de Math\'ematiques de Jussieu - Paris Rive Gauche (IMJ-PRG)\\
Universit\'e Paris Diderot\\
B\^atiment Sophie Germain\\
8 Place Aur\'elie Nemours \\ 75013 Paris, France}
\email{ale.vignati@gmail.com}
\urladdr{http://www.automorph.net/avignati}

\date{\today}%
\begin{document}

\begin{abstract}
 We prove rigidity results for large classes of corona algebras, assuming the Proper Forcing Axiom. In particular, we prove that a conjecture of Coskey and Farah holds for all separable $\Cstar$-algebras with the metric approximation property and an increasing approximate identity of projections.
\end{abstract}
\maketitle

\section{Introduction}
The Weyl-von Neumann theorem, one of the fundamental results of operator theory, classifies the self-adjoint elements of the Calkin algebra\footnote{$H$ denotes a separable complex Hilbert space. We write $\mathcal B(H)$ for the $\Cstar$-algebra of bounded linear operators on $H$, and $\mathcal K(H)$ for its ideal of compact operators. The quotient, $\SQ(H)=\mathcal B(H)/\mathcal K(H)$, is called the Calkin algebra.} $\SQ(H)$ up to unitary equivalence via their spectra. This classification depends heavily on the fact that a self-adjoint element of $\SQ(H)$ necessarily lifts to a self-adjoint operator in $\mathcal B(H)$.  Berg (\cite{Berg}) and Sikonia (\cite{Sikonia}) independently generalized such classification to those operators in $\SQ(H)$ which are images, via the canonical quotient map, of normal operators of $\mathcal B(H)$. Since normal elements of $\SQ(H)$ do not necessarily lift to normal operators in $\mathcal B(H)$, an extension of the Weyl-von Neumann theorem to normal elements of $\SQ(H)$ had to wait for the pioneering work of Brown, Douglas and Fillmore, who in \cite{BDF.ext} illuminated deep connections between algebraic topology, index theory, and extensions of $\Cstar$-algebras, and brought to light new questions about $\SQ(H)$. Among them, the following: does $\SQ(H)$ have an automorphism which sends the unilateral shift $S$ to its adjoint? Or, even weaker, does $\SQ(H)$ have an outer automorphism?  (Since inner automorphisms of $\SQ(H)$ preserve the Fredholm index, they cannot send the unilateral shift to its adjoint.)

These problems remained open for decades, until Phillips and Weaver (\cite{Phillips-Weaver}) showed that the Continuum Hypothesis ($\CH$) implies the existence of outer automorphisms of $\SQ(H)$, and Farah (\cite{Farah.C}) showed that the Open Colouring Axiom ($\OCA$) implies that every automorphism of $\SQ(H)$ is inner. These results show that the existence of outer automorphisms of $\SQ(H)$ is independent of $\ZFC$. Whether there can be an automorphism sending the shift to its adjoint in some model of $\ZFC$ remains open.

Farah's theorem is just one of many results illustrating the effect of the Proper Forcing Axiom $\PFA$ (of which $\OCA$ is a consequence) on the rigidity of certain uncountable quotient structures; similarly, $\CH$ has often been found to have the opposite effect on the same structure (\cite{Just, Velickovic.OCAA, Farah.AQ}). In retrospect, the proofs of most of these rigidity results take the following common form: 
\begin{itemize}
\item first, one uses $\PFA$ to show that every automorphism of the structure of interest is, in a  canonical way, determined by a Borel subset of $\RR$. These automorphisms are `absolute', being present in every model of $\ZFC$ (see the argument, using Shoenfield's absoluteness Theorem \cite[Theorem 13.15]{Kanamori}, present in the proof of \cite[Lemma 7.2]{Coskey-Farah}); we think of them as being trivial in a topological sense. 
\item Second, one shows, using only $\ZFC$,  that such topologically trivial automorphisms must have a certain desired algebraic structure (depending on the nature of the quotients of interests, e.g., groups, algebras, etc..). This second part is tightly connected with Ulam stability, the study of whether approximate morphisms can be perturbed uniformly to on-the-nose morphisms; the connections between Ulam stability and quotient rigidity were first explored in  \cite{Farah.Lifts} (see also \cite[\S1.4--1.8]{Farah.AQ}). 
\end{itemize}
This subject started with the analysis of the Boolean algebra $\SP(\NN)/\Fin$. Here topologically trivial automorphisms and the ones admitting a lift preserving the algebraic structure coincide (\cite[Theorem 1.2]{Velickovic.OCAA}), and we refer at these automorphisms simply as trivial. An automorphism of $\SP(\NN)/\Fin$ is then trivial if it is induced by an almost permutation of $\NN$ (see \S\ref{subsec:ForcingAxiom}). Countable saturation implies that, under $\CH$, $\SP(\NN)/\Fin$ has $2^{\aleph_1}$ automorphisms and in particular nontrivial ones. On the other hand, Shelah, in the groundbreaking \cite{Shelah.PF}, showed the consistency of the statement `All automorphisms of $\SP(\NN)/\Fin$ are trivial'. Shelah's intricate forcing construction was replaced by the assumption of $\PFA$ in \cite{Shelah-Steprans.PFAA}; this argument was later simplified by \cite{Velickovic.OCAA} where it was shown that all automorphisms of $\SP(\NN)/\Fin$ are trivial under $\OCA+\MA_{\aleph_1}$. ($\MA_{\aleph_1}$ is Martin's Axiom at level $\aleph_1$, another consequence of $\PFA$, see \cite[\S3]{Kunen.2011}). Other relevant results along these lines were obtained for automorphisms of quotients of $\SP(\NN)$ by analytic ideals different than $\Fin$ (\cite{Just,Farah.AQ}), or considered question regarding the existence of embeddings between different quotients (\cite{DowHart.ContImg,DowHart.Measure,Farah.AQ}).

On the opposite side of the spectrum, $\CH$ can usually be used with back-and-forth arguments to show that there are too many automorphisms for all of them to be determined by a Borel subset of $\RR$. It is often the case that $\CH$ is used together with certain degrees of countable saturation to show the existence of $2^{\aleph_1}$ automorphisms (e.g., \cite[Theorem 2.13]{Farah-Hart.CS}). That $\CH$ provides the optimal set-theoretic assumption to show that the existence of a large amount of automorphisms of a given quotient structure is a consequence of Woodin's $\Sigma_1^2$-absoluteness (\cite{Woodin.Abs}).

In the category of $\Cstar$-algebras, the objects relevant to this nonrigidity/rigidity phenomena are \emph{corona algebras}. This class includes the Calkin algebra $\SQ(H)$ and all algebras of the form $C(\beta X\sm X)$, where $X$ is locally compact, noncompact, and Hausdorff. Corona algebras form a wide class of $\Cstar$-algebras, since to every nonunital $\Cstar$-algebra $A$ one may associate its corona algebra $\SQ(A)$, which is the quotient $\SQ(A) = \SM(A)/A$, where $\SM(A)$ is the \emph{multiplier algebra} of $A$, defined (up to isomorphism) to be the largest unital $\Cstar$-algebra containing $A$ as an essential ideal. The quotient map $\SM(A)\to\SQ(A)$ is denoted by $\pi_A$. Multiplier and corona algebras have been of use to $\Cstar$-algebraists since at least the 1960's, when Busby showed that the extensions of $A$ by $B$ are determined (up to a certain notion of equivalence) by $^*$-homomorphisms from $A$ into $\SQ(B)$, see~\cite{Busby}. 

Following the strategy highlighted earlier, at first we give a topological notion of triviality, and we show that under Forcing Axioms automorphisms of coronas of separable $\Cstar$-algebras (with some technical assumptions, see Theorem~\ref{theoi:Borel}) are topologically trivial. Secondly, we provide algebraic notions of triviality and we discuss how these are related to Ulam stability phaenomena.

Even though $\SM(A)$ is nonseparable in norm topology, there is a second topology turning $\SM(A)$ into a standard Borel space. For $a\in A$, define two seminorms on $\SM(A)$ by $\ell_a(x)=\norm{xa}$ and $r_a(x)=\norm{ax}$. The \emph{strict topology} on $\SM(A)$ is the weak topology induced by $\{\ell_a,r_a\mid a\in A\}$. If $A$ is separable, then $\SM(A)$ is separable and metrizable in its strict topology;  moreover every closed norm-bounded subset of $\SM(A)$ is Polish in this topology (see for example \cite[\S13.1]{Farah.Book}). In this category, the triviality requirements for an automorphism $\varphi\in\Aut(\SQ(A))$ are given by the existence of a well behaved \emph{lift}, that is, a map $\Phi\colon\SM(A)\to\SM(A)$ preserving some (topological or algebraic) properties and making the following diagram commute:
\begin{center}
\begin{tikzpicture}
 \matrix[row sep=1cm,column sep=2cm] {
&\node (A1) {$\SM(A)$};
& \node (A2) {$\SM(A)$};
\\
&\node (B1) {$\SQ(A)$};
& \node (B2) {$\SQ(A)$};
\\
};
\draw (A1) edge[->] node [above] {$\Phi$} (A2) ;
\draw (A1) edge[->] node [left] {$\pi_A$} (B1) ;
\draw (A2) edge[->] node [right] {$\pi_A$} (B2) ;
\draw (B1) edge[->] node [above] {$\varphi$} (B2) ;
\end{tikzpicture}.
\end{center}

In \cite{Coskey-Farah}, Coskey and Farah formalized a notion of `topologically trivial' for automorphisms of a corona algebra $\SQ(A)$, when $A$ is separable. They purposefully chose a broad notion, since their intent was to show the existence of many nontrivial automorphisms under the assumption of $\CH$. 
\begin{definitioni}\label{defin:Boreliso}
 Let $A$ and $B$ be separable $\Cstar$-algebras. A function $\iso\colon\SQ(A)\to\SQ(B)$ is \emph{topologically trivial} if its \emph{graph}
 \[
 \Gamma_\iso= \set{(a,b)\in\SM(A)_{\le 1}\times \SM(B)_{\le 1}}{\iso(\pi_A(a)) = \pi_B(b)}
 \]
 is Borel in the product of the strict topologies, where $\SM(A)_{\leq 1}$ and $\SM(B)_{\le 1}$ denote the closed unital balls of $\SM(A)$ and $\SM(B)$ respectively.
\end{definitioni}

Aiming to capture the rigidity phenomena described above, Coskey and Farah made the following two conjectures:

\begin{conjecturei}\label{conj:theconj}
 Suppose $A$ is a separable, nonunital $\Cstar$-algebra. Then,
 \begin{enumerate}
 \item\label{theconj:CH} $\CH$ implies that there are $2^{2^{\aleph_0}}$ automorphisms of $\SQ(A)$ which are not topologically trivial, and
 \item\label{theconj:PFA} $\PFA$ implies that every automorphism of $\SQ(A)$ is topologically trivial.
 \end{enumerate}
\end{conjecturei}

In~\cite{Coskey-Farah}, part~\ref{theconj:CH} of Conjecture~\ref{conj:theconj} was verified for a wide class of $\Cstar$-algebras, including the case where $A$ is simple or stable; a large class of abelian $\Cstar$-algebras have been dealt with in \cite{V.Nontrivial}. Our focus is on part~\ref{theconj:PFA} of the conjecture, which we confirm for a large class of $\Cstar$-algebras. (Part~\ref{theconj:PFA} of the conjecture was finally verified for all separable $\Cstar$-algebras in \cite{V.Rigidity}; this result heavily relies on the `noncommutative $\OCA$ lifting Theorem' proved here as Theorem~\ref{thm:lifting}). For the definition of the metric approximation property, see \S\ref{subsec:cstar}.

\begin{theoremi}\label{theoi:Borel}
 Assume $\OCA$ and $\MA_{\aleph_1}$. Let $A$ and $B$ be separable $\Cstar$-algebras, each with an increasing approximate identity of projections, and suppose $A$ has the metric approximation property. Then every isomorphism $\SQ(A) \to \SQ(B)$ is topologically trivial.
\end{theoremi}

The metric approximation property is known to hold for a large class of $\Cstar$-algebras including, but not limited to, all exact $\Cstar$-algebras, and therefore all nuclear $\Cstar$-algebras. It does not hold for all separable $\Cstar$-algebras, however (\cite{Szank:AP}). Combined with the main result of \cite{Coskey-Farah}, Theorem~\ref{theoi:Borel} gives the following corollary:
\begin{corollaryi}
 Let $A$ be a separable unital $\Cstar$-algebra with the metric approximation property. Then the statement ``all automorphisms of $\SQ(A\otimes\mathcal K(H))$ are topologically trivial'' is independent of $\ZFC$.
\end{corollaryi}

We turn to discuss the second part of the rigidity problem, concerned in inferring, in $\ZFC$, that topologically trivial automorphisms necessarily carry some algebraic properties. We focus on reduced products. (An algebraic notion of triviality for general coronas was formulated in \cite{V.Rigidity}). Given a sequence of unital $\Cstar$-algebras $A_n$, for $n\in\NN$, the reduced product of the sequence $A_n$,  $\prod A_n / \bigoplus A_n$, is the corona algebra of $\bigoplus A_n$.  Suppose that $A_n$ and $B_n$ are $\Cstar$-algebras and $\varphi_n\colon A_n\to B_n$ is an isomorphism for all but finitely many $n$. The sequence $(\varphi_n)_n$ induces an isomorphism between $\prod A_n/\bigoplus A_n$ and $\prod B_n/\bigoplus B_n$. If we allow reindexing of the algebras $B_n$, we have our definition of \emph{algebraically trivial} isomorphism (see~\ref{defin:trivialredprod} for the specific definition).  We study whether all topologically trivial isomorphisms of reduced products are algebraically trivial, and we connect this with Ulam stability phaenomena and the notion of $\e$-isomorphisms.  (Here an \emph{$\e$-isomorphism} is a map which preserves all of the $\Cstar$-algebraic operations on the unit ball, up to an error (in norm) of $\e$, see~Definition~\ref{defin:apmaps} for a precise definition.) In doing so, we introduce an intermediate notion of algebraic triviality for isomorphisms of reduced products, referred to as \emph{asymptotical algebraicity} in Definition~\ref{defin:trivialredprod}. 
 The following is a consequence of Theorem~\ref{thm:trivialredprod} and Proposition~\ref{prop:trivials}.
\begin{theoremi}\label{theoi:redprod}
Let $A_n$ and $B_n$, for $n\in\NN$,  be unital, separable $\Cstar$-algebras with no nontrivial central projections, and suppose that each $A_n$ has the metric approximation property. Then all topologically trivial isomorphisms between $\prod A_n/\bigoplus A_n$ and $\prod B_n/\bigoplus B_n$ are asymptotically algebraic.
\end{theoremi}

The hypotheses of the theorem are satisfied, for example, if each $A_n$ and each $B_n$ is a simple nuclear $\Cstar$-algebra. The above result shows that two such reduced products are isomorphic under $\OCA+\MA_{\aleph_1}$ only if they are isomorphic provably in $\ZFC$. Parovi\v{c}enko's Theorem (\cite{Parovicenko}) in the commutative case, and the results of Ghasemi in \cite{Ghasemi.FDD} in the noncommutative case, give example of reduced products which are isomorphic under $\CH$ but are not isomorphic in $\ZFC$.

The next natural question to ask is whether all asymptotically algebraic isomorphisms of reduced products are algebraically trivial. This question is related to Ulam stability in the same way as Farah related isomorphisms of quotients of the form $\SP(\NN)/\mathcal I$ and approximate morphisms between finite Boolean algebras (\cite[\S4]{Farah.Lifts}). Motivated by the conclusion of Theorem~\ref{theoi:redprod}, we study the conditions under which two $\e$-isomorphic $\Cstar$-algebras $A$ and $B$ must be isomorphic, with a focus on whether we can make $\e$ uniform over all $A$ and $B$ in some class of $\Cstar$-algebras. These are known as Ulam stability phaenomena (see \S\ref{sec:apmaps} for the specifics). In Corollaries~\ref{cor:trivialredprod1}--\ref{cor:trivialredprod4}, we obtain positive results for certain well behaved classes of $\Cstar$-algebras. The statement in Corollary~\ref{cor:trivialredprod1} was shown to be consistent in (\cite{Ghasemi.FDD}), and to follow from Forcing Axioms in \cite{McKenney.UHF}.

The proofs of Theorem~\ref{theoi:Borel} and Theorem~\ref{theoi:redprod} are based on a technical and powerful lifting result, Theorem~\ref{thm:lifting}, which is the noncommutative extension of the `$\OCA$ lifting theorem', \cite[Theorem~3.3.5]{Farah.AQ}. 
Our version provides well behaved liftings for maps of the form
\[
 \psi\colon \prod E_n / \bigoplus E_n \to \SQ(A)
\]
where each $E_n$ is a finite-dimensional Banach space, $A$ is a separable $\Cstar$-algebra, and the map $\psi$ preserves a limited set of algebraic operations. Many other lifting results in this area can be viewed as restricted versions of Theorem~\ref{thm:lifting}. By Stone duality, Boolean algebra correspond to zero-dimensional topological spaces, which in turns, thanks to Gelfand-Naimark duality, correspond to real rank zero abelian $\Cstar$-algebras. Using this, one can rephrase Farah's results for quotients of Boolean algebra, and show that the `$\OCA$ lifting theorem' of~\cite{Farah.AQ} applies to the case where each $E_n$ is $\mathbb C$ and $A=c_{\mathcal I}$, the algebra those elements in $\ell_\infty$ vanishing on $\mathcal I$ for some analytic $P$-ideal $\mathcal I\subseteq\SP(\NN)$. 

The proof of our lifting theorem makes up the technical core of the paper, and most of our other results are derived from it; we should therefore spend few words on it. Our proof mimics that of \cite[Theorem~2.1]{Velickovic.OCAA} in its general structure and strategy. Indeed, \cite{Velickovic.OCAA} has served as a blueprint for many results regarding rigidity of isomorphisms of quotient structures from $\OCA + \MA_{\aleph_1}$. (The blueprint for homomorphism was designed in \cite{Farah.AQ}). The idea of showing that the ideal of subsets of $\NN$ on which one has a well behaved lift cannot miss an uncountable treelike almost disjoint family was already present in \cite[Lemma 2.2]{Velickovic.OCAA}, and the general strategy mimics the one  of  \cite[\S3]{Farah.AQ}. In particular, the proof of Lemma~\ref{lemma:C->asymptotic2} uses a modification of the notion of stabilizers, already present in the work of Just (\cite{Just.EU,Just}), see also \cite[Definition 3.11.5]{Farah.AQ}. 

Although the strategy to prove Theorem~\ref{thm:lifting} was similar to the one of \cite[\S3]{Farah.AQ}, ours is not  merely a technically challenging adaptation to a broader setting of the work of Farah's. Just to mention a few, the lack of commutativity, the fact that we might not have projections in $\SM(A)$, and the fact that $A$ is not necessarily an AF algebra, each add a level of difficulty which we bypass adding new ideas to the blueprint provided by Veli\v{c}kovi\'c and Farah. Another novelty can be found in the proof of Theorem~\ref{theoi:Borel}, where we use a stratification for the corona of $A$ which takes inspiration from the one used for the Calkin algebra in \cite{Farah.C} (in fact this idea comes from \cite[Theorem~3.1]{Elliott.Der2}). In this case, the use of internal approximation properties (such as the Metric Approximation Property) to obtain a stratification indexed by partition of $\NN$ into consecutive intervals \emph{and} by functions in ${}^{\NN}\NN$ is completely new. (This was further refined in \cite{V.Rigidity}.)




The paper is structured as follows. \S\ref{sec:prel} is dedicated to preliminaries from both set theory and operator algebras. In \S\ref{sec:apmaps} we define the notion of an $\varepsilon$-isomorphism and prove several results concerning $\e$-isomorphisms and their relationship to isomorphisms between reduced products. In \S\ref{sec:liftstatements} and \S\ref{sec:lemma_proofs} we state and then prove our lifting theorem (Theorem~\ref{thm:lifting}). In \S\ref{sec:cons.RedProd1}, \S\ref{sec:cons.RedProd2}, \S\ref{sec:cons.Borel}, and \S\ref{sec:cons.NonembLift} we provide several consequences of Theorem~\ref{thm:lifting}, including Theorems~\ref{theoi:Borel} and \ref{theoi:redprod}. Finally, in \S\ref{sec:conclusion} we offer some open problems.
\subsubsection*{Acknowledgements}
This work started in 2016 when AV was a student at York University and visited Miami University, supported by the grants of I. Farah and C. Eckart. Since then, AV was partially supported by a Susan Man Scholarship, a Prestige co-fund programme and a FWO scholarship. Currently AV is partially funded by the ANR Project AGRUME (ANR-17-CE40-0026). We would like to thank all the funding bodies and hosting institutions.
We are indebted to Ilijas Farah and the anonymous referee for many useful remarks.



\section{Preliminaries}\label{sec:prel}
Here we define the main objects of our investigation and we record few well known facts about them.

\subsection{Descriptive set theory and ideals in $\SP(\NN)$}\label{subsec:DSTandIdeals}

A topological space is \emph{Polish} if it is separable and completely metrizable. As all compact metrizable spaces are Polish, so is $\SP(\NN)$ when identified with $2^{\NN}$. If $X$ is Polish, $Y\subset X$ is \emph{meager} if it is the union of countably many nowhere dense sets.  The following characterizes comeager subsets of certain compact spaces; it is a consequence of the work of Jalali-Naini and Talagrand (\cite{Jalali-Naini, Talagrand.Compacts}). For its proof, see~\cite[\S3.10]{Farah.AQ}.

\begin{proposition}
 \label{prop:talagrand}
Let $Y_n$ be finite sets, for $n\in\NN$. A set $\SG\subseteq\prod Y_n$ is comeager if and only if there is a partition $\seq{E_i}{i\in\NN}$ of $\NN$ into finite intervals and a sequence $t_i\in \prod_{n\in E_i} Y_n$, such that $y\in\SG$ whenever $\{i\mid y\restriction (\prod_{n\in E_i} Y_n)=t_i\}$ is infinite.\qed
\end{proposition}

If $X$ is Polish, $Y\subseteq X$ is \emph{Baire-measurable} if it has meager symmetric difference with an open set, and \emph{analytic} if it is the continuous image of a Borel subset of a Polish space. Every analytic subset of a Polish space is Baire-measurable. If $X$ and $Y$ are Polish spaces, a function $f\colon X\to Y$  is $\Cmeas$-\emph{measurable} if for every open $U\subseteq Y$, $f^{-1}(U)$ is in the $\sigma$-algebra generated by the analytic sets. $\Cmeas$-measurable functions are Baire-measurable (see~\cite[Theorem~21.6]{Kechris.CDST}).
We state two descriptive set-theoretic results about uniformization of functions. The first is the Jankov-von Neumann Theorem\footnote{It has to be noted that von Neumann's original motivation for Theorem~\ref{thm:JVN} came from the study of operator algebras, see e.g., \cite[\S16, Lemma 5]{vonNeumannROO}}.

\begin{theorem}[{{\cite[Theorem~18.1]{Kechris.CDST}}}]\label{thm:JVN}
Let $X$ and $Y$ be Polish and $A\subseteq X\times Y$ be analytic. Let $D = \set{x\in X}{\exists y\in Y\; (x,y)\in A}$. Then there exists a $\Cmeas$-measurable function $f\colon X\to Y$ such that for all $x\in D$, $(x,f(x))\in A$. We say that $f$ \emph{uniformizes} $A$.\qed
\end{theorem}

In general it is not possible to uniformize a Borel set with a Borel function. This is however possible when the vertical sections of $A$ are well behaved.

\begin{theorem}[{{\cite[Theorem~8.6]{Kechris.CDST}}}]\label{thm:BorelUnif}
 Let $X,Y$ be Polish and $A\subseteq X\times Y$ be a Borel set with the property that its vertical sections are either empty or nonmeager. Then there is a Borel function uniformizing $A$.\qed
\end{theorem}

\subsubsection*{Ideals in $\SP(\NN)$}

A set $\SH\subseteq\SP(\NN)$ is \emph{hereditary} if for all $B\in\SH$ and $A\subseteq B$, we have $A\in\SH$. Proposition~\ref{prop:talagrand} implies that the intersection of finitely many hereditary and nonmeager subsets of $\SP(\NN)$ is hereditary and nonmeager (\cite[\S3.10]{Farah.AQ}). This extends to countable intersections when the sets contain $\Fin$, the ideal of finite subsets of $\NN$.

A subset $\CI\subseteq\SP(\NN)$ is an \emph{ideal} on $\NN$ if it is hereditary and closed under finite unions. $\CI$ is \emph{proper} if $\CI\neq\SP(\NN)$. All ideals are, unless otherwise stated, assumed to be proper. An ideal $\CI$ is \emph{dense} if for every infinite $X\subseteq\NN$ there is an infinite $Y\subseteq X$ with $Y\in\CI$. Ideals are in duality with filters: if $\CI$ is an ideal, then $\CI^*=\{X\subseteq\NN\mid\NN\setminus X\in\CI\}$ is a filter. A proper ideal $\CI$ is maximal if and only if $\CI^*$ is an ultrafilter.
Proposition~\ref{prop:talagrand} implies the following:\footnote{This characterises the ideal of finite sets as minimal, with respect to the Rudin-Blass ordering, among all ideals with the Baire Property.}
\begin{proposition}[{{\cite{Jalali-Naini, Talagrand.Compacts}}}]\label{prop:JT2}
 Let $\CJ\subseteq\SP(\NN)$ be an ideal containing $\Fin$. Then the following are equivalent:
 \begin{enumerate}
 \item $\CJ$ has the Baire Property;
 \item $\CJ$ is meager;
 \item\label{Jal3} there is a partition $\seq{E_i}{i\in\NN}$ of $\NN$ into finite intervals such that for any infinite set $L$, $\bigcup_{n\in L} E_n$ is not in $\CJ$.\qed
 \end{enumerate}
\end{proposition}
Notice that \eqref{Jal3} is equivalent to the existence of an algebra embedding $\varphi\colon \SP(\NN)/\Fin\to\SP(\NN)/\CJ$ admitting a lifting $\Phi\colon\SP(\NN)\to\SP(\NN)$ which maps finite sets to finite sets and is completely additive (i.e., $\Phi(\bigcup A_n)=\bigcup\Phi(A_n)$ whenever $A_n\subseteq \NN$, for $n\in\NN$, are sets), or, equivalently, continuous.

A family $\SF\subseteq\SP(\NN)$ of infinite sets is \emph{almost disjoint} if for every distinct $A, B\in\SF$ we have that $A\cap B$ is finite. An almost disjoint family is \emph{treelike} if there is a bijection $f \colon \NN\to 2^{<\omega}$ such that for every $A\in\SF$, $f[A]$ is a branch through $2^\omega$, i.e., a pairwise comparable subset of $2^{<\omega}$. An ideal $\CI\subseteq\SP(\NN)$ is \emph{ccc/Fin} if $\CI$ meets every uncountable, almost disjoint family $\SF\subseteq\SP(\NN)$. An easy argument shows that if $\CJ$ satisfies condition~3 of Proposition~\ref{prop:JT2}, then there is an almost disjoint family of size continuum which is disjoint from $\CJ$. (The same is true if $\CJ$ does not contain the finite sets). Thus, every ccc/Fin ideal is nonmeager.

\subsection{Forcing axioms and their consequences}\label{subsec:ForcingAxiom}

Forcing axioms were introduced as extensions of the Baire Category Theorem. For a comprehensive account on (some of) them, see for example~\cite{Todorcevic.FA}.

One of the most studied forcing axioms is the Proper Forcing Axiom ($\PFA$), introduced by Baumgartner in \cite{Baum}, although implicitly present in earlier work of Shelah (\cite{Shelah.PF}). Its consistency was proved in \cite{Shelah.PF}. We focus on two consequences of $\PFA$: Todor\v cevi\'c's $\OCA$ and $\MA_{\aleph_1}$, a local version of Martin's Axiom.

We work with a formally stronger version of $\OCA$ defined by Farah in \cite{Farah.LG}. This axiom is known as $\OCA_\infty$. Recently, it was shown by Justin Moore (\cite{Moore.OCA}) that $\OCA_\infty$ and $\OCA$ are equivalent, therefore while assuming the second, we often use the first. The axiom currently known as $\OCA$ has roots in the work of Baumgartner (\cite{Baum.Aleph1}), and it is a modification of several colouring axioms appearing in work of Abraham, Rubin, and Shelah (\cite{ARS}). For a more detailed historical account on the evolution of this axiom, see \cite[\S8]{Todorcevic.PPIT}. 

We write $[X]^2$ to indicate the set of unordered pairs of elements of a set $X$. The axiom $\OCA_\infty$ asserts the following: for every separable metric space $X$ and every sequence of partitions $[X]^2=K_0^n\cup K_1^n$, if every $K_0^n$ is open in the product topology on $[X]^2$, and $K_0^{n+1}\subseteq K_0^n$ for
every $n$, then either
\begin{enumerate}
 \item there are $X_n$ ($n\in \NN$) such that $X=\bigcup_n X_n$ and $[X_n]^2\subseteq K_1^n$ for every $n$, or
 \item there is an uncountable $Z\subseteq 2^\NN$ and a continuous bijection $f\colon Z\to X$ such that for all distinct $x$ and $y$ in $Z$ we have
 \[
 \{f(x),f(y)\}\in K_0^{\Delta(x,y)}
 \]
 where $\Delta(x,y)=\min\set{n}{x(n)\neq y(n)}$.
\end{enumerate}
The original statement of $\OCA$ is the restriction of $\OCA_{\infty}$ to the case where $K_0^n=K_{0}^{n+1}$ for every $n$, but, as mentioned, the two are equivalent.

$\OCA$ contradicts $\CH$. Moreover, $\OCA$ implies that $\mathfrak b = \omega_2$, where $\mathfrak b$ is the minimal cardinality of a family of functions in $\NN^\NN$ that is unbounded with respect to the relation $f\leq^*g$ defined by $f\leq^*$ if there is $m\in\NN$ such that $f(n)\leq g(n)$ whenever $n\geq m$, see \cite[\S8]{Todorcevic.PPIT}.

Let $\PP$ be a partially ordered set (poset). Two elements of $\PP$ are \emph{incompatible} if there is no element of $\PP$ below both of them. A set of pairwise incompatible elements is an \emph{antichain}. If all antichains of $\PP$ are countable, $\PP$ has the \emph{countable chain condition} (ccc). A set $D\subseteq \PP$ is \emph{dense} if $\forall p\in\PP\exists q\in D$ with $q\leq p$. A \emph{filter} is a subset $G$ of $\PP$ which is upward closed (that is, if $p\in G$ and $p\leq q$, then $q\in G$) and downward directed (meaning if $p$ and $q$ are in $G$, then there is $r\in G$ with $r\leq p,g$).

Martin's Axiom at the cardinal $\kappa$ (written $\MA_\kappa$) states: for every poset $(\PP,\leq)$ that has the ccc, and every family of dense subsets $D_\alpha\subseteq\PP$ ($\alpha < \kappa$), there is a filter $G\subseteq\PP$ with $G\cap D_\alpha\neq\emptyset$ for every $\alpha < \kappa$. $\MA_{\aleph_0}$ is a theorem of $\ZFC$, as is the negation of $\MA_{2^{\aleph_0}}$. In particular, $\MA_{\aleph_1}$ contradicts $\CH$.

For many of the results of this paper we assume $\OCA$ and $\MA_{\aleph_1}$ in addition to $\ZFC$. Every model of $\ZFC$ has a forcing extension which has the same $\omega_1$ and satisfies $\OCA$ and $\MA_{\aleph_1}$, see the sketch contained in \cite[\S2]{Velickovic.OCA1} resembling a technique already used in \cite{ARS}. In most of the proofs we use $\OCA_\infty$ instead of $\OCA$. 

Probably the best known consequence of the assumption of $\OCA$ and $\MA_{\aleph_1}$ is due to Veli\v{c}kovi\'c. Recall that a bijection $g\colon\NN\setminus F_1\to\NN\setminus F_2$, where $F_1$ and $F_2$ are finite subsets of $\NN$, is called an \emph{almost permutation} of $\NN$. If $g$ is an almost permutation, then it induces an automorphism $\iso_g$ of the Boolean algebra $\SP(\NN)/\Fin$ defined by 
\[
\iso_g([X])=[g[X]],
\]
where, for $X\subseteq\NN$, $[X]$ denotes its class in $\SP(\NN)/\Fin$. Automorphisms of this sort are called trivial, and Veli\v{c}kovi\'c showed in the groundbreaking \cite{Velickovic.OCAA} that all automorphisms of $\SP(\NN)/\Fin$ are trivial if $\OCA$ and $\MA_{\aleph_1}$ are assumed. The blueprints of  Veli\v{c}kovi\'c's argument are still actual, and we will develop a sophisticated version of them to prove our main technical result, Theorem~\ref{thm:lifting}.
\subsection{$\Cstar$-algebras}\label{subsec:cstar}

For the basics of $\Cstar$-algebras see~\cite{Dixmier} or \cite{Blackadar.OA}. If $A$ is a $\Cstar$-algebra, we write $A_{\leq 1}$ and $A_{1}$ for the closed unit ball and the closed unit sphere of $A$, and $\U(A)$ for the set of unitaries in $A$. If $J\subseteq A$ is a subalgebra\footnote{Subalgebras are always $\Cstar$-subalgebras}, $J$ is an \emph{ideal} of $A$ if $ax$ and $xa$ are in $J$ for all $a\in A$. An ideal $J$ in $A$ is \emph{essential} if the only $a\in A$ satisfying $ax = xa = 0$ for all $x\in J$ is $a = 0$.

\subsubsection*{Multipliers, coronas and lifts}\label{subsection:multipliersandcoronas}

The main $\Cstar$-algebraic objects we study are multiplier algebras and their associated corona algebras.

\begin{definition}
Let $A$ be a $\Cstar$-algebra. The \emph{multiplier algebra} $\SM(A)$ is the unique, up to isomorphism, unital $\Cstar$-algebra which contains $A$ as an essential ideal and which has the property that whenever $A$ is an essential ideal of a $\Cstar$-algebra $B$, there is a unique embedding $B\to \SM(A)$ which is the identity on $A$. The quotient $\SQ(A) = \SM(A)/A$ is the \emph{corona algebra} of $A$, and we write $\pi_A$ for the quotient map $\SM(A)\to \SQ(A)$.
\end{definition}

The multiplier algebra $\SM(A)$ is the largest unital $\Cstar$-algebra containing $A$ in a `dense' way. It plays the same role in $\Cstar$-algebras theory as the \v{C}ech-Stone compactification does in topology. In fact, if $X$ is a locally compact space and $A=C_0(X)$ then $\SM(A)=C(\beta X)$ and $\SQ(A)=C(\beta X\setminus X)$, $\beta X$ and $\beta X\setminus X$ being the \v{C}ech-Stone compactification and reminder of $X$.

In general, the construction of $\SM(A)$ from $A$ is nontrivial. We refer the reader to \cite[II.7.3]{Blackadar.OA} for a discussion. For our purposes we find very useful the following alternative characterization of the multiplier algebra, which the reader may take as a concrete description of $\SM(A)$. Recall that any $\Cstar$-algebra $A$ may be realized as a $\Cstar$-subalgebra of $\mathcal B(H)$, for a Hilbert space $H$, by the Gelfand-Naimark-Segal construction. In this setting, we define the \emph{strict topology} on $\mathcal B(H)$ to be the the topology generated by the seminorms $x\mapsto \norm{ax}$ and $x\mapsto \norm{xa}$, for $a\in A$. $\SM(A)$ is the closure of $A$ with respect to the strict topology.

If $A$ is nonunital, then $\SM(A)$ is nonseparable in its norm topology. However, if $A$ is separable, then $\SM(A)$ is separable in the strict topology, and every bounded norm-closed subset of $\SM(A)$ is Polish in the strict topology.
We point out a few examples of particular multiplier algebras and corona algebras.

\begin{itemize}
\item If $A$ is unital, $\SM(A)=A$;
\item If each $A_n$ is unital, then $\SM(\bigoplus A_n)=\prod A_n$. The corona $\prod A_n/\bigoplus A_n$ is the \emph{reduced product} of the $A_n$'s.
\item If each $A_n$ is unital and $\CI\subseteq\SP(\NN)$ is an ideal, the algebra $\bigoplus_{\CI}A_n$ is defined as follows:
\[
\bigoplus_{\CI}A_n=\{(a_n)\in\prod A_n\mid\forall\varepsilon>0 (\{n\mid \norm{a_n}>\varepsilon\}\in\CI).\}
\]
If $\CI$ contains $\Fin$, the multiplier algebra of $\bigoplus_{\CI}A_n$ is $\prod A_n$. The corresponding corona algebra is known as the $\CI$-reduced product of the $A_n$'s or, if $\CI$ is maximal, the ultraproduct.
\end{itemize}
$\CI$-reduced products were studied in \cite{Ghasemi.FFV} and \cite{Farah-Shelah.RCQ}, see also \cite[\S16.2]{Farah.Book}.

The following lemma provides a stratification of $\SQ(A)$ into subspaces analogous to reduced products. Various forms of this stratification have been used in the literature already: see, for instance, \cite[Theorem~3.1]{Elliott.Der2}, \cite[Lemma 1.2]{Farah.C} and \cite{Arveson.Notes}.

\begin{lemma}
 \label{lem:stratification}
 Let $A$ be a $\Cstar$-algebra with an increasing countable approximate identity of positive contractions $\{e_n\}$ with, for all $n$, $e_n e_{n+1} = e_n$. Given an interval $I\subseteq\NN$ we write $e_I = e_{\max(I)} - e_{\min(I)}$. Let $t\in\SM(A)$. Then there are finite intervals $I_n^i\subseteq \NN$, for each $n\in\NN$ and $i = 0,1$, and $t_0$ and $t_1$ in $\SM(A)$, such that for each $i\in\{0,1\}$,
 \begin{enumerate}
 \item the intervals $I_n^i$, for $n\in\NN$, are pairwise disjoint and consecutive,
 \item $t_i$ commutes with $e_{I_n^i}$ for each $n\in\NN$, and
 \item $t - (t_0 + t_1)\in A$.
 \end{enumerate}
\end{lemma}

\begin{proof}
 For each $k\in\NN$, $e_k t$ and $t e_k$ are both in $A$, and hence we may find for each $\varepsilon > 0$ a $k' > k$ such that $\norm{e_k t (1 - e_{k'})}$ and $\norm{(1 - e_{k'}) t e_k}$ are both less than $\varepsilon$. Applying this recursively we may construct a sequence $0 = k_0 <k_0+1< k_1 < \cdots$ such that
 \[
 \norm{(1 - e_{k_{n+1}}) t e_{k_n}} + \norm{e_{k_n} t (1 - e_{k_{n+1}})} \le 2^{-n}.
 \]
 Define $J_n = [k_n,k_{n+1})$, and let
 \[
 t_0 = \sum_{n=0}^\infty e_{J_{2n}} t e_{J_{2n}} + e_{J_{2n+1}} t e_{J_{2n}} + e_{J_{2n}} t e_{J_{2n+1}}
 \]
 and
 \[
 t_1 = \sum_{n=0}^\infty e_{J_{2n+1}} t e_{J_{2n+1}} + e_{J_{2n+2}} t e_{J_{2n+1}} + e_{J_{2n+1}} t e_{J_{2n+2}}
 \]
 Note that these sums converge in the strict topology since $e_k e_J = 0$ for any interval $J$ with $k< \min(J)$, and each $J_n$ has size $\geq2$. Moreover, since
 \begin{eqnarray*}
 \norm{(1-e_{k_m}) (t - t_0 - t_1)} & \le& \sum_{i=m}^\infty \norm{(1 - e_{k_i}) (t-t_0-t_1) e_{k_{i-1}}} \\&+&\sum_{i=m+1}^\infty \norm{e_{k_i} (t-t_0-t_1) (1 - e_{k_{i+2}})} \\
  & \le& 2^{-m+2}+2^{-m+1},
 \end{eqnarray*}
then $t - (t_0 + t_1) \in A$. Finally, we have
 \[
 e_{J_{2n}\cup J_{2n+1}} t_0 = e_{J_{2n}} t e_{J_{2n}} + e_{J_{2n+1}} t e_{J_{2n}} + e_{J_{2n}} t e_{J_{2n+1}} = t_0 e_{J_{2n}\cup J_{2n+1}}
 \]
 and
 \[
 e_{J_{2n+1}\cup J_{2n+2}} t_1 = e_{J_{2n+1}} t e_{J_{2n+1}} + e_{J_{2n+2}} t e_{J_{2n+1}} + e_{J_{2n+1}} t e_{J_{2n+2}} = t_1 e_{J_{2n+1}\cup J_{2n+2}}.
 \]
Setting $I_n^i = J_{2n+i}\cup J_{2n+i+1}$, we have the required intervals.
\end{proof}

The main concern of this paper is the study of isomorphisms $\varphi\colon \SQ(A)\to\SQ(B)$, where $A$ and $B$ are nonunital separable $\Cstar$-algebras. Given such $\varphi$, a map $\Phi\colon \SM(A)\to\SM(B)$ is a \emph{lift} of $\varphi$ if the following diagram commutes:
\begin{center}
\begin{tikzpicture}
 \matrix[row sep=1cm,column sep=2cm] {
&\node (A1) {$\SM(A)$};
& \node (A2) {$\SM(B)$};
\\
&\node (B1) {$\SQ(A)$};
& \node (B2) {$\SQ(B)$};
\\
};
\draw (A1) edge[->] node [above] {$\Phi$} (A2) ;
\draw (A1) edge[->] node [left] {$\pi_A$} (B1) ;
\draw (A2) edge[->] node [right] {$\pi_B$.} (B2) ;
\draw (B1) edge[->] node [above] {$\varphi$} (B2) ;
\end{tikzpicture}
\end{center}
The existence of a lift is always ensured by the Axiom of Choice; however, such a lift is not guaranteed to respect the algebraic or topological structure of the multiplier algebras involved. If $X\subseteq\SM(A)$ and $\Phi$ has the property that $\pi_B(\Phi(x))=\varphi(\pi_A(x))$ for all $x\in X$ we say that \emph{$\Phi$ is a lift of $\varphi$ on $X$}. If $A=\bigoplus A_n$ for some unital $\Cstar$-algebras $A_n$, and $x\in\prod A_n$, define
\[
\supp(x)=\{n\mid x_n\neq 0\}.
\]
If $\CI\subseteq\SP(\NN)$, abusing the notation, we say that \emph{$\Phi$ is a lift of $\varphi$ on $\CI$} if $\Phi$ lifts $\varphi$ on $\{x\in\prod A_n\mid \supp(x)\in\CI\}$.

\subsubsection{Classes of $\Cstar$-algebras}\label{ssubsec:classes}
There are several interesting classes of $\Cstar$-algebras we mention in this paper. Some of the most important ones are obtained by considering objects that can be approximated in a certain way by finite-dimensional building blocks, for example the classes of UHF (limits of full matrix algebras), AF (limits of finite-dimensional algebras), and nuclear algebras\footnote{That the class of nuclear algebras coincides with the one satisfying the Completely Positive Approximation Property, and in turn with the class of $\Cstar$-algebras which are amenable, is a combination of deep theorems, see \cite[IV.3]{Blackadar.OA}}. Another important class of $\Cstar$-algebras, the one of \emph{exact} algebras, can be defined by an external approximation property. (For more on these classes, see \cite{Blackadar.OA}).

Among these classes the one  whose technical definition we use in the following is the one of algebras satisfying the Metric Approximation Property. Recall that an \emph{operator system} is a unital $^*$-closed vector subspace of $\mathcal B(H)$.
\begin{definition}\label{defin:MAP}
A $\Cstar$-algebra $A$ has the \emph{metric approximation property} (MAP) if the identity map can be approximated, uniformly on finite sets, by contractive linear maps of finite rank. Formally, $A$ has the MAP if and only if for all finite $F\subseteq A$ and $\varepsilon>0$ there is a finite-dimensional operator system $E$ and unital linear contractions $\varphi\colon A\to E$ and $\psi\colon E\to A$ with $\norm{\psi\circ\varphi(a)-a}<\varepsilon$ for all $a\in F$.
\end{definition}
Examples of $\Cstar$-algebras with the MAP are nuclear and exact $\Cstar$-algebras. Szankowski \cite{Szank:AP} has proven that $\mathcal B(H)$ does not even have the weaker \emph{approximation property}\footnote{This is obtained by removing the requirement that the maps involved are contractions.}, therefore, by Blackadar's closing off argument/L\"owenheim-Skolem Theorem (\cite[\S7]{Farah.Book}) there is a separable $\Cstar$-algebra without the approximation property.

A $\Cstar$-algebra $A$ is purely infinite and simple if for all $a$ and $b$ in $A$ there are $x$ and $y$ in $A$ with $a=xby$. A unital purely infinite, separable, simple, and nuclear algebra is a Kirchberg algebra. Among Kirchberg algebras the most important ones (and the only ones known so far)  are those satisfying the Universal Coefficient Theorem (UCT, see \cite[V.1.5]{Blackadar.OA}). This class is extremely well behaved: if $A$ and $B$ are Kirchberg algebras satisfying the UCT, $A$ and $B$ are isomorphic if and only if they have the same $K$-theory (e.g., \cite{KirchPhil,Phil:Class}). (For the basics of $K$-theory, see \cite[V]{Blackadar.OA}). For this reason, these $\Cstar$-algebras are often refered to as `classifiable'. It is a (very deep) open problem whether all separable nuclear algebras satisfy the UCT. For (a lot) more on these matters and on the Elliott classification programme (aiming to classify large classes of $\Cstar$-algebras by algebraic and topological invariants), see for example \cite{WinterAbel}, or any text focusing on classification of $\Cstar$-algebras (we recommend \cite{Giordano2018} for a friendly introduction to the subject).

\section{Approximate maps}\label{sec:apmaps}

If $\varphi\colon A\to B$ is any (non necessarily linear) function between normed vector spaces, we write $\norm{\varphi}$ for the quantity $\sup_{\norm{a}\leq1}\norm{\varphi(a)}$.  We say that $\varphi$ is a contraction if $\norm{\varphi}\leq 1$.

\begin{definition}\label{defin:apmaps}
Let $A$ and $B$ be $\Cstar$-algebras, and $\varepsilon\geq 0$. A map $\varphi\colon A\to B$ is
\begin{enumerate}
\item\label{apmap-lin} \emph{$\varepsilon$-linear} if $\sup_{x,y\in A_{\leq 1},|\lambda|,|\mu|\leq 1}\norm{\varphi(\lambda x+\mu y)-\lambda\varphi(x)-\mu\varphi(y)}<\varepsilon$;
\item\label{apmap-star} \emph{$\varepsilon$-$^*$-preserving} if $\sup_{x\in A_{\leq 1}}\norm{\varphi(x^*)-\varphi(x)^*}<\varepsilon$;
\item\label{apmap-mult} \emph{$\varepsilon$-multiplicative} if $\sup_{x,y\in A_{\leq 1}}\norm{\varphi(xy)-\varphi(x)\varphi(y)}<\varepsilon$;
\item\label{apmap-nonzero} \emph{$\varepsilon$-nonzero} if there is $a\in A_1$ with $\norm{\varphi(a)}\geq 1-\varepsilon$;
\item\label{apmap-inj} \emph{$\varepsilon$-isometric} if $\sup_{x\in A_{\leq1}}|\norm{x}-\norm{\varphi(x)}|<\varepsilon$;
\item\label{apmap-surj} \emph{$\varepsilon$-surjective} if for all $b\in B_{\leq 1}$ there is $x\in A_{\leq 1}$ with $\norm{b-\varphi(x)}<\varepsilon$.
\end{enumerate}
\end{definition}

A contraction satisfying  \eqref{apmap-lin}--\eqref{apmap-mult} is an \emph{$\varepsilon$-$^*$-homomorphism}. An $\varepsilon$-$^*$-homomorphism satisfying~\eqref{apmap-inj} is an \emph{$\varepsilon$-embedding}, and an $\varepsilon$-embedding satisfying~\eqref{apmap-surj} is an \emph{$\varepsilon$-isomorphism}.

\begin{definition}
We say that two $\Cstar$-algebras $A$ and $B$ are \emph{$\varepsilon$-isomorphic} if there is an $\varepsilon$-isomorphism from $A$ to $B$.
\end{definition}


The notions of $\varepsilon$-embedding and $\varepsilon$-isomorphism are related to the Hausdorff distance between the unit balls of $\Cstar$-subalgebras of $\mathcal B(H)$. This is known as Kadison-Kastler distance: if $A$ and $B$ are $\Cstar$-subalgebras of the same $\mathcal B(H)$, we define
\[
d_{KK}(A,B)=\max\{\sup_{a\in A_1}\inf_{b\in B}\norm{a-b},\sup_{b\in B_1}\inf_{a\in A}\norm{a-b}\}.
\]
We say that $A\subseteq_\varepsilon B$ if for all $a\in A_1$ there is $b\in B$ with $\norm{a-b}<\varepsilon$. If $d_{KK}(A,B)<\varepsilon$, the Axiom of Choice gives a $2\varepsilon$-isomorphism from $A$ to $B$, and if $A\subseteq_\varepsilon B$ one can find a $2\varepsilon$-embedding from $A$ to $B$. On the other hand, it is not clear whether $\Cstar$-algebras which are $\varepsilon$-isomorphic must have isomorphic images in $\mathcal B(H)$ with small Kadison-Kastler distance (see Question~\ref{ques2}).

The first to study proximity phaenomena in the setting of operator algebras were Kadison and Kastler in the seminal \cite{Kadison-Kastler}. The first of the the following two conjectures is known as the Kadison and Kastler conjecture, and the second one as its strong version\footnote{This is not the way these were stated in the original Kadison and Kastler's paper, see \cite[p.38]{Kadison-Kastler}}.
\begin{conjecture}
\begin{enumerate}
\item There is $\e>0$ such that whenever $A$ and $B$ are unital separable $\Cstar$-algebras with $d_{KK}(A,B)<\e$ then $A$ and $B$ are isomorphic;
\item for every $\e>0$ there is $\delta>0$ such that if $A$ and $B$ are separable subalgebras of $\mathcal B(H)$ with $d_{KK}(A,B)<\delta$ then there is a unitary $u\in\mathcal B(H)$ such that $\norm{u-1}<\e$ and $uAu^*=B$.
\end{enumerate}
\end{conjecture} 

These conjectures were verified in various situations. Notably, if $A$ and $B$ are separable nuclear $\Cstar$-algebras which are Kadison-Kastler close, they are isomorphic (\cite{CSSWW}). 
In the same spirit, we ask the following questions about $\varepsilon$-isomorphisms between $\Cstar$-algebras in a given class $\mathcal{C}$.
\begin{enumerate}[label=($Q$\arabic*)]
\item\label{stab1}  Does there exist, for every $\varepsilon>0$, a $\delta>0$ such that if $A$ and $B$ are elements of $\mathcal C$ and  $\varphi\colon A\to B$ is a $\delta$-isomorphism, then there is an isomorphism $\psi \colon A\to B$ with $\norm{\varphi - \psi} < \varepsilon$?

\item\label{stab2} Is there an $\varepsilon_0 > 0$ such that for all $A\in\mathcal C$ and $B\in\mathcal C$, if $A$ and $B$ are $\varepsilon_0$-isomorphic then $A$ and $B$ must be isomorphic?

\item\label{stab3} Is there an $\varepsilon > 0$ such that if $A\in \mathcal C$, $B$ is any $\Cstar$-algebra, and $A$ and $B$ are $\varepsilon$-isomorphic then one can conclude that $B\in\mathcal C$?
\end{enumerate}
The phaenomenon described in \ref{stab1} has its origin from the work of Ulam on approximate group homomorphism. Using his terminology,  (see e.g., \cite{KanoveiReeken.Ulam}) we say that these approximate maps are stable.
\begin{definition}\label{defin:ulam}
Let $\mathcal C$ be a class of $\Cstar$-algebras. $\mathcal C$ is said \emph{Ulam stable} if \ref{stab1} has a positive solution for $\mathcal C$.
\end{definition}
Farah shows in \cite[Theorem~5.1]{Farah.C} that the class of finite-dimensional $\Cstar$-algebras is Ulam stable, and \v{S}emrl proved the same for the class of abelian $\Cstar$-algebras in \cite{Semrl.USAbel}. 
Showing that larger classes of $\Cstar$-algebras are Ulam stable presents substantial difficulties even for natural generalizations of the classes treated above, such as the class of UHF algebras.\footnote{We believe one should attempt to prove that the class of unital separable subhomogeneous algebras is Ulam stable, but we do not dare to conjecture it.} 

If $\mathcal P$ is a property of $\Cstar$-algebras, $\mathcal C_{\mathcal P}$ denotes the class of all $\Cstar$-algebras having $\mathcal P$.
\begin{definition}
Let $\mathcal C$ be a class of $\Cstar$-algebras. We say that
\begin{itemize}
\item $\mathcal C$ is \emph{classifiable by approximate isomorphisms} if for all $A\in\mathcal C$ there is $\e>0$ such that if $B\in\mathcal C$ and $A$ and $B$ are $\e$-isomorphic, then $A\cong B$;
\item $\mathcal C$ is \emph{uniformly classifiable by approximate isomorphisms} if there is a positive solution for \ref{stab2} relatively to $\mathcal C$;
\item $\mathcal C$ is \emph{stable under approximate isomorphisms} if there is a positive solution for \ref{stab3} relatively to $\mathcal C$.
\end{itemize}
If $\mathcal P$ is a property of $\Cstar$-algebras, we say that $\mathcal P$ is \emph{stable under approximate isomorphisms} if $\mathcal C_{\mathcal P}$ is.
\end{definition}

The following summarizes some of the known results regarding positive answers to \ref{stab1} and \ref{stab2}. Recall that a $\Cstar$-algebra $A$ is AF, approximately finite-dimensional, if for all $\e>0$ and finite $F\subseteq A$ there is a finite-dimensional subalgebra of $A$ having distance $<\e$ from $F$, see \cite[II.8.2]{Blackadar.OA}.

\begin{theorem}\label{thm:apmap-known}
Suppose that $\varepsilon<\frac{1}{4}$. Then,
\begin{itemize}

\item there is $K>0$ such that if $A$ is a finite-dimensional $\Cstar$-algebra and $B$ is a $\Cstar$-algebra then for every $\varepsilon$-$^*$-homomorphism $\varphi\colon A\to B$ there is a $^*$-homomorphism $\varphi\colon A\to B$ with
\[
 \norm{\varphi-\psi}<K\sqrt{\varepsilon}.
\]
Moreover, if $\varepsilon$ is sufficiently small and $\varphi$ is an $\varepsilon$-isomorphism, then $\psi$ is an isomorphism;

\item the class of abelian algebras is Ulam stable. Moreover, there is $K>0$ such that if $\e\in(0,1)$, then one can choose $\delta=K\e^2$ in Definition~\ref{defin:ulam};

\item\label{apmap3} the class of separable AF algebras is stable under approximate isomorphisms and uniformly classifiable by approximate isomorphisms.
\end{itemize}
\end{theorem}
\begin{proof}
The first statement is \cite[Theorem~1.7]{MKAV.UC} while the second is the main result of \cite{Semrl.USAbel}, so only the third statement requires a proof.

We then need to find $\varepsilon>0$ such that if $A$ is a separable AF algebra and $B$ is a $\Cstar$-algebra then $A\cong B$ whenever $A$ and $B$ are $\e$-isomorphic. By~\cite[Theorem 2.4]{MKAV.UC} there is $\e\in (0,10^{-10})$ such that if $\varphi\colon A\to\mathcal B(H)$ is a $\e$-embedding, then there is an embedding $\psi\colon A\to\mathcal B(H)$ with $\norm{\varphi-\psi}<10^{-10}$, and the choice of $\e$ is uniform over all separable AF algebras.  Pick a $\Cstar$-algebra $B$ such that $A$ and $B$ are $\e$-isomorphic, and let $\varphi'\colon A\to B$ be an $\e$-isomorphism. Without loss of generality, we can assume that $B\subseteq\mathcal B(H)$. By composing $\varphi'$ with such inclusion, we get an $\e$-embedding $\varphi\colon A\to \mathcal B(H)$. Applying \cite[Theorem 2.4]{MKAV.UC}, we can find $\psi\colon A\to\mathcal B(H)$ with $\norm{\varphi-\psi}<10^{-9}$. Notice that $d_{KK}(\psi(A),B)<10^{-10}+10^{-10}<10^{-9}$ hence, by \cite[Theorems 6.1]{Christensen.NI}, $B$ is AF, and by \cite[Theorem 6.2]{Christensen.NI}, $A\cong B$.
\end{proof}
The proof of the third part of Theorem~\ref{thm:apmap-known} can be reproduced verbatim to show the following connection between Kadison-Kastler stability and stability under approximate isomorphisms.
\begin{proposition}\label{prop:genulam}
Let $\mathcal C$ be a class of $\Cstar$-algebras such that there are positive $\delta$ and $\e$ with the following properties:
\begin{enumerate}
\item\label{part1} if $A\in\mathcal C$ and $B$ is a $\Cstar$-algebra with $d_{KK}(A,B)<\e$, then $A\cong B$
\item if  $\varphi\colon A\to\mathcal B(H)$ is a $\delta$-embedding then there is an embedding $\psi\colon A\to\mathcal B(H)$ with $\norm{\varphi-\psi}<\e$. 
\end{enumerate}
Then $\mathcal C$ is  stable under approximate isomorphisms and uniformly classifiable by approximate isomorphisms.\qed
\end{proposition}
\begin{remark}\label{rem:almostlinear}
The largest class known to satisfy condition~\ref{part1} of Proposition~\ref{prop:genulam} is the one of separable nuclear $\Cstar$-algebras, by \cite[Theorem A]{CSSWW}. More than that: a theorem of Johnson (\cite[Theorem 7.1]{Johnson.AMNM}), asserts that there is a constant $K$ such that if $A$ is a separable nuclear $\Cstar$-algebra then for each linear $\varepsilon$-$^*$-homomorphism $A\to\mathcal B(H)$, one can find a $^*$-homomorphism which is $K\varepsilon^2$ close to the original map. Therefore showing that the class of nuclear separable $\Cstar$-algebras is stable under approximate isomorphisms and uniformly classifiable by approximate isomorphisms, `only' amounts to show that one can perturb, uniformly over nuclear separable algebras,  $\e$-$^*$-homomorphisms to linear maps.
\end{remark}
As $K$-theory serves as (part of the) classifying invariant for  large classes of $\Cstar$-algebras, to obtain further answers to \ref{stab2}, we would like a way of comparing the $K$-theory of two $\Cstar$-algebras which are $\varepsilon$-isomorphic. However, the absence of linearity means that the amplification\footnote{If $\varphi\colon A\to B$ is a map, its amplifications are the maps $\varphi^{(n)}\colon M_n(A)\to M_n(B)$ defined as $\varphi^{(n)}((a_{i,j})_{i,j})=(\varphi(a_{i,j})_{i,j})$} of an $\varepsilon$-isomorphism may not be an $\varepsilon$-isomorphism. For this reason, it is difficult to obtain information on the $K$-theory simply from the existence of an approximate isomorphism. (This is different in the case of Kadison-Kastler perturbation, where informations on certain invariants can be obtained from Kadison-Kastler proximity, see \cite{CSSW}). This obstacle disappears if one can compute the $K$-theory of $A$ without having to pass through  $A\otimes\mathcal K(H)$. 

The version of the following (and of Corollary~\ref{cor:KirchUCT}) in case $A$ and $B$ have small Kadison-Kastler distance was proved in \cite{CSSW}. The adaption of the proof to this setting is an easy exercise for an expert in operator algebras, but we reproduce it for completeness.
\begin{lemma}\label{lemma:Kth}
There is $\varepsilon>0$ such that if $A$ and $B$ are unital purely infinite, simple, and $\varepsilon$-isomorphic then
\[
 K_0(A)\cong K_0(B), \,\,\, K_1(A)\cong K_1(B).
\]
Moreover such an isomorphism sends the class of the identity in $K_0(A)$ to the class of the identity in $K_0(B)$.
\end{lemma}
\begin{proof}
As said, the key point is that, if $A$ is unital purely infinite and simple, one can compute the $K$-theory of $A$ without passing to $A\otimes\mathcal K(H)$. In particular Cuntz showed (see \cite[p.188]{cuntz1981k}) that in this case $K_0(A)$ is isomorphic to the set of projections in $A$ modulo Murray-von Neumann equivalence (\cite[II.3.3.3]{Blackadar.OA}), and that $K_1(A)$ is isomorphic to $\mathcal U(A)/\mathcal U_0(A)$, where $\mathcal U_0(A)$ is the connected component of the identity.

Fix $\varepsilon<1/16$ and let $\varphi$ be an $\varepsilon$-isomorphism between $A$ and $B$. A well known argument shows that for every projection $p\in A$ there is a projection $q\in B$ with $\norm{q - \varphi(p)} < 1/2$; define $\tilde\varphi_0(p) = q$. We want to show that the map $\tilde\varphi_0$ induces an isomorphism between $K_0(A)$ and $K_0(B)$ which is mapping the unit to the unit. First, we show that $\tilde\varphi_0$ is well defined. Suppose that $p,q\in A$ are Murray-von Neumann equivalent and choose $v\in A$ with $vv^*=p$ and $v^*v=q$. By weak stability of the set of partial isometries (see \cite[\S 4.1]{Loring.LSPP} or \cite[Example 3.2.7]{bourbaki}) we can find a partial isometry $w$ with $\norm{w-\varphi(v)}<\frac{1}{8}$. Then $\norm{\tilde\varphi_0(p)-ww^*}$ and $\norm{\tilde\varphi_0(q)-w^*w}<1/2$, so $\norm{\tilde\varphi_0(p) - \tilde\varphi_0(q)}< 1$, and it follows (see \cite[II.3.3.5]{Blackadar.OA}) that $\tilde\varphi_0(p)$ and $\tilde\varphi_0(q)$ are Murray-von Neumann equivalent. This shows that $\tilde\varphi_0$ induces a map from $K_0(A)$ to $K_0(B)$; to check that $\tilde\varphi_0$ is an isomorphism is routine, and we leave it to the reader.

Similarly, using that two unitaries which are close to each other are in the same connected component, that, in the purely infinite simple unital case, we have that $\mathcal U_0(A)=\{\exp(ia)\mid a=a^*, \norm{a}\leq 2\pi\}$ (see \cite{Phillips.Exp}) and that all almost unitaries are close to unitaries, from $\varphi$ we can define a map $\tilde\varphi_1\colon\mathcal U(A)\to\mathcal U(B)$ which induces an isomorphism between $K_1(A)$ and $K_1(B)$.
\end{proof}

\begin{corollary}\label{cor:KirchUCT}
The class of Kirchberg algebras satisfying the UCT is uniformly classifiable by approximate isomorphisms.
\end{corollary}
\begin{proof}
If $A$ and $B$ are Kirchberg UCT algebras, then $A$ and $B$ are isomorphic if and onyl if they have the same $K$-theory (e.g., \cite{KirchPhil, Phil:Class}). The result then follows from Lemma~\ref{lemma:Kth}
\end{proof}

Before turning on how to relate these concepts to isomorphisms of reduced products, we list few properties which are stable under approximate isomorphisms. Examples of such are ``purely infinite and simple", ``real rank zero", ``tracial" algebras, and many more. In general, let $\mathcal C$ be a class $\Cstar$-algebras, and suppose that both $\mathcal C$ and its complement are closed under isomorphisms, ultraproducts, and ultraroots. (If $\mathcal U$ is an ultrafilter, and $A$ a $\Cstar$-algebra, we say that $A$ is the ultraroot of $A^{\mathcal U}$.). Such classes are known as both axiomatizable and co-axiomatizable in the language of continuous model-theory, see \cite[\S3.13]{bourbaki}. In this case, the property `belonging to $\mathcal C$' is stable under approximate isomorphisms, as the class $\mathcal C$ is isolated by a first order formula (see~\cite[\S3.14]{bourbaki}). We don't know whether certain important properties of $\Cstar$-algebras which are known to be preserved by Kadison-Kastler proximity, such as nuclearity or simplicity, are stable under approximate isomorphisms, see Question~\ref{ques1}.

\subsection*{$\varepsilon$-isomorphisms and isomorphisms of reduced products}
Here we detail the relationship between approximate maps and $^*$-homomorphisms between reduced products. These relation was first studied systematically for `discrete' quotient structure of the form $\SP(\NN)/\mathcal I$ in \cite{Farah.Lifts}.

\begin{proposition}\label{prop:welldefinedmaps}
 Suppose $A_n$ and $B_n$ are sequences of $\Cstar$-algebras and let $A=\prod A_n/\bigoplus A_n$ and $B=\prod B_n/\bigoplus B_n$. Then every sequence $\varphi_n\colon A_n\to B_n$ of $\varepsilon_n$-$^*$-homomorphisms, where $\varepsilon_n\to 0$, induces a $^*$-homomorphism $\Phi\colon A\to B$. Moreover, if each $\varphi_n$ is an $\varepsilon_n$-embedding, then $\Phi$ is an embedding; and if each $\varphi_n$ is an $\varepsilon_n$-isomorphism, $\Phi$ is an isomorphism.
\end{proposition}
\begin{proof}
Let $\pi_A$ and $\pi_B$ be the canonical quotient maps from the products to the reduced products. Define, for a contraction $(a_n)_n\in\prod A_n$, 
\[
\Phi(\pi_A((a_n)_n))=\pi_B((\varphi_n(a_n))_n),
\]
and extend $\Phi$ to $A$, by setting $\Phi(a)=\norm{a}\Phi(a/\norm{a})$, in case $\norm{a}>1$. Notice that $\norm{\Phi(\pi_A((a_n)_n))}=\limsup_n\norm{\varphi_n(a_n)}$ whenever $(a_n)_n\in\prod A_n$. In particular, $\Phi$ is well defined, since $\norm{\varphi_n(a_n)}\leq\norm{a_n}+\varepsilon_n$ for all $n$, and $\varepsilon_n\to 0$ as $n\to 0$.

We now sketch the proof that $\Phi$ is linear, and that if each $\varphi_n$ is an $\varepsilon_n$-embedding then $\Phi$ is injective. The other conditions are left to the reader.

If $(a_n)_n$ and $(b_n)_n$ are contractions in $\prod A_n$ and $\lambda$ and $\mu$ are in $\ce$, we have that 
\begin{eqnarray*}
&&\norm{\Phi(\pi_A((\lambda a_n+\mu b_n)_n))-\Phi(\pi_A((\lambda a_n)_n))-\Phi(\pi_A((\mu b_n)_n))}=\\&&\norm{\pi_B((\varphi_n(\lambda a_n+\mu b_n))_n)-\pi_B((\varphi_n(\lambda a_n))_n)-\pi_B((\varphi_n(\mu b_n))_n)}=\\&&\limsup_n\norm{\varphi_n(\lambda a_n+\mu b_n)-\lambda \varphi_n(a_n)-\mu \varphi_n(b_n)}=0
\end{eqnarray*}
This shows that $\Phi$ is linear.

Suppose now that each $\varphi_n$ is an $\varepsilon_n$-embedding, and pick $a\in A$ of norm $1$. We want to show that $\norm{\Phi(a)}=1$. Let $(a_n)_n\in\prod A_n$ with $a=\pi_A((a_n)_n)$ and $\limsup\norm{a_n}= 1$. Since each $\varphi_n$ is an $\varepsilon_n$-embedding, we have that for all $n\in\NN$, $\norm{\varphi_n(a_n)}\geq \norm{a_n}-\varepsilon_n$. As $\varepsilon_n\to 0$, we then have that $\limsup_n\norm{\varphi_n(a_n)}=1$, meaning $\norm{\Phi(a)}=1$.
\end{proof}
The following can be viewed as the converse of Proposition~\ref{prop:welldefinedmaps}. We state it as a lemma, as it will be used later. (Recall that an ideal $\CI\subseteq\SP(\NN)$ is dense if whenever $Y\subseteq\NN$ is infinite then there is an infinite $X\in\CI$ with $X\subseteq Y$, see \S\ref{subsec:DSTandIdeals}).

\begin{lemma}\label{lemma:seqofappmaps}
Let $A_n$, $B_n$ be $\Cstar$-algebras, $g$ an almost permutation of $\NN$ and $\varphi_n\colon A_n\to B_{g(n)}$ be maps where $\prod\varphi_n$ is the lift of an isomorphism
\[
\iso\colon\prod A_n/\bigoplus A_n\to\prod B_n/\bigoplus B_n
\]
on a dense ideal $\CI$. Then there is a sequence $\varepsilon_n$ tending to $0$ such that each $\varphi_n$ is an $\varepsilon_n$-isomorphism.
\end{lemma}

\begin{proof}
Let $\pi_A$ be the canonical quotient map $\prod A_n\to\prod A_n/\bigoplus A_n$. It is enough to show that for every $\varepsilon>0$ there is $n$ such that each $\varphi_m$ is an $\varepsilon$-isomorphism whenever $m\geq n$. We just show $\varepsilon$-additivity, and leave the rest to the reader.

Suppose then that there is an infinite set $I\subseteq\NN$ and, for each $n\in I$, contractions $x_n$ and $y_n$ in $A_n$ with $\norm{x_n},\norm{y_n} \le 1$ and
\[
 \norm{\varphi_n(x_n+y_n)-\varphi_n(x_n)-\varphi_n(y_n)}>\varepsilon.
\]
Choose an infinite $X\in\CI$ with $X\subseteq I$. Such $X$ exists by density of $\CI$. Without loss of generality we may assume that $g$ is defined on $X$. Let $x=\sum_{n\in X} x_n$ and $y=\sum_{n\in X}x_n$. Since $\prod \varphi_n$ is a lift of $\iso$ on $\CI$, and  $\varphi_n\varphi_m=0$ whenever $n\neq m$, we have
\begin{eqnarray*}
 \norm{\iso(\pi_A(x+y))-\iso(\pi_A(x))-\iso(\pi_A(y))}=\\
 =\limsup_{n\in X}\norm{\varphi_n(x_n+y_n)-\varphi_n(x_n)-\varphi_n(y_n)}>\varepsilon.
\end{eqnarray*}
This is a contradiction.
\end{proof}

\begin{definition}\label{defin:trivialredprod}
Let $A_n$ and $B_n$ be $\Cstar$-algebras with no nontrivial central projections.
 An isomorphism $\iso\colon \prod A_n/\bigoplus A_n\to\prod B_n/\bigoplus B_n$ is \emph{asymptotically algebraic} if there are finite $F_1,F_2\subseteq\mathbb N$, a bijection $g\colon \mathbb N\setminus F_1\to\mathbb N\setminus F_2$ and maps $\varphi_n\colon A_n\to B_{g(n)}$ such that the map $\Phi\colon \prod A_n\to \prod B_n$ defined by
\[
 \Phi(x)_{m} = \left\{\begin{array}{ll} \varphi_n(x_n) & m=g(n) \\ 0 & m\in F_2 \end{array}\right.
\]
is a lift of $\iso$. The maps $g$ and $(\varphi_n)_n$ are said to witness the well behaveness of $\iso$.

If there is $n_0$ such that for all $n\geq n_0$ the map $\varphi_n$ can be chosen to be an isomorphism, we say that $\iso$ is \emph{algebraically trivial}.
\end{definition}

\begin{remark}
If one allows $A_n$ or $B_n$ to have nontrivial central projections, it is easy to produce an example of an isomorphism of reduced products which is not asymptotically algebraic. For example, let $A_n=C_n\oplus D_n$, with $C_n$ and $D_n$ unital, and set $B_{2n}=C_n$ and $B_{2n+1}=D_n$. Then the identity isomorphism $\prod A_n/\bigoplus A_n\to\prod B_n/\bigoplus B_n$ is not asymptotically algebraic. One can reformulate the definition of asymptotically algebraic isomorphisms of reduced products, to obtain a more general notion allowing nontrivial central projections, by allowing the range of $\varphi_n$ to be contained in $\prod_{i\in F_n}B_i$, where $F_n$ is a finite set. In order to suppress an exponential growth of the notation, we stick to Definition~\ref{defin:trivialredprod}.
\end{remark}

There are three kinds of isomorphisms for reduced products: topologically trivial, asymptotically algebraic, and algebraically trivial isomorphisms. We focus on how these are related. (The Metric Approximation Property was introduced in Definition~\ref{defin:MAP}.)

\begin{proposition}\label{prop:trivials}
Let $A_n$ and $B_n$ be separable unital $\Cstar$-algebras, and
\[
 \iso\colon\prod A_n/\bigoplus A_n\to\prod B_n/\bigoplus B_n
\]
be an isomorphism. Then
\begin{enumerate}
\item\label{triv1} If $\iso$ is asymptotically algebraic then $\iso$ is topologically trivial;
\item\label{triv2} if $\iso$ is topologically trivial, each $A_n$ and $B_n$ have no nontrivial central projections, and each $A_n$ has the Metric Approximation Property, then $\iso$ is asymptotically algebraic.
\end{enumerate}
\end{proposition}

\begin{proof}
\ref{triv1}: Let $g\colon \NN\setminus F_1\to\NN\setminus F_2$ and $\varphi_n\colon A_n\to B_{g(n)}$ be the maps witnessing that $\iso$ is asymptotically algebraic. By Lemma~\ref{lemma:seqofappmaps} there is a sequence $\varepsilon_n\to 0$ such that each $\varphi_n$ is an $\varepsilon_n$-$^*$-isomorphism. Let
\[
 \Gamma^n=\set{(x,y)\in A_n\times B_{g(n)}}{\varphi_n(x)=y)}
\]
be the graph of $\varphi_n$, and consider its $2^{-n}$-fattening
\[
 \Gamma^{n,2^{-n}}=\{(x,y)\mid \exists (w,z)\in \Gamma^n (\norm{x-w},\norm{y-z}<2^{-n})\}.
\]
$\Gamma^{n,2^{-n}}$ is open in the norm topology and each of its sections is nonmeager. By Theorem~\ref{thm:BorelUnif}, we can find a norm-norm Borel function $\psi_n\colon A_n\to B_{g(n)}$ uniformizing $\Gamma^{n,\varepsilon_n}$. Define
\[
\prod\varphi_n\colon \prod A_n\to\prod B_n 
\]
by 
\[
(\prod\varphi_n((a_n)_n))_k=\begin{cases}
\varphi_n(a_n)&\text{ if there is }n \text{ such that }k=g(n)\\
0&\text{ else},
\end{cases}
\]
and notice that $\prod\varphi_n$ lifts $\Lambda$. 
Equally, we can define $\prod\psi_n$. For $(a_n)_n\in\prod A_n$, we have that $\lim_n\norm{\varphi_n(a_n)-\psi_n(a)}=0$, and therefore for all $x\in\prod A_n$ we have that $\prod\psi_n(x)-\prod\varphi_n(x)\in \bigoplus B_n$. In particular $\prod\psi_n$ lifts $\Lambda$. Since on bounded sets the strict topology on $\prod A_n$ coincides with the product of the norm topologies on each $A_n$, the function $\prod\psi_n$ is Borel when restricted to the set of contractions in $\prod A_n$. Since it lifts $\Lambda$, the latter is topologically trivial.

\ref{triv2}: Suppose $\iso$ is topologically trivial, and that $A_n$ and $B_n$, for each $n$, have no nontrivial central projections and have the metric approximation property. Let $V[G]$ be a forcing extension of the universe $V$ which satisfies $\OCA+\MA_{\aleph_1}$ and has the same $\omega_1$ as $V$ (e.g., \cite[\S2]{Velickovic.OCA1}). Let $\overline{A}_n$ and $\overline{B}_n$ be the completions of $A_n$ and $B_n$, respectively, in $V[G]$. Then, $\overline{A}_n$ and $\overline{B}_n$ are $\Cstar$-algebras with the metric approximation property and no nontrivial central projections; this is because of Shoenfield's absoluteness theorem, as the sentence `there is a nontrival central projection' is absolute, being $\forall\exists$. Since, since each $A_n$ and each $B_n$ has the MAP in $V$, the same operator systems and maps show that each $A_n$ (and each $B_n$) has the MAP in $V[G]$. Since having the MAP is a local property, each $\overline{A}_n$ (and each $\overline{B}_n$) has the MAP. Let $\Gamma$ be the graph of $\iso$; by reinterpreting the Borel code of $\Gamma$ in the extension $V[G]$, we obtain a Borel subset $\overline{\Gamma}$ of $\prod \overline{A}_n \times \prod \overline{B}_n$. The statement ``$\Gamma$ defines an isomorphism'' is $\mathbf{\Pi}^1_2$; thus by Schoenfield's absoluteness theorem, $\overline{\Gamma}$ defines, in $V[G]$, an isomorphism $\overline{\iso}$ from $\prod \overline{A}_n / \bigoplus \overline{A}_n$ to $\prod \overline{B}_n / \bigoplus \overline{B}_n$. Since $\OCA+\MA_{\aleph_1}$ holds in $V[G]$, by Theorem~\ref{theoi:redprod}, $\overline{\iso}$ is asymptotically algebraic. Finally, the statement ``$\overline{\iso}$ is asymptotically algebraic'' is $\mathbf{\Sigma}^1_2$, so again by Schoenfield's absoluteness,  $\iso$ is asymptotically algebraic in $V$.
\end{proof}
Using \cite[Theorem 4.17]{V.Rigidity} in place of Theorem~\ref{theoi:redprod} one can remove the Metric Approximation Property from the hypotheses above.

We provide the key link between Ulam stability and rigidity for quotients in the context of reduced products. What follows is the correspondent in this setting of \cite[\S4, Proposition 6]{Farah.Lifts}. (This was first stated and proved by the second author as \cite[Theorem 5.6 and 5.9]{V.Rigidity}. We reproduced the proof for completeness.)

\begin{theorem}\label{thm:apmapstrivial}
Let $\mathcal C$ be a class of unital separable $\Cstar$-algebras with no nontrivial central projections. Then:
\begin{enumerate}
\item\label{ulam1} $\mathcal C$ is Ulam stable if and only if whenever $A_n$ and $B_n$ are elements of $\mathcal C$ all asymptotically algebraic isomorphisms between $\prod A_n/\bigoplus A_n$ and $\prod B_n/\bigoplus B_n$ are algebraically trivial.
\item\label{ulam2} $\mathcal C$ is stable under approximate isomorphisms if and only if whenever $A_n$ are elements of $\mathcal C$ and $B_n$ are unital separable $\Cstar$-algebras with no nontrivial central projections, then if $\prod A_n/\bigoplus A_n$ and $\prod B_n/\bigoplus B_n$ are isomorphic by an asymptotically algebraic isomorphism we have that $B_n\in\mathcal C$ for all but finitely many $n\in\NN$.
\item\label{ulam3} $\mathcal C$ is uniformly classifiable by approximate isomorphisms if and only if whenever $A_n$ and $B_n$ are elements of $\mathcal C$, if there is an asymptotically algebraic isomorphism between $\prod A_n/\bigoplus A_n$ and $\prod B_n/\bigoplus B_n$ then there is an algebraically trivial one.
\end{enumerate}
\end{theorem}
\begin{proof}
As the proofs of the three statements are analogous, we only prove \eqref{ulam1} and leave the rest to the reader. 

Suppose that $\mathcal C$ is Ulam stable. Fix algebras $A_n$ and $B_n$ in $\mathcal C$ and an asymptotically algebraic isomorphism $\iso\colon\prod A_n/\bigoplus A_n\to\prod B_n/\bigoplus B_n$, Let $g$ be an almost permutation on $\NN$,  $\e_n\to 0$ be a sequence, and $\varphi_n\colon A_n\to B_{g(n)}$ be $\e_n$-isomorphisms witnessing that $\iso$ is asymptotically algebraic. Recall that $\prod \varphi_n$ is a lift for $\iso$. Without loss of generality, by reindexing the $B_n$'s and eventually forgetting about finitely many coordinates, we can assume that $g$ is the identity map and that $\{\e_n\}$ is a decreasing sequence. Let $k_0=0$. If $k_i$ has been defined for all $i<n$, let $k_n$ be the minimum natural greater than $k_{n-1}$ with the property that all $\e_{k_n}$-isomorphisms between elements of $\mathcal C$ can be perturbed up to $2^{-n}$ to isomorphisms. Such a $k_n$ can be defined since $\mathcal C$ is Ulam stable. If $j\in [k_n,k_{n+1})$, let $\psi_j\colon A_j\to B_j$ be an isomorphism with $\norm{\psi_j-\varphi_j}<2^{-n}$. Since for all $(x_n)\in\prod A_n$ we have that $\limsup_n\norm{\varphi_n(x_n)-\psi(x_n)}=0$, we have that $\prod\varphi_n$ and $\prod\psi_n$ define the same map between $\prod A_n/\bigoplus A_n$ and $\prod B_n/\bigoplus B_n$. Such map is $\iso$, which is then algebraically trivial.

Conversely, suppose $\mathcal C$ is not Ulam stable.  Then there is $\e>0$ such that for all $\delta>0$ there is are two $A_\delta$ and $B_\delta$ in $\mathcal C$ and a $\delta$-isomorphism $\varphi_\delta\colon A_\delta\to B_\delta$ that cannot be perturbed to an isomorphism which is uniformly, over the unit ball of $A_\delta$, $\e$ far away from $\varphi_\delta$. Let $A_n=A_{1/n}$, $B_n=B_{1/n}$ and $\varphi_n=\varphi_{1/n}$. Then $\prod\varphi_n$ induced an asymptotically algebraic isomorphism between $\prod A_n/\bigoplus A_n$ and $\prod B_n/\bigoplus B_n$ by Proposition~\ref{prop:welldefinedmaps}. It is routine to check that such isomorphism is not algebraically trivial.
\end{proof}
 \section{A lifting theorem I: Statements}\label{sec:liftstatements}

 In this section we state, and then outline the proof of, Theorem~\ref{thm:lifting}, the lifting results that lies behind all of our main results. The proof is broken into a sequence of lemmas which we will then prove in \S\ref{sec:lemma_proofs}. The following will be fixed for this and the next section.
 
\begin{notation} \label{notation}
We fix:
\begin{itemize} 
\item a sequence of finite-dimensional Banach spaces $\{E_n\}$;
\item a separable, nonunital $\Cstar$-algebra $A$;
\item an increasing approximate identity for $A$ consisting of positive contractions, $\{e_n\}$, with $e_{n+1}e_n=e_n$ for all $n\in\NN$ (such an approximate identity always exists, see e.g., \cite[Proposition 1.9.3]{Farah.Book}).
\end{itemize}
Moreover
\begin{itemize}
\item $\pi$ is the quotient map $\pi\colon \SM(A)\to\SQ(A)$, and $\pi_E$ is the quotient map $\pi_E\colon \prod E_n\to\prod E_n/\bigoplus E_n$.
\end{itemize}
\end{notation}
 All definitions and results are given with these objects in mind. If $I\subseteq\NN$ is a finite interval, we write
 \[
 e_I = e_{\max(I)} - e_{\min(I)}.
 \]
 Given $S\subseteq\NN$, we write
 \[
 E[S] = \prod_{n\in S} E_n
 \]
 If $S\subseteq T$, we view $E[S]$ as the linear subspace of $E[T]$ consisting of those elements with support contained in $S$. For $x\in E[\NN]$ we write $x\restriction S$ for the unique element of $E[S]$ which is equal to $x$ on the coordinates in $S$, and $\supp(x)$ for the set of non-zero entries of $x$. With this, se have that $E[S]=\{x\in\prod E_n\mid \supp(x)\subseteq S\}$. We write $\norm{\cdot}$ for the sup-norm on $E[\NN]$; however, we often work with the (separable, metrizable) product topology on $E[\NN]$ instead of the norm topology. In particular, any discussion involving descriptive-set-theoretic concepts (Borel sets, Baire measurability, etc.) refers to the product topology.

 \begin{definition}\label{defin:asympt}
Working in the setting of Notation~\ref{notation}, let $\alpha \colon E[\NN]\to \SM(A)$ be a function. We say that $\alpha$ is \emph{asymptotically additive} if there exists a sequence of finite intervals $I_n\subseteq\NN$, for $n\in\NN$, and functions $\alpha_n \colon E_n\to e_{I_n} A e_{I_n}$ with $\min(I_n)\to\infty$ as $n\to\infty$ and such that, for each $x\in E[\NN]$, the sum
 \[
 \sum_{n=0}^\infty \alpha_n(x_n)
 \]
 converges, in the strict topology, to $\alpha(x)$.

 On the other hand, we say that $\alpha$ is \emph{block diagonal} if there exist finite intervals $I_n,J_n\subseteq\NN$, for $n\in\NN$, and functions $\alpha_n \colon E[J_n]\to e_{I_n} A e_{I_n}$ such that
 \begin{itemize}
 \item $\NN = \bigcup I_n = \bigcup J_n$,
 \item $n\neq m$ implies $I_n\cap I_m = \emptyset$ and $J_n\cap J_m = \emptyset$, and
 \item for each $x\in E[\NN]$, the sum
 \[
 \sum_{n=0}^\infty \alpha_n(x\rs J_n)
 \]
 converges, in the strict topology, to $\alpha(x)$.
 \end{itemize}
 \end{definition}

 The functions $\alpha_n$ involved in the above definitions are not assumed to have any structure other than what is described; in particular, they are not assumed to be linear, $^*$-preserving, multiplicative, or even continuous. Our block diagonal maps resemble asymptotically additive ones as in \cite[Definition 1.5.1]{Farah.AQ}, in which the intervals $I_n$ were required to be disjoint. The reason for which we need to add a level of complexity is given by the highly noncommutative nature of the objects of our interest. In fact, although every block diagonal function is asymptotically additive, the converse does not hold in general. Also, while a finite sum of asymptotically additive functions is asymptotically additive, the sum of two block diagonal functions may not be block diagonal. If $A$ is a $\Cstar$-algebra and $e\in A$ is positive, $eAe$ is said a corner of $A$.
 
 \begin{lemma}\label{lemma:aa-corners}
Working in the setting of Notation~\ref{notation}, let $\alpha \colon E[\NN]\to \SM(A)$ be an asymptotically additive map and let $q\in \SM(A)$ be positive. Then there is an asymptotically additive map $\gamma\colon E[\NN]\to\SM(A)$ such that $\pi(\gamma(x)) = \pi(q\alpha(x)q)$ for every $x\in E[\NN]$.
 \end{lemma}

 \begin{proof}
 Choose functions $\alpha_n \colon E_n\to e_{I_n} A e_{I_n}$ (where $I_n$ is a finite interval in $\NN$) witnessing that $\alpha$ is asymptotically additive. By Lemma~\ref{lem:stratification} there are finite intervals $J_n^i\subseteq\NN$ for each $n\in\NN$ and $i = 0,1$, and $q_0$ and $q_1$ in $\SM(A)$, such that
 \begin{itemize}
 \item for each $i = 0,1$, the intervals $J_n^i$, for $n\in\NN$, are disjoint and consecutive,
 \item for each $n\in\NN$ and $i = 0,1$, $q_i$ commutes with $e_{J_n^i}$, and
 \item $\pi(q) = \pi(q_0 + q_1)$.
 \end{itemize}
 For each $n\in\NN$, set
 \[
 K_n = \bigcup\set{J_m^i}{m\in\NN\land i\in\{0,1\}\land J_m^i\cap I_n\neq\emptyset}
 \]
 Then $K_n$ is a finite interval containing $I_n$, and $\min(K_n)\to \infty$ as $n\to\infty$. Moreover $\gamma_n^{ij}(x) = q_i \alpha_n(x)q_j$ is in $e_{K_n} A e_{K_n}$ for all $x\in E_n$, so
 \[
 \gamma^{i,j}(x) = \sum_{n=0}^\infty \gamma_n^{i,j}(x_n)
 \]
 converges in the strict topology for each $x\in E[\NN]$, and the resulting function $\gamma^{ij}$ is asymptotically additive. Letting $\gamma = \gamma^{0,0} + \gamma^{0,1} + \gamma^{1,0} + \gamma^{1,1}$, we have
 \[
 \pi(\gamma(x)) = \pi(q \alpha(x) q).\qedhere
 \]
 \end{proof}

The proof of the following Proposition is a straightforward modification of that of Lemma~\ref{lemma:seqofappmaps}. (Notice that finite-dimensionality of the $E_n$'s is not needed).

\begin{proposition}\label{prop:FA.ApproxStruct}
Working in the setting of Notation~\ref{notation},  let $\iso\colon \prod E_n/\bigoplus E_n\to\SQ(A)$ be a map. Suppose
 \[
 \sum\alpha_n\colon E[\NN]\to \SM(A)
 \]
 is an asymptotically additive lift of $\iso$ on a dense ideal $\CI$. Then for every $\varepsilon>0$
 \begin{enumerate}
 \item if $\iso$ is linear there is $n_0$ such that for every $n\geq n_0$, then $\sum_{n_0\leq j\leq n}\alpha_j$ is $\varepsilon$-linear;
 \item if $E_n$ is an operator system and $\iso$ is also $^*$-preserving then there is $n_0$ such that for every $n\geq n_0$, $\sum_{n_0\leq j\leq n}\alpha_j$ is $\varepsilon$-$^*$-preserving;
 \item if $E_n$ is a Banach algebra and $\iso$ is multiplicative then there is $n_0$ such that for all $n\geq n_0$, $\sum_{n_0\leq j\leq n}\alpha_j$ is $\varepsilon$-multiplicative.
 \end{enumerate}
 Also,
 \begin{enumerate}\setcounter{enumi}{3}
 \item if $\supp(x)\in\CI$ and $\iso$ is norm-preserving, then $\lim_n | \norm{x_n}-\norm{\alpha_n(x_n)}|=0$.\qed
 \end{enumerate}
 \end{proposition}

 \begin{definition}\label{defin:preserves}
We work in the setting of Notation~\ref{notation}. A function
 \[
 \varphi \colon \prod E_n / \bigoplus E_n \to \SQ(A)
 \]
  \emph{preserves the coordinate structure} if there are positive contractions $p_S\in\SM(A)$ for $S\subseteq\NN$ such that
 \begin{enumerate}[label=(\roman*)]
 \item each $ \pi(p_S)$ commutes with $\varphi(\prod E_n / \bigoplus E_n)$,
 \item if $S\subseteq\NN$ and $x\in E[\NN]$,
 \[
 \pi(p_S)\varphi(\pi(x))=\varphi(\pi(x\restriction S)),\]
 \item if $S$ is finite, $p_S\in A$,
 \item if $S$ and $T$ are disjoint subsets of $\NN$ then $\pi(p_{S\cup T})=\pi(p_T)+\pi(p_S)$, and
 \item if $S\cap T$ is finite then $\pi(p_S)\pi(p_T)=0$.
 \end{enumerate}
 \end{definition}

 Examples of such functions are given by $^*$-homomorphisms, in case each $E_n$ is a unital $\Cstar$-algebra. In this case, $\pi(p_S)$ can be chosen to be a projection, but in general this is not true. For example,  consider $A=C_0(\er^+)$, and let $f_n$ be a positive function of norm $1$ supported on $[n-\frac{1}{4},n+\frac{1}{4}]$. Consider the map $\Phi\colon\ell_\infty\to \SM(A)=C_b(\er^+)$ obtained by $\Phi((a_n))=\sum a_nf_n$. Such a $\Phi$ induces a contraction $\varphi\colon\ell_\infty/c_0\to\SQ(A)$, obtained by taking the quotient map. To show that such a map preserves the coordinate structure, pick $g_n$ to be any positive contraction whose support is included in $[n-\frac{1}{3},n+\frac{1}{3}]$ and such that $g_n\restriction [n-\frac{1}{4},n+\frac{1}{4}]=1$. For $S\subseteq \NN$, let $p_S=\sum_{n\in S}g_n$. These elements witness that $\varphi$ preserves the coordinate structure.
\subsection{The `noncommutative' $\OCA$ lifting Theorem}
 We can now state the main result of this and the following section. Recall that an ideal $\CI\subseteq\SP(\NN)$ is ccc/Fin if it meets every uncountable almost disjoint family $\SA\subseteq\SP(\NN)$ (see \S\ref{subsec:DSTandIdeals}).

 \begin{theorem}[Noncommutative $\OCA$ lifting Theorem] \label{thm:lifting}
 Assume $\OCA$ and MA$_{\aleph_1}$. Let $E_n$ be finite dimensional Banach spaces, and let $A$ be a separable $\Cstar$-algebra. Let
 \[
 \varphi \colon \prod E_n / \bigoplus E_n \to \SQ(A)
 \]
 be a bounded, linear map, which preserves the coordinate structure. Then there is an asymptotically additive function $\alpha \colon E[\NN]\to \SM(A)$ and a ccc/Fin ideal $\CI$ such that for all $S\in\CI$ and $x\in E[S]$,
 \[
 \varphi(\pi_E(x)) = \pi(\alpha(x)).
 \]
 Moreover, there are sequences $\seq{I_n}{n\in\NN}$ and $\seq{J_n}{n\in\NN}$ of consecutive, finite intervals in $\NN$ such that for all $x\in E[I_n]$,
 \[
 \alpha(x)\in e_{J_{n-1}\cup J_n\cup J_{n+1}} A e_{J_{n-1}\cup J_n\cup J_{n+1}}
 \]
 In particular, $\alpha$ is the sum of three block diagonal functions.
 \end{theorem}

\begin{remark}
When dealing with isomorphisms, or maps obtained by restricting isomorphisms of coronas, one usually finds that $\CI=\SP(\NN)$, essentially by arguments similar to the one  of Proposition~\ref{prop:onedominates}. This is not the case for embeddings, as \cite[Example 3.2.1]{Farah.AQ} and the nontrivial copies of $\beta\NN\setminus\NN$ exhibited by Dow in \cite{Dow.Copy} show, where the ideal $\CI$ of Theorem~\ref{thm:lifting} is not equal to $\SP(\NN)$. As Vaccaro recently showed (\cite{Vaccaro}), this cannot happen for endomorphisms of the Calkin algebra. Understanding for which coronas we can have such behaviour seems related to commutativity and the Murray-von Neumann semigroup, but we do not have, at the moment, a picture which is clear enough to dare to conjecture anything substantial.
\end{remark}

For a bounded, linear $\varphi \colon \prod E_n / \bigoplus E_n \to \SQ(A)$ which preserves the coordinate structure. Our efforts are focused on finding lifts which have various desirable properties with respect to the ambient topological structure of $E[\NN]$.

 \begin{definition}
Working in the setting of Notation~\ref{notation}, let $\varphi\colon\prod E_n/\bigoplus E_n\to \SQ(A)$ be a bounded linear map which preserves the coordinate structure, let $\e \ge 0$ be given and $X\subseteq E[\NN]$.
 \begin{itemize}
 \item An \emph{$\e$-lift of $\varphi$ on $X$} is a function $F$ with $X\subseteq \dom(F)\subseteq E[\NN]$ and $\ran(F)\subseteq \SM(A)$, such that $\norm{\pi(F(x)) - \varphi(\pi_E(x))} \le \e$ for all $x\in X$ with $\norm{x} \le 1$.
 \item A \emph{$\sigma$-$\e$-lift of $\varphi$ on $X$} is a sequence of functions $F_n$, for $n\in\NN$,  with $X\subseteq \dom(F_n)\subseteq E[\NN]$ and $\ran(F_n) \subseteq \SM(A)$, such that for all $x\in X$ with $\norm{x} \le 1$, there is $n\in\NN$ such that $\norm{\pi(F_n(x)) - \varphi(\pi_E(x))} \le \e$.
 \item If $\CJ\subseteq\SP(\NN)$, we say that $F$ is an \emph{$\e$-lift of $\varphi$ on $\CJ$} if for every $S\in\CJ$, $F$ is an $\e$-lift of $\varphi$ on $E[S]$.
 \end{itemize}
 When $\e = 0$ we refer to \emph{lifts} and \emph{$\sigma$-lifts} on $X$.
 \end{definition}

From now on, we  work towards the proof of Theorem~\ref{thm:lifting}. We need therefore to update our notation.

\begin{notation}\label{notation2}
Working in the setting of Notation~\ref{notation}, we fix a linear bounded map $\varphi\colon\prod E_n/\bigoplus E_n\to \SQ(A)$ which preserves the coordinate structure. Let $\e\geq 0$.  We define $\CI^\e$ (respectively $\CI_C^\e$) to be the set of $S\subseteq\NN$ such that there exists an asymptotically additive (resp. $\Cmeas$-measurable) $\e$-lift of $\varphi$ on $E[S]$. When $\e = 0$ we write $\CI$ and $\CI_C$.
\end{notation}

 Although a general asymptotically additive function $\alpha \colon E[\NN]\to \SM(A)$ may not have any topological structure, if $\alpha$ is an $\e$-lift of $\varphi$ on $X$ then there is a $\Cmeas$-measurable (in fact, even Borel-measurable), asymptotically additive $\beta$ such that $\pi(\beta(x)) = \pi(\alpha(x))$ for all $x\in E[\NN]$. (This can be obtained by replacing $\alpha$ with a skeletal map; for details, see Lemma~\ref{lemma:skeletal-replacement}.) In particular, $\CI^\e \subseteq \CI_C^\e$ for each $\e \ge 0$.

 \begin{lemma}\label{lemma:ideals}
Working in the setting of Notation~\ref{notation2}, for each $\e \ge 0$, $\CI^\e$ and $\CI_C^\e$ are ideals on $\NN$.
 \end{lemma}
 \begin{proof}
 Clearly, each $\CI^\e$ and $\CI_C^\e$ is hereditary. We want to show that $\CI^\e_C$ and $\CI^\e$ are closed under finite unions.

 For $\CI^\e_C$, let $S$ and $T$ in $\CI_C^\e$. Without loss of generality we may assume that $S\cap T = \emptyset$. Let $F$ and $G$ be C-measurable functions $E[\NN]\to \SM(A)$ which are $\e$-lifts of $\varphi$ on $E[S]$ and $E[T]$ respectively. Define a function $H \colon E[\NN]\to \SM(A)$ by
 \[
 H(x) = p_SF(x\restriction S)p_S + p_TG(x\restriction T)p_T
 \]
 Then $H$ is $\Cmeas$-measurable, and if $x\in E[S\cup T]$ has norm at most 1, then
 \[
 \norm{\pi(F(x\restriction S)) - \varphi(\pi_E(x\restriction S))}, \norm{\pi(G(x\restriction T)) - \varphi(\pi_E(x\restriction T))} \le \e
 \]
 and
 \begin{align*}
 \pi(H(x)) - \varphi(\pi_E(x)) & = \pi(p_S)(\pi(F(x\restriction S)) - \varphi(\pi_E(x\restriction S))) \pi(p_S) \\
 & + \pi(p_T)(\pi(G(x\restriction T)) - \varphi(\pi_E(x\restriction T)))\pi(p_T).
 \end{align*}
Since $\pi(p_S)$ and $\pi(p_T)$ are orthogonal, it follows that $\norm{\pi(H(x)) - \varphi(\pi_E(x))} \le \e$, hence $ S\cup T\in\CI^\e_C$.

 Now suppose that in the above, $F$ and $G$ are asymptotically additive. Lemma~\ref{lemma:aa-corners} shows that there is an asymptotically additive $\gamma$ such that $\pi(\gamma(x)) = \pi(H(x))$ for all $x\in E[\NN]$, and hence $S\cup T\in\CI^\e$.
 \end{proof}

 The following five lemmas form the bulk of the proof of Theorem~\ref{thm:lifting}. They tell us that, under the assumption of $\OCA+\MA_{\aleph_1}$, the ideals $\CI^\e$ and $\CI^\e_C$ are, in various senses, large. We state them here, but their proofs, which are long and self-contained, will be deferred to Section~\ref{sec:lemma_proofs}.

 \begin{lemma} \label{lemma:oca->treelike}
Working in the setting of Notation~\ref{notation2}, assume $\OCA$. Let $\e > 0$ and let $\SA\subseteq \SP(\NN)$ be a treelike almost-disjoint family. Then for all but countably many $S\in\SA$, there is a $\sigma$-$\e$-lift of $\varphi$ on $E[S]$ consisting of C-measurable functions.
 \end{lemma}

 \begin{lemma} \label{lemma:sigma->single}
Working in the setting of Notation~\ref{notation2}, let $\e > 0$ and $S\subseteq\NN$. Suppose that there is a $\sigma$-$\e$-lift of $\varphi$ on $E[S]$, consisting of $\Cmeas$-measurable functions, and that $S = \bigcup_{n=1}^\infty S_n$ is a partition of $S$ into infinite sets. Then there is $n\in \NN$ such that $S_n\in \CI_C^{4\e}$.
 \end{lemma}

 \begin{lemma} \label{lemma:oca->alternative}
Working in the setting of Notation~\ref{notation2}, assume $\OCA$ and MA$_{\aleph_1}$. Then either
 \begin{enumerate}
 \item there is an uncountable, treelike, almost disjoint family $\SA\subseteq\SP(\NN)$ which is disjoint from $\CI$, or
 \item for every $\e > 0$, there is a sequence $F_n \colon E[\NN]\to \SM(A)$, for $n\in\NN$, of $\Cmeas$-measurable functions such that for every $S\in\CI^\e$, there is  $n\in\NN$ such that $F_n$ is an $\e$-lift of $\varphi$ on $E[S]$.
 \end{enumerate}
 \end{lemma}

 \begin{lemma} \label{lemma:sigma-baire->borel}
Working in the setting of Notation~\ref{notation2}, suppose $F_n \colon E[\NN] \to \SM(A)$, for $n\in\NN$,  is a sequence of Baire-measurable functions, $\e > 0$, and $\CJ \subseteq\SP(\NN)$ is a nonmeager ideal such that for all $S\in\CJ$, there is $n\in\NN$ such that $F_n$ is an $\e$-lift of $\varphi$ on $E[S]$. Then there is a Borel-measurable map $H \colon E[\NN] \to \SM(A)$ which is a $12\e$-lift of $\varphi$ on $\CJ$.
 \end{lemma}

 \begin{lemma} \label{lemma:C->asymptotic}
Working in the setting of Notation~\ref{notation2}, suppose $F \colon E[\NN] \to \SM(A)$ is a $\Cmeas$-measurable function and $\CJ\subseteq\SP(\NN)$ is a nonmeager ideal such that, for every $S\in\CJ$, $F$ is a lift of $\varphi$ on $E[S]$. Then there is an asymptotically additive $\alpha$ such that, for all $S\in\CJ$, $\alpha$ is a lift of $\varphi$ on $E[S]$; in fact, $\alpha$ is the sum of three block diagonal functions.
 \end{lemma}

Lemma~\ref{lemma:oca->treelike} is proved as Lemma~\ref{lemma:oca->treelike2}, \ref{lemma:sigma->single} is \ref{lemma:sigma->single2}, \ref{lemma:oca->alternative} is \ref{lemma:oca->alternative2}, \ref{lemma:sigma-baire->borel} is \ref{lemma:sigma-baire->borel2}, and \ref{lemma:C->asymptotic} is \ref{lemma:C->asymptotic2}.
 We end the section with a proof of Theorem~\ref{thm:lifting}, using the above lemmas.

 \begin{lemma} \label{lemma:epsilon-baire->asymptotic}
Working in the setting of Notation~\ref{notation2}, suppose that $\CJ\subseteq \SP(\NN)$ is a nonmeager ideal and that for every $\e > 0$, there is a Baire-measurable $G_\e \colon E[\NN]\to \SM(A)$ which is an $\e$-lift of $\varphi$ on $\CJ$. Then there is an asymptotically additive lift of $\varphi$ on $\CJ$.
 \end{lemma}

 \begin{proof}
 By Lemma~\ref{lemma:sigma-baire->borel} (setting $F_n = G_\e$ for each $n$) we may assume that $G_\e$ is actually Borel-measurable. Define
 \[
 \Gamma = \set{(x,y)}{\forall n\in\NN\quad \norm{\pi(y - G_{1/n}(x))} \le 1/n}
 \]
Since $A$ is Borel in $\SM(A)$, by the same argument in \cite[Lemma 17.4.3]{Farah.Book} (see also \cite[\S17.6]{Farah.Book}), $\Gamma$ is a Borel set. Moreover, if $S\in\CJ$ and $x\in E[S]$, then for any $y$ such that $\pi(y) = \varphi(\pi_E(x))$, we have $(x,y)\in \Gamma$. Let $F$ be a $\Cmeas$-measurable uniformization of $\Gamma$ given by the Jankov-von Neumann theorem (Theorem~\ref{thm:JVN}). Then $F$ is a lift of $\varphi$ on $\CJ$; hence Lemma~\ref{lemma:C->asymptotic} implies that there is an asymptotically additive lift of $\varphi$ on $\CJ$.
 \end{proof}

 \begin{corollary} \label{corollary:epsilon-C->C}
 $\bigcap_{\e > 0} \CI^\e_C = \CI$.
 \end{corollary}

 \begin{proof}
By the remark preceding Lemma~\ref{lemma:ideals}, we have that $\CI\subseteq \CI^\e_C$ for all $\e$, so we only need to show the converse inclusion. Given $S\in \bigcap_{\e > 0} \CI^\e_C$, apply Lemma~\ref{lemma:epsilon-baire->asymptotic} with $\CJ = \SP(S)$. This shows that $S\in\CI$.
 \end{proof}

We are ready to prove Theorem~\ref{thm:lifting}, which asserts that, if we are given a separable $\Cstar$-algebra $A$ and a sequence of finite-dimensional Banach spaces $E_n$, then if  $\varphi\colon\prod E_n/\bigoplus E_n\to \SQ(A)$ is a bounded linear map which preserves the coordinate structure, then, under $\OCA$ and $\MA_{\aleph_1}$, one can find an asymptotically additive (Definition~\ref{defin:asympt}) map $\alpha\colon\prod E_n\to\SM(A)$ which is a lift for $\varphi$ on a ccc/Fin ideal.

 \begin{proof}[Proof of Theorem~\ref{thm:lifting}]
 Suppose the first alternative of Lemma~\ref{lemma:oca->alternative} holds; that is, there is an uncountable, treelike, almost disjoint family $\SA\subseteq\SP(\NN)$ which is disjoint from $\CI$. By Corollary~\ref{corollary:epsilon-C->C}, there is an $\e > 0$ such that $\CI_C^\e$ is disjoint from an uncountable subset of $\SA$. Without loss of generality, we may assume $\CI_C^\e$ and $\SA$ are disjoint. By $\MA_{\aleph_1}$, there is an uncountable  almost disjoint family $\SB$ such that for every $T\in\SB$, there are infinitely many $S\in\SA$ such that $S\subseteq^* T$. This follows from the fact that under Martin's Axiom there cannot be $(\aleph_0,\aleph_1)$-gaps in $\mathcal P(\omega)$, see \cite[III.3.82]{Kunen.2011}. By \cite[Lemma~2.3]{Velickovic.OCAA}, there is an uncountable $\SC\subseteq\SB$ and, for each $T\in\SC$, a partition $T = T_0\cup T_1$, such that for each $i = 0,1$, the family $\SC_i = \set{T_i}{T\in\SC}$ is treelike. By applying Lemma~\ref{lemma:oca->treelike} to $\SC_0$ and $\SC_1$, for all but countably elements $T\in\SC$ there are $\sigma$-$\e/4$-lifts of $\varphi$ on $E[T_0]$ and $E[T_1]$, consisting of $\Cmeas$-measurable functions. Since $E[T]=E[T_0]+E[T_1]$, as in the proof of Lemma~\ref{lemma:ideals} we can show that for all but countably many $T\in\SC$ there are $\sigma$-$\e/4$-lifts of $\varphi$ on $E[T]$, consisting of $\Cmeas$-measurable functions. Fix such a $T$. Let $\{S_n\}$ be sets such that, for all $n\in\NN$, $S_n\in\SA$ and $S_n\subseteq^* T$. Since there are $\sigma$-$\e/4$-lifts of $\varphi$ on $E[\bigcup S_n\cap T]$, by Lemma~\ref{lemma:sigma->single} there is $n\in \NN$ such that $S_n\cap T\in\CI_C^\e$. Since $S_n\subseteq^*T$, and $S_n\in\SA$ we have a contradiction.

 Thus the first alternative of Lemma~\ref{lemma:oca->alternative} fails, hence $\CI$ meets every uncountable treelike almost disjoint family, and the second alternative of Lemma~\ref{lemma:oca->alternative} holds. 
 Hence, by Lemma~\ref{lemma:sigma-baire->borel}, for every $\e > 0$ there is a Borel-measurable map which is an $\e$-lift of $\varphi$ on $E[S]$ for every $S\in\CI$. By Lemma~\ref{lemma:epsilon-baire->asymptotic}, it follows that there is an asymptotically additive function which lifts $\varphi$ on $\CI$, and moreover $\alpha$ is the sum of three block diagonal functions.

 Finally, we show that $\CI$ is not only nonmeager but ccc/Fin. (Recall the definition of ccc/Fin, \S\ref{subsec:DSTandIdeals}). Let $\SA$ be an arbitrary uncountable almost disjoint family. By \cite[Lemma~2.3]{Velickovic.OCAA}, there is an uncountable $\SB\subseteq\SA$ and, for each $S\in\SB$, a partition $S = S_0\cup S_1$, such that for each $i = 0,1$, the family $\SB_i = \set{S_i}{S\in\SB}$ is treelike. Since $\CI$ contains uncountably many elements of $\SB_0$ and $\SB_1$, there is $S\in\SB$ such that $S_0,S_1\in\CI$, hence $S\in\CI$, that is, $\CI$ meets $\SB$; since $\SA$ is arbitrary, $\CI$ is ccc/Fin.
 \end{proof}

 \section{A lifting theorem II: Proofs} \label{sec:lemma_proofs}
The notation for this section is the one fixed in Notation~\ref{notation2}: in particular the algebra $A$, its approximate identity $\{e_n\}$, and the finite-dimensional Banach spaces $\{E_n\}$ are fixed. Equally, we fix a linear bounded map 
\[
\varphi\colon\prod E_n/\bigoplus E_n\to\SQ(A)
\] which preserves the coordinate structure. The positive elements $p_S$, for $S\subseteq\NN$, are chosen  accordingly in $\SM(A)$ (see Definition~\ref{defin:preserves}), and fixed for the entire section.

 For each $n$ we fix a set $X_n$ such that $X_n$ is a finite, $2^{-n}$-dense subset of the unit ball of $E_n$, and $0,1_{E_n}\in X_n$. We call $X_n$ the \emph{skeleton} of $E_n$. We define $\rho_n\colon (E_n)_{\le 1}\to X_n$ to be the map which is the identity on $X_n$ and which sends $x\in (E_n)_{\le 1}$ to the first element of $X_n$ (according to some fixed linear order) which is within $\leq 2^{-n}$ of $x$. Note that $\rho_n\circ \rho_n = \rho_n$. Moreover, $\rho_n$ is Borel-measurable. Let $\rho=\prod\rho_n$.  Clearly, for all $x\in E[\NN]$ with $\norm{x}\le 1$, $\pi(x) = \pi(\rho(x))$.

If $\alpha_n\colon E_n\to A$ are maps with the property that for all sequences $(x_n)\in E[\NN]$ the sequence $(\sum_{m\leq n}\alpha_m(x_m))_n$ strictly converges in $\SM(A)$, we write
\[
\sum\alpha_n\colon E[\NN]\to\SM(A)
\]
for the map sending $x$ to the strict limit of $(\sum_{m\leq n}\alpha_m(x_m))_n$.

 \begin{definition}\label{defin:skeletal}
Let $A$ be a $\Cstar$-algebra and $n\in\NN$. A map $\alpha_n\colon E_n\to A$ is \emph{skeletal} if for all $x\in E_n$ with $\norm{x}\le 1$, we have $\alpha_n(x) = \alpha_n(\rho_n(x))$, where the $\rho_n$'s are the Borel functions choosing elements of $X_n$ defined in the previous paragraph. If $\sum\alpha_n\colon E[\NN]\to\SM(A)$, with $\alpha_n\colon E_n\to A$, we say that $\sum\alpha_n$ is skeletal if each $\alpha_n$ is.
 \end{definition}

 \begin{lemma} \label{lemma:skeletal-replacement}
Working in the setting of Notation~\ref{notation2}, let $\e\geq 0$. Suppose $F$ is an $\e$-lift of $\varphi$ on a set $Z$ with the property that $Z+\bigoplus E_n\subseteq Z$. Then the function $G = F\circ \rho$ is a skeletal $\e$-lift of $\varphi$ on $Z\cap (E[\NN])_{\le 1}$. If $F$ is asymptotically additive, so if $G$.
 \end{lemma}
 \begin{proof}
Fix $z\in Z$, and write $z=(z_n)_n$ with $z_n\in E_n$. If $F\colon\prod E_n\to\SM(A)$ is an $\e$-lift of $\varphi$, by definition, we have that $\norm{\pi(F(z))-\varphi(\pi_E(z))}<\varepsilon$. Since $z-\rho(z)\in\bigoplus E_n$, then $\rho(z)\in Z$, hence $F$ is an $\e$-lift on $\rho(z)$, that is, $\norm{\pi(F(\rho(z)))-\varphi(\pi_E(\rho(z)))}<\varepsilon$. As $\varphi(\pi_E(\rho(z)))=\varphi(\pi_E(z))$, and $G(z)=F\circ\rho(z)$, we have the thesis.

For the second statement, notice that if $F$ is asymptoptically additive, so is $F\restriction\ran(\rho)$.
 \end{proof}

 We write $X[S]$ for the product $\prod_{n\in S} X_n$, viewed as a subset of $E[S]$. We view $X[\NN]$ as the set of branches through a finitely-branching tree, with the Cantor-space topology. We also view elements of $X[S]$ as functions with domain $S$. Hence if $S\subseteq T$, $x\in X[S]$ and $y\in X[T]$, then $x\subseteq y$ means that $y$ extends $x$, or in other words that $x_n = y_n$ whenever $n\in S$. In this situation we write $y\rs S = x$. If $x$ and $y$ are elements of $X[\NN]$ with  a common extension we denote the minimal one by $x\cup y$. This happens if and only if $x_n=y_n$ whenever $n\in\supp(x)\cap\supp(y)$. We make frequent use of Proposition~\ref{prop:talagrand} with the spaces $X[S]$ and $\SP(\NN)$.

We want to prove Lemma~\ref{lemma:oca->treelike} above. Its proof takes its form from \cite[Lemma 2.2]{Velickovic.OCAA}, already generalized to the continuous setting in \cite[Lemma 7.2]{Farah.C} (see also \cite[Lemma 4.6]{McKenney.UHF}). 
\begin{remark}
The reader familiar with Veli\v{c}kovi\`c's argument should note that our proofs requires additional indexing, and that it is not enough to choose the homogeneous sets $\SH_m$ in the `obvious way', after the proper partition has been defined. A further refinement had to be made: this was done in order to accommodate for the fact that the sets $e_n\SM(A)e_n$ are not compact. (This is the case if  $A=\mathcal K(H)$, or $A=\bigoplus A_n$ where each $A_n$ is finite-dimensional). This further layer of difficulty  makes more difficult the (crucial) definition of the sets $\SF_{k,j,n,m}$, which now need additional indexes. On the other hand the pay up is higher, as we are able to get C-measurable lifts for maps between quotients over ideals which are not $\sigma$-compact in any meaningful way.
\end{remark}

For $a,b\in\SM(A)$ and $\e>0$, we write $a\sim_\e b$ for $\norm{a-b}<\e$ and $a\sim_{m,\e} b$ for $\norm{(1-e_m)(a-b)(1-e_m)}<\e$.

 \begin{lemma}\label{lemma:oca->treelike2}
Working in the setting of Notation~\ref{notation2}, assume $\OCA$. Let $\e > 0$ and let $\SA\subseteq \SP(\NN)$ be a treelike almost-disjoint family. Then for all but countably many $S\in\SA$, there is a $\sigma$-$3\e$-lift of $\varphi$ on $E[S]$ consisting of C-measurable functions.
 \end{lemma}
\begin{proof}

 Fix $\e > 0$ and a treelike, almost-disjoint family $\SA\subseteq\SP(\NN)$. Let $F$ be an arbitrary lift of $\varphi$ such that $\norm{F(x)}\le \norm{\varphi} \norm{x}$ for all $x\in E[\NN]$. First of all, notice that by Lemma~\ref{lemma:skeletal-replacement}, it is enough to find a $\sigma$-$\e$-lift on $X[S]$ for all but countably many elements of $\SA$.

 Let $f\colon \NN\to 2^{<\omega}$ be a bijection witnessing that $\SA$ is treelike, and for each $S\subseteq\NN$, let $\tau(S) = \bigcup f[S]$, the branch of $2^\NN$ containing the image of $S$. Note that if $T\in\SA$ and $S$ is an infinite subset of $T$, then $\tau(S) = \tau(T)$.

 Let $\SR$ be the set of all pairs $(S,x)$ such that for some $T\in\SA$, $S$ is an infinite subset of $T$, and $x\in X[S]$. We define colourings $[\SR]^2 = K_0^m\cup K_1^m$ by $\{(S,x),(T,y)\}\in K_0^m$ if and only if

 \begin{enumerate}[label=($K$-\arabic*)]
 \item\label{K0:branches} $\tau(S)\neq\tau(T)$,
 \item\label{K0:equal} $x\rs (S\cap T) = y\rs (S\cap T)$, and
 \item\label{K0:unequal} $F(x) p_T \not\sim_{m,\e} p_S F(y)$ or $p_T F(x) \not\sim_{m,\e} F(y)p_S$.
 \end{enumerate}

 Note that $K_0^m\supseteq K_0^{m+1}$ for every $m$. We give $\SR$ the separable metric topology obtained by identifying $(S,x)\in\SR$ with the tuple
 \[
 (S,\tau(S),x,F(x),p_S)\in \SP(\NN)\times \SP(\NN)\times X[\NN]\times \SM(A)_{\le \norm{\varphi}} \times \SM(A)_{\le 1}
 \]
 where $\SM(A)$ is given the strict topology, and $\SP(\NN)$ the usual Cantor set topology.
The following claim is the reason for which we require $\SA$ to be treelike, and why \ref{K0:branches} is necessary in the definition of $K_0^m$.
 \begin{claim}\label{claimK0open}
 For each $m$, $K_0^m$ is open.
 \end{claim}
 \begin{proof}
 Suppose $\{(S,x),(T,y)\}\in K_0^m$. By~\ref{K0:branches}, there is $n\in\NN$ such that $\tau(S)\rs n\neq \tau(T)\rs n$. Let $s = f^{-1}(2^n)$; then $S\cap T\subseteq s$. By~\ref{K0:unequal} we may choose $p\in\NN$ and $\delta > 0$ such that
 \[
 \norm{e_p(1 - e_m)(F(x) p_T - p_S F(y))(1 - e_m)} > \e + \delta
 \]
 or
 \[
 \norm{e_p(1 - e_m)(p_T F(x) - F(y) p_S)(1 - e_m)} > \e + \delta.
 \]
 Let $U$ be the set of pairs $\{(\bar{S},\bar{x}),(\bar{T},\bar{y})\}$ in $[\SR]^2$ such that
 $S\cap s = \bar{S}\cap s$, $T\cap s = \bar{T}\cap s$, $x\rs s = \bar{x}\rs s$, $y\rs s = \bar{y}\rs s$, and
 \begin{align*}
 \norm{e_p(1 - e_m) p_T (F(x) - F(\bar{x}))(1 - e_m)} & + \norm{e_p(1 - e_m) F(\bar{y}) (p_S - p_{\bar{S}})(1 - e_m)} \\
 + \norm{e_p(1 - e_m) p_S (F(y) - F(\bar{y})) (1 - e_m)} & + \norm{e_p(1 - e_m)F(\bar{x})(p_T - p_{\bar{T}}) (1 - e_m)} <
\delta
 \end{align*}
 Then $U$ is an open neighbourhood of the pair $\{(S,x),(T,y)\}$, and moreover $U\subseteq K_0^m$.
 \end{proof}

 Recall that for $a$ and $b$ in $2^\NN$, $\Delta(a,b) = \min\set{n}{a(n)\neq b(n)}$.
 \begin{claim}
 The first alternative of $\OCA_\infty$ fails for the colours $K_0^m$ ($m\in\NN$), that is, if $Z\subseteq 2^\NN$ is uncountable and $\zeta \colon Z\to \SR$ is an injection, then there exist distinct $a$ and $b$ in $2^\NN$ such that $\{\zeta(a),\zeta(b)\}\in K_1^{\Delta(a,b)}$.
 \end{claim}
 \begin{proof}
 Let $\zeta\colon Z\to \SR$ be an injection, where $Z\subseteq 2^\NN$ is uncountable, and suppose for sake of contradiction that $\{\zeta(a),\zeta(b)\}\in K_0^{\Delta(a,b)}$ for all distinct $a$ and $b$ in $Z$. Let $\SH = \zeta[Z]$, and let $z\in X[\NN]$ be such that $z\rs S = x$ for all $(S,x)\in\SH$. (Such a $z$ exists by~\ref{K0:equal}.) Then for all $(S,x)\in\SH$, we have
 \[
 \pi(F(z)p_S) = \pi(p_SF(z)) = \pi(F(x))
 \]
 and hence there is $m\in\NN$ such that
 \[
 (1 - e_m)F(z)p_S \sim_{\e/2\norm{\varphi}} (1 - e_m)p_SF(z) \sim_{\e/2\norm{\varphi}} (1 - e_m)F(x),
 \]
 and
 \[
 F(z)p_S(1 - e_m) \sim_{\e/2\norm{\varphi}} p_SF(z)(1 - e_m) \sim_{\e/2\norm{\varphi}} F(x)(1 - e_m).
 \]
 We may thus refine $Z$ to an uncountable subset (which we still call $Z$) such that for a fixed
 $m\in\NN$, the above holds for all $(S,x)\in\SH = \zeta[Z]$. Since $Z$ is uncountable, we may find distinct $a$ and $b$ in $Z$ with $\Delta(a,b)\ge m$. Let $(S,x) = \zeta(a)$ and $(T,y) = \zeta(b)$. Then,
 \[
 F(x)p_T \sim_{m,e/2} p_SF(z)p_T \sim_{m,\e/2} p_SF(y)
 \]
 which implies $F(x)p_T \sim_{m,\e} p_SF(y)$. Similarly, we have $p_TF(x) \sim_{m,\e} F(y)p_S$, and this contradicts that
 $\{(S,x),(T,y)\}\in K_0^m$.
 \end{proof}

By our assumption of $\OCA$ (and hence $\OCA_\infty$), there exists a sequence $\SH_m$ ($m\in\NN$) of sets covering $\SR$, such that $[\SH_m]^2\subseteq K_1^m$. For $n\in\NN$, let $\SH_{j,n,m}$, for $j\in\NN$, be sets such that 
\begin{itemize}
\item $\bigcup_{j}\SH_{j,n,m}=\SH_m$ for all $m$ and $n$;
\item if $(S,x)$ and $(S',x')$ are in $\SH_{j,n,m}$, then $e_np_S\sim_\e e_np_{S'}$ and $e_nF(x)\sim_\e e_nF(x')$. 
\end{itemize}
 Such sets can be found since $e_n\SM(A)=e_nA$ is separable. Fix $\SD_{j,n,m}$ a
 countable dense subset of $\SH_{j,n,m}$, and for $k\in \NN$, let $\SF_{k,j,n,m}
 \subseteq \SD_{j,n,m}$ be a finite set with the property that if $(S,x)\in
 \SD_{j,n,m}$ then there is $(S',x')\in \SF_{k,j,n,m}$ with the property that
 $S\cap k=S'\cap k$ and $x\rs k=x'\rs k$. This finite set can be found since both the set of all traces of elements of $\SD_{j,n,m}$, $\{S\cap k\mid S\in \SD_{j,n,m}\}$, and the set of all $x\rs k$ for $x\in X[\NN]$, are finite, being a subset of $\SP(k)$ and $\prod_{i\leq k}X_k$ respectively. Let 
 \[
 \SC=\{\tau(S)\mid S\in\bigcup_{j,n,m}\SD_{j,n,m}\}.
 \]
Then  $\SC$ is countable. We claim that for every $T\in\SA\setminus\SC$  there is a $\sigma$-$3\e$-lift of $\varphi$ on $X[T]$ consisting of C-measurable functions. 
 
Fix such $T$. For $k\in\NN$, let $k^+$ be the minimum natural number greater than $k$ such that 
 \[
 k^+>\max\{T\cap S\mid S\in\bigcup_{j,n,m\leq k}\SF_{k,j,n,m}\}.
 \]
 Since each $T\cap S$, for $S\in\bigcup_{j,n,m\leq k}\SF_{k,j,n,m}$, is finite, and $\bigcup_{j,n,m\leq k}\SF_{k,j,n,m}$ is finite, $k^+$ exists. Let $k_0=0$ and, for $i\in\NN$, $k_{i+1}=k_i^+$. Define 
 \[
 T_0=T\cap \bigcup_{i}[k_{2i},k_{2i+1}) \,\,\,\text{ and } T_1=T\setminus T_0.
 \]
Let $\Lambda_m^0$ be the set of all pairs $(x,z)$ such that $x\in X[T_0]$, $z-zp_{T_0}\in A$, and for all $k$ there is $(S,y)\in \bigcup_{j,n}\SD_{j,n,m}$ with 
\begin{enumerate}
\item $x\rs (S\cap T_0)=y\rs (S\cap T_0)$
\item $S\cap k=T_0\cap k$
\item $e_kp_S\sim_\e e_kp_{T_0}$
\item $e_kF(y)\sim_\e e_kz$.
\end{enumerate}

 Since each $\SD_{j,n,m}$ is countable and each of the conditions (1)-(4) defines a Borel set, $\Lambda^0_m$ is Borel.

\begin{claim}
If $(T_0,x)\in\SH_m$, then $(x,F(x))\in\Lambda_m^0$.
\end{claim} 
\begin{proof}
Fix $n$. Let $j$ such that $(T_0,x)\in \SH_{j,n,m}$. Fix $i$ such that $k_{2i+1}>j,n,m$. By density of $\SD_{j,n,m}$, we can find $(S,y)\in\SD_{j,n,m}$ such that $S\cap k_{2i+1}=T_{0}\cap k_{2i+1}$ and $x\rs k_{2i+1}=y\rs k_{2i+1}$. By definition of $\SF_{k_{2i+1},j,n,m}$, there is $(S',y')\in \SF_{k_{2i+1},j,n,m}$ such that $S\cap k_{2i+1}=S'\cap k_{2i+1}$ and $y\rs k_{2i+1}=y'\rs k_{2i+1}$. Notice that $T\cap S'\subseteq k_{2i+2}$ and $T_0\cap[k_{2i+1},k_{2i+2})=\emptyset$, so $T_0\cap S'\subseteq k_{2i+1}$, that is, 
\[
x\rs (S'\cap T_0)=x\rs k_{2i+1}=y\rs k_{2i+1}=y'\rs k_{2i+1}=y'\rs (S'\cap T_0),
\]
since $S'\cap k_{2i+1}=S\cap k_{2i+1}=T_{0}\cap k_{2i+1}$. Moreover, since $(S',y')\in \SF_{k_{2i+1},j,n,m}\subseteq\SD_{j,n,m}\subseteq\SH_{j,n,m}$, we have that $e_np_{T_0}\sim_\e e_np_{S'}$ and $e_nF(x)\sim_\e e_nF(y')$. Since $n$ is arbitrary, this concludes the proof.
\end{proof}

\begin{claim}
If $(T_0,x)\in\SH_m$ and $z$ is such that $(x,z)\in \Lambda_m^0$ then
$\norm{\pi(p_{T_0}F(x)-zp_{T_0})}\le 3\e$.
\end{claim}
\begin{proof}
It is enough to show that $\norm{(1-e_m)(p_{T_0}F(x)-zp_{T_0})(1-e_m)}\le 3\e$. Suppose not, then there is $n$ such that 
\[
\norm{e_n(1-e_m)(p_{T_0}F(x)-zp_{T_0})(1-e_m)}>3\e.
\]
Since $(x,z)\in\Lambda_m^0$ we may choose $(S,y)\in\SH_m$ such that 
\[
x\rs (S\cap T_0)=y\rs (S\cap T_0),\,\,\ e_np_S\sim_\e e_np_{T_0},\,\,\text{ and }e_nz\sim_\e e_nF(y).
\]
 Notice that since $(T_0,x)$ and $(S,y)$ are in $\SH_m$, by the $K_1^m$-homogeneity of
 $\SH_m$, we have $p_SF(x)\sim_{m,\e} F(y)p_{T_0}$. Then
\begin{eqnarray*}
e_n(1-e_m)p_{T_0}F(x)(1-e_m)&\sim_\e& (1-e_m)e_np_{S}F(x)(1-e_m)\\
&=& e_n(1-e_m)p_{S}F(x)(1-e_m)\\
&\sim_\e&e_n(1-e_m)F(y)p_{T_0}(1-e_m)\\
&=&(1-e_m)e_nF(y)p_{T_0}(1-e_m)\\
&\sim_\e&(1-e_m)e_nzp_{T_0}(1-e_m)\\
&=&e_n (1-e_m)zp_{T_0}(1-e_m).
\end{eqnarray*}

Bringing all together, we have that 
\[3\e<\norm{e_n(1-e_m)(p_{T_0}F(x)-zp_{T_0})(1-e_m)}\leq3\e.
\]
 This is a contradiction.
\end{proof}
The same construction leads to the definition of $\Lambda_m^1$, for $m\in\NN$.

 For each $m\in\NN$ and $i < 2$, let $F_m^i$ be a $\Cmeas$-measurable uniformization of $\Lambda_m^{T_i}$, which exists by Theorem~\ref{thm:JVN}. Define 
 \[
 G_m^i(x) = p_{T_i}F_m^i(x)p_{T_i};
 \]
  then $G_m^i$ is $\Cmeas$-measurable, and by the above claim, for each $i < 2$ and $x\in X[T_i]$, there is $m$ such that $G_m^i(x)$ is defined and
 \[
 \norm{\pi(G_m^i(x)) - \varphi(\pi_E(x))} \le 3\e.
 \]
Let $H_{m,n}(x) = G_m^0(x\rs T_0) + G_n^1(x\rs T_1)$ for $m,n\in\NN$. If $x\in
X[T]$, pick natural numbers $n$ and $m$ such that
 \[
\max\{ \norm{\pi(G_m^0(x\restriction T_0)) - \varphi(\pi_E(x\restriction T_0))},  \norm{\pi(G_n^1(x\restriction T_1)) - \varphi(\pi_E(x\restriction T_1))}\} \le 3\e.
 \]
Notice that, as $G_m^i$ is obtained by cutting $F_m^i$ by $p_{T_i}$, and the images of $p_{T_i}$ are orthogonal in $\SQ(A)$, the elements $\pi(G_m^0(x\restriction T_i)) - \varphi(\pi_E(x\restriction T_i))$, for $i<2$, are orthogonal. Therefore the norm of their sum is equal to their maximum. As $x=x\restriction T_0+x\restriction T_1$, this shows that $\norm{\pi(H_{m,n}(x)) -\varphi(\pi_E(x))}\leq 3\e$, so the functions $H_{m,n}(x)$ form a $\sigma$-$3\e$-lift of $\varphi$ on $X[T]$.
\end{proof}

The following is Lemma~\ref{lemma:sigma->single}.
 \begin{lemma} \label{lemma:sigma->single2}
Working in the setting of Notation~\ref{notation2}, let $\e > 0$ and $S\subseteq\NN$. Suppose that there is a $\sigma$-$\e$-lift of $\varphi$ on $E[S]$, consisting of $\Cmeas$-measurable functions, and that $S = \bigcup S_n$ is a partition of $S$ into infinite sets. Then there is $n$ such that $S_n\in \CI_C^{4\e}$.
 \end{lemma}

 \begin{proof}
Working towards a contradiction, suppose that, for all $n$, $S_n\not\in\CI_C^{4\e}$. Let $F$ be an arbitrary lift of $\varphi$ on $E[\NN]$, and let $F_n$, for $n\in\NN$, be a $\sigma$-$\e$-lift of $\varphi$ on $E[S]$, consisting of $\Cmeas$-measurable functions. Since each $F_n$ is Baire-measurable, it follows that there is a comeager $Y\subseteq X[S]$ on which each $F_n$ is continuous. By Proposition~\ref{prop:talagrand}, there are a partition of $S$ into finite sets $t_i$ ($i\in\NN$), and $z_i\in X[t_i]$, such that $x\in Y$ whenever $x\in X[S]$ and $x\rs t_i = z_i$ for infinitely many $i$. Define 
 \[
 \CT_0 = \bigcup_it_{2i}, \,\,\CT_1 = \bigcup_i t_{2i+1},\,\,z_0^* = \sum_i z_{2i},\text{ and }z_1^* = \sum_i z_{2i+1}.
 \]
 Then the functions
 \[
 F_n'(x) = F_n(x\rs \CT_0 + z_1^*) - F_n(z_1^*) + F_n(x\rs \CT_1 + z_0^*) - F_n(z_0^*)
 \]
 are continuous on $X[S]$, and form a $\sigma$-$2\e$-lift of $\varphi$ on $X[S]$. Throughout the rest of the proof, we write $F_n$ for $F_n'$.

 For each $n$, let $T_n = \bigcup_{m>n} S_m$. We construct sequences
 \begin{itemize}
 \item $x_n\in X[S_n]$,
 \item $T_n^*\subseteq T_n$, and
 \item $z_n\in X[T_n^*]$,
 \end{itemize}
 such that for all $n < m$,
 \begin{enumerate}
 \item $S_n\sm T_m^*\not\in\CI_C^{4\e}$,
 \item $T_n^*\cap T_m\subseteq T_m^*$,
 \item $z_n\rs (T_n^*\cap T_m^*)\subseteq z_m$,
 \item $z_{n-1}\rs (T_{n-1}^*\cap S_n) \subseteq x_n$, and
 \item for all $y\in X[T_n]$, if $y\supseteq z_n$, then
 \[
 \norm{\pi(F_n(x_0\cup\cdots\cup x_n\cup y) - F(x_n))\pi(p_{S_n}))} > 2\e.
 \]
 \end{enumerate}
 The construction goes by induction on $n$. Suppose we have constructed $x_k$, $T_k^*$ and $z_k$ for all $k < n$. For each $x\in X[S_n\sm T_{n-1}^*]$ and $y\in \SM(A)_{\le 1}$, define $\SE_n(x,y)$ to be the set of $z\in X[T_n\sm T_{n-1}^*]$ such that
 \[
 \norm{\pi(F_n(x_0\cup\cdots\cup x_{n-1}\cup x\cup z_{n-1}\cup z) - y)\pi(p_{S_n})} \le 2\e.
 \]
 Since $F_n$ is continuous, $\SE_n(x,y)$ is Borel for each $x$ and $y$.

 \begin{claim}
 There is $x\in X[S_n\sm T_{n-1}^*]$ such that $\SE_n(x,F(x))$ is not comeager in $X[T_n\sm T_{n-1}^*]$.
 \end{claim}
 \begin{proof}
 Suppose otherwise. Let $\SR$ be the set of $(x,y)\in X[S_n\sm T_{n-1}^*]\times \SM(A)_{\le 1}$ such that $\SE_n(x,y)$ is comeager. Then $\SR$ is analytic and hence has a $\Cmeas$-measurable uniformization $G$. Hence for all $x\in X[S_n\sm T_{n-1}^*]$, $F(x)\cap G(x)\neq \emptyset$, and this implies
 \[
 \norm{\pi((F(x) - p_{S_n}G(x))p_{S_n})}\le 4\e
 \]
 which means that $x\mapsto p_{S_n}G(x)p_{S_n}$ is a $\Cmeas$-measurable $4\e$-lift of $\varphi$ on $X[S_n\sm T_{n-1}^*]$, contradicting our induction hypothesis.
 \end{proof}

 Let $x\in X[S_n\sm T_{n-1}^*]$ be such that $\SE_n(x,F(x))$ is not comeager, and let
 \[
 x_n = x\cup (z_{n-1}\rs (T_{n-1}^*\cap S_n)).
 \]
 Since $\SE_n(x,F(x))$ is Borel, there is a finite $a\subseteq T_n\sm T_{n-1}^*$ and $\sigma\in X[a]$ such that the set of $z\in\SE_n(x,F(x))$ extending $\sigma$ is meager. Applying Proposition~\ref{prop:talagrand}, we may find a partition of $T_n\sm (a\cup T_{n-1}^*)$ into finite sets $s_i$, and $u_i\in X[s_i]$, such that for any $z\in X[T_n\sm T_{n-1}^*]$, if $z$ extends $\sigma$ and infinitely many $u_i$, then $z\not\in\SE_n(x,F(x))$.

 \begin{claim}
 There is an infinite set $I\subseteq\NN$ such that
 \[
 S_m\sm (T_{n-1}^*\cup \bigcup_{i\in I} s_i)\not\in\CI^{4\e}_C
 \]
 for all $m > n$.
 \end{claim}
 \begin{proof}
 Recursively construct infinite sets $J_{n+1}\supseteq J_{n+2}\supseteq\cdots$ such that for each $m > n$,
 \[
 S_m\sm (T_{n-1}^*\cup \bigcup_{i\in J_m} s_i)\not\in\CI^{4\e}_C
 \]
 using the fact that $S_m\sm T_{n-1}^*\not\in\CI^{4\e}_C$ for all $m > n$. Any infinite $I\subseteq \NN$ such that $I\subseteq^* J_m$ for all $m > n$ satisfies the claim.
 \end{proof}

 Let $I$ be as in the claim, and put 
 \[
 T_n^* = T_n\cap (T_{n-1}^*\cup \bigcup_{i\in I} s_i).
 \] 
 Let 
 \[
 z_n = (z_{n-1}\rs (T_{n-1}^*\cap T_n))\cup \bigcup_{i\in I} u_i.
 \]
  This completes the construction.

 Now we let $x = \bigcup_n x_n$. Then $x\in X[S]$, and hence there is $n$ such that
 \[
 \norm{\pi(F_n(x) - F(x))}\le 2\e.
 \]
 Notice that if $y = \bigcup_{m>n} x_m$, then $x = x_0\cup\cdots\cup x_n\cup y$, $y\in X[T_n]$, and $y$ extends $z_n$; hence
 \[
 \norm{\pi(F_n(x) - F(x_n))\pi(p_{S_n})} > 2\e.
 \]
 But $\pi(F(x))\pi(p_{S_n}) = \pi(F(x_n))$. This is a contradiction.
 \end{proof}
The following is Lemma~\ref{lemma:oca->alternative}.
 \begin{lemma}\label{lemma:oca->alternative2}
Working in the setting of Notation~\ref{notation2}, assume $\OCA$ and MA$_{\aleph_1}$. Then either
 \begin{enumerate}
 \item\label{alt:1} there is an uncountable, treelike, almost disjoint family $\SA\subseteq\SP(\NN)$ which is disjoint from $\CI$, or
 \item\label{alt:2} for every $\e > 0$, there is a sequence $F_n\colon E[\NN]\to \SM(A)$, , for $n\in\NN$, of $\Cmeas$-measurable functions such that for every $S\in\CI^\e$, there is an $n$ such that $F_n$ is an $\e$-lift of $\varphi$ on $E[S]$.
 \end{enumerate}
 \end{lemma}

 \begin{proof}

 For each $S\in\CI$, fix an asymptotically additive, skeletal $\alpha^S$ such that $\alpha^S$ is a lift of $\varphi$ on $E[S]$, and $\alpha_n^S = 0$ whenever $n\not\in S$. Since $\alpha^S$ is skeletal, we may identify $\alpha^S$ with an element of the separable metric space $A^\NN$.

 We apply $\OCA_\infty$ instead of $\OCA$. Fix $\e > 0$, and define colourings $[\CI]^2 = K_0^m\cup K_1^m$ by placing $\{S,T\}\in K_0^m$ if and only if there are pairwise disjoint, finite subsets $w_0,\ldots,w_{m-1}$ of $(S\cap T)\sm m$, such that for all $i < m$, there is $(x_n^i)\in X[w_i]$ such that
 \[
 \norm{\sum_{n\in w_i} \alpha_n^S(x_n^i) - \alpha_n^T(x_n^i)} > \e.
 \]
 Define a separable metric topology on $\CI$ by identifying $S\in\CI$ with the pair $(S,\alpha^S)\in\SP(\NN)\times A^\NN$. Then each $K_0^m$ is open in $[\CI]^2$, and $K_0^m\supseteq K_0^{m+1}$.

 The proof now divides into two parts, in which we show that the two alternatives of $\OCA_\infty$ imply, respectively, \eqref{alt:1} and~\eqref{alt:2}.

 Suppose that the first alternative of $\OCA_\infty$ holds, and fix an uncountable $Z\subseteq 2^\NN$ and a map $\zeta\colon Z\to\CI$ such that for all $a$ and $b$ in $Z$, $\{\zeta(a),\zeta(b)\}\in K_0^{\Delta(a,b)}$. 
 
 We now define a poset $\PP$ which is intended to add an uncountable treelike, almost disjoint family disjoint from $\CI^\e$. (Compare this with the poset $\mathcal P_{\omega_1}$ as defined in \cite[p.89]{Farah.AQ}).
 
 The conditions $p\in\PP$ are of the form $p = (I_p,G_p,n_p,s_p,x_p,f_p)$, where
 \begin{enumerate}[label=($\PP$-\arabic*)]
 \item\label{PP-1} $I_p$ is a finite subset of $\omega_1$, $n_p\in\NN$, $G_p\colon I_p\to [Z]^{<\omega}$, $s_p\colon I_p\times n_p\to 2$, $x_p\colon I_p\to X[n_p]$, and $f_p\colon n_p\to 2^{<\omega}$,
 \item\label{PP-2} for all $\xi\in I_p$ and $m,n\in n_p$, if $s_p(\xi,m) = s_p(\xi,n) = 1$, then $f_p(m)$ and $f_p(n)$ are comparable, and
 \item\label{PP-3} for all $\xi\in I_p$ and distinct $S$ and $T$ in $\zeta[G_p(\xi)]$, there exists 
 \[
 w\subseteq \set{n < n_p}{s_p(\xi,n) = 1}
 \]
  such that
 \[
 \norm{\sum_{n\in w} \alpha_n^S(x_p(\xi,n)) - \alpha_n^T(x_p(\xi,n))} >\e.
 \]
 \end{enumerate}
 (We view $x_p$ as a function with domain $I_p\times n_p$ in the obvious way.) We define $p\le q$ if and only if
 \begin{enumerate}[label=($\le$-\arabic*)]
 \item $I_p\supseteq I_q$, $n_p\ge n_q$, $s_p\supseteq s_q$, $f_p\supseteq f_q$, $x_p\supseteq x_q$, for all $\xi\in I_q$, $G_p(\xi)\supseteq G_q(\xi)$, and
 \item for all $m,n\in [n_q,n_p)$, if there exist distinct $\xi$ and $\eta$ in $I_q$ such that $s_p(\xi,m) = s_p(\eta,n) = 1$, then $f_p(m)\perp f_q(n)$.
 \end{enumerate}

Let us spend few words on what the different parts of $\mathbb P$ are intended to do. The sets $I_p$ are an approximation to an uncountable subset of $\omega_1$, $I$, which will index our generic uncountable family of subsets of $\NN$, $\{S_\xi\mid\xi\in I\}$. The function $s_p\restriction(\xi,\cdot)$ decide which naturals are, or not, in $S_\xi$, while the element $x_p\restriction(\xi,\cdot)$ is designed to witness that $S_\xi$ is not in $\mathcal I^{\varepsilon}$. Finally, the functions $f_p$ are intended to show that $\{S_\xi\mid\xi\in I\}$ is treelike and almost disjoint. Before giving the precise definition of such objects, we prove that a generic intersecting $\aleph_1$ dense subsets of $\mathbb P$ exist.

 \begin{claim}
 $\PP$ has the ccc.\footnote{We prove a stronger condition than ccc. In fact, we show that every uncountable subset of $\mathbb P$ contains a set of pairwise compatible elements of size $\aleph_1$. This condition is known as the Knaster condition.}
 \end{claim}
 \begin{proof}
 Let $\SQ\subseteq\PP$ be uncountable. By refining $\SQ$ to an uncountable subset, we may assume that the following hold for $p\in\SQ$.
 \begin{enumerate}
 \item There are $n\in\NN$ and $f\colon N\to 2^{<\omega}$ such that $n_p = N$ and $f_p = f$ for all $p\in\SQ$,
 \item The sets $I_p$ ($p\in\SQ$) form a $\Delta$-system with root $J$, \footnote{If $\mathcal F$ is a family of sets, $\mathcal F$ is a $\Delta$-system with root $r$ if $F\cap G=r$ for all $F\neq G\in\mathcal F$.} and the tails $I_p\sm J$ have the same size $\ell$.
 \item For each $\xi\in J$, the sets $G_p(\xi)$ ($p\in\SQ$) form a $\Delta$-system with root $G(\xi)$, and the tails $G_p(\xi)\sm G(\xi)$ all have the same size $m(\xi)$.
 \item There are functions $t\colon J\times N\to 2$ and $y\colon J\to X[N]$ such that for all $(\xi,n)\in J\times N$ and $p\in\SQ$, $s_p(\xi,n) = t(\xi,n)$ and $x_p(\xi,n) = y(\xi,n)$.
 \item\label{condition-u} If $I_p\sm J = \{\xi_0^p < \cdots < \xi_{\ell-1}^p\}$, then the map $u\colon \ell\times N\to 2$ given by
 \[
 u(i,n) = s_p(\xi_i,n)
 \]
 is the same, for all $p\in\SQ$.
 \item\label{condition-large-Delta} If $\xi\in J$ and $G_p(\xi)\sm G(\xi) = \{z_0^p(\xi),\ldots,z_{m(\xi)-1}^p(\xi)\}$, then for all $p,q\in\SQ$ and $i < m(\xi)$, we have $\Delta(z_i^p(\xi),z_i^q(\xi)) \ge M$, where $M = \max\{N,\sum_{\xi\in J} m(\xi)\}$.
 \end{enumerate}

 Let $p,q\in\SQ$ be given; we claim that $p$ and $q$ are compatible.
 We define an initial attempt at an amalgamation $r = (I_r,G_r,n_r,s_r,x_r,f_r)$ as follows. Let \[
 I_r = I_p\cup I_q,\, n_r = N,\,f_r = f,\,s_r = s_p\cup s_q,\,\text{ and }x_r = x_p\cup x_q.
 \]
 For each $\xi\in I_r$, let $G_r(\xi) = G_p(\xi)\cup G_q(\xi)$. If $r$ were in $\PP$, then $r\leq p$ and $r\leq q$ would hold as required; however, condition~\ref{PP-3} may not be satisfied by $r$.

 It is easily verified that in the following cases, condition~\ref{PP-3} is already satisfied by $r$:
 \begin{enumerate}[label=(\roman*)]
 \item $\xi\not\in J$,
 \item $\xi\in J$ and $S$ and $T$ in $\zeta[G(\xi)]$, and
 \item $\xi\in J$ and $S = \zeta(z_i^p(\xi))$, $T = \zeta(z_j^q(\xi))$, where $i\neq j$.
 \end{enumerate}
 ((i) and (ii)  use the fact that $p$ and $q$ are in $\PP$; (iii) uses, in addition,~\eqref{condition-u} above.) For the last remaining case, fix $\xi\in J$ and $i < m(\xi)$, and put $S = \zeta(z_i^p(\xi))$, $T = \zeta(z_i^q(\xi))$. By~\eqref{condition-large-Delta}, we have $\{S,T\}\in K_0^M$, hence there are $M$ many pairwise-disjoint, finite subsets $w$ of $(S\cap T)\sm M$, such that
 \[
 \exists x\in X[w] \quad \norm{\sum_{n\in w} \alpha_n^S(x_n) - \alpha_n^T(x_n)} > \e.
 \]
 Since $M\ge \sum_{\xi\in J} m(\xi)$, we may choose pairwise disjoint, finite sets $w(\xi,i)$ for each $\xi\in J$ and $i < m(\xi)$, such that for each $\xi\in J$ and $i < m(\xi)$, $w(\xi,i)$ satisfies the above, with $S = \zeta(z_i^p(\xi))$ and $T = \zeta(z_i^q(\xi))$. Let $x^{\xi,i}\in X[w(\xi,i)]$ be the corresponding witness. Let $\bar{N}\ge M$ be large enough to include every set $w(\xi,i)$, and define $s \colon I_r\times \bar{N}\to 2$, $x \colon I_r\to X[\bar{N}]$, and $g \colon \bar{N}\to 2^{<\omega}$ so that
 \begin{itemize}
 \item $s\supseteq s_r$, $x\supseteq x_r$, and $g\supseteq f_r$,
 \item for all $\xi\in J$ and $i < m(\xi)$, and $n \in w(\xi,i)$, $s(\xi,n) = 1$ and $x(\xi,n) = x^{\xi,i}_n$,
 \item $s(\eta,k) = 0$ and $x(\eta,k) = 0$ for all other values of $(\eta,k)\in I_r\times \bar{N}$,
 \item for all $\xi\in J$ and
 \[
 n,n'\in \bigcup_{i< m(\xi)} w(\xi,i),
 \]
 $g(n)$ and $g(n')$ are comparable and extend $\bigcup\set{g(k)}{k < N\land s_r(\xi,k) = 1}$,
 \item for all distinct $\xi$ and $\eta$ in $J$, if
 \[
 n\in \bigcup_{i<m(\xi)} w(\xi,i)\qquad n'\in \bigcup_{i<m(\eta)} w(\eta,i),
 \]
 then $g(n)\perp g(n')$.
 \end{itemize}
 It follows that $r' = (I_r, G_r, \bar{N}, s, x, g)\in \PP$ and $r'\le p,q$, as required.
 \end{proof}

 By MA$_{\aleph_1}$, we may find a filter $\GG\subseteq\PP$ such that $I =
 \bigcup_{p\in\GG}I_p$ is uncountable, and for all $\xi\in I$,
 \[
 \SH_\xi = \zeta[\bigcup_{p\in\GG}G_p(\xi)]
 \]
 is uncountable. For each $\xi\in I$, let
 \[
 S_\xi = \bigcup\set{n}{\exists p\in\GG\; (n < n_p\land s_p(\xi,n) = 1)}.
 \]
 Then we may also assume that $S_\xi$ is infinite for all $\xi\in I$. The function $f = \bigcup_{p\in\GG} f_p$ witnesses that $\SA =
 \set{S_\xi}{\xi\in I}$ is a treelike, almost disjoint family. It remains to show that
 $\SA$ is disoint from $\CI^\e$.

 For each $\xi\in I$, define $x^\xi = \bigcup_{p\in\GG} x_p(\xi,\cdot)$. Note that for any $T$ and $T'$ in $\SH_\xi$, we have
 \begin{gather}\label{eqn:homogeneity}
 \exists w\in [T\cap T'\cap S_\xi]^{<\omega} \quad \norm{\sum_{n\in w} \alpha_n^T(x^\xi_n) - \alpha_n^{T'}(x^\xi_n)} > \e \tag{$\ast$}
 \end{gather}

 \begin{claim}
 For all $\xi\in I$, $S_\xi\not\in\CI^\e$.
 \end{claim}
 \begin{proof}
 Suppose otherwise, and fix $\xi\in I$ and an asymptotically additive $\beta$ which is an $\e$-lift of $\varphi$ on $E[S_\xi]$. For each $T\in\SH_\xi$, there is $N\in\NN$ such that for any finite $w\subseteq T\cap S_\xi\sm N$,
 \[
 \norm{\sum_{n\in w} \alpha_n^T(x^\xi_n) - \beta_n(x^\xi_n)} < \e
 \]
 since $\alpha^T$ and $\beta$ both lift $\varphi$ on $E[S_\xi\cap T]$. By the pigeonhole principle, there is $N$ such that the above holds for all $T$ in an uncountable $\SL\subseteq\SH_\xi$. Moreover, we may find distinct $T,$ and $T'$ in $\SL$ such that $\norm{\alpha_n^T(x) - \alpha_n^{T'}(x)} < \e$ for all $n\in T\cap T'\cap N$ and $x\in X_n$. This contradicts~\eqref{eqn:homogeneity}.
 \end{proof}

 This completes the first part of the proof. For the remainder we assume the second alternative of $\OCA_\infty$, and we prove~\eqref{alt:2}.

 Suppose $\SH\subseteq\CI$ satisfies $[\SH]\subseteq K_1^m$ for some $m$; we show that there is a $\Cmeas$-measurable function $F$ such that, for every $S\in\SH$, $F$ is an $\e$-lift of $\varphi$ on $E[S]$. Let $\SD\subseteq\SH$ be a countable, dense subset of $\SH$ in the topology on $\CI$ defined above. We define $\SR$ to be the subset of $X[\NN]_{\le 1}\times \SM(A)_{\le 1}$ consisting of those $(x,y)$ such that there is a sequence $T_p$ ($p\in\NN$) in $\SD$ for which $y$ is the strict limit of $\alpha^{T_p}(x)$ as $p\to\infty$, and $T_p$ converges (in $\SP(\NN)$) to a set containing the support of $x$. Since $\SD$ is countable, $\SR$ is analytic, and the density of $\SD$ in $\SH$ implies that $(x,\alpha^S(x))\in \SR$ for all $S\in\SH$ and $x\in E[S]_{\le 1}$.

 It suffices to prove that for all $S\in\SH$ and $(x,y)\in\SR$ with $x\in E[S]_{\le 1}$, we have $\norm{\pi(y - \alpha^S(x))}\le \e$; for then any $\Cmeas$-measurable uniformization $F$ of $\SR$ satisfies the required properties. So fix a sequence $T_p$ ($p\in\NN$) witnessing that $(x,y)\in\SR$. Suppose, for sake of contradiction, that $\norm{\pi(y - \alpha^S(x))} > \e$. Then there is $\delta > 0$ such that for all $k\in\NN$,
 \[
 \norm{(1 - e_k)(y - \alpha^S(x))} > \e + \delta.
 \]
 Since $\norm{y - \alpha^S(x))} > \e$, we may find $r_0\in\NN$ and $N_1\in\NN$ such that $T_{r_0}\cap N_1$ contains the support of $x\rs N_1$, and
 \begin{gather}
 \norm{\sum_{n\leq N_1 - 1} \alpha_n^{T_{r_0}}(x_n) - \alpha_n^S(x_n)} > \e + \delta.
 \end{gather}
 Since $\alpha^{T_{r_0}}$ and $\alpha^S$ are asymptotically additive, we may find $k_0$ such that the ranges of $\alpha_n^{T_{r_0}}$ and $\alpha_n^S$ are contained in $e_{k_0} A e_{k_0}$ for each $n < N_1$. Then, as $\norm{(1 - e_{k_0})(y - \alpha^S(x))} > \e + \delta$, we may find $r_1 > r_0$ and $N_2 > N_1$ such that $T_{r_1}\cap N_1 = T_{r_0}\cap N_1$, $T_{r_1}\cap N_2$ contains the support of $x\rs N_2$,
 \begin{gather} \label{cdn:close-1}
 \norm{\sum_{n\leq N_1 - 1} \alpha_n^{T_{r_0}}(x_n) - \alpha_n^{T_{r_1}}(x_n)} < \frac{\delta}{2},
 \end{gather}
 and
 \begin{gather}
 \norm{(1 - e_{k_0})\left(\sum_{n\leq N_2 - 1} \alpha_n^{T_{r_1}}(x_n) - \alpha_n^S(x_n)\right)} > \e + \delta.
 \end{gather}
 It follows that
 \begin{gather}
 \norm{\sum_{N_1\leq n\leq N_2-1} \alpha_n^{T_{r_1}}(x_n) - \alpha_n^S(x_n)} > \e + \frac{\delta}{2}.
 \end{gather}
 Repeating this construction, we may find $N_m > \cdots > N_1$ and a set $T = T_{r_{m-1}}\in \SD$ such that $T\cap N_m$ contains the support of $x\rs N_m$, and for each $i < m$,
 \[
 \norm{\sum_{N_i\leq n \leq N_{i+1}-1} \alpha_n^T(x_n) - \alpha_n^S(x_n)} > \e.
 \]
 Then $\{S,T\}\in K_0^m$, contradicting the $K_1^m$-homogeneity of $\SH$.
 \end{proof}

 A set $\SH\subseteq\SP(\NN)$ is \emph{everywhere nonmeager} if for every nonempty open $\SU\subseteq\SP(\NN)$, $\SH\cap \SU$ is nonmeager. A proof of the following can be found in~\cite[\S 3.10 and \S 3.11]{Farah.AQ}.

 \begin{lemma} \label{lemma:hnm}
Working in the setting of Notation~\ref{notation2}, let $\SH$ be a hereditary and nonmeager subset of $\SP(\NN)$. Then there is  $k$ such that $\set{S\subseteq\NN}{S\sm k\in\SH}$ is everywhere nonmeager. Moreover, if $\SH$ and $\SK$ are hereditary and everywhere nonmeager, then so is $\SH\cap\SK$.\qed
 \end{lemma}

The following is Lemma~\ref{lemma:sigma-baire->borel}
 \begin{lemma} \label{lemma:sigma-baire->borel2}
Working in the setting of Notation~\ref{notation2}, suppose $F_n\colon E[\NN] \to \SM(A)$, for $n\in\NN$, is a sequence of Baire-measurable functions, $\e > 0$, and $\CJ \subseteq\SP(\NN)$ is a nonmeager ideal such that for all $S\in\CJ$, there is  $n$ such that $F_n$ is an $\e$-lift of $\varphi$ on $E[S]$. Then there is a Borel-measurable map $H\colon E[\NN] \to \SM(A)$ which is a $12\e$-lift of $\varphi$ on $\CJ$.
 \end{lemma}

 \begin{proof}
 Let $\SH_n$ be the family of $S\in\CJ$ such that $F_n$ is an $\e$-lift of $\varphi$ on $E[S]$. Then each $\SH_n$ is hereditary, and $\CJ = \bigcup_n\SH_n$.

 \begin{claim}
 $\CJ$ is equal to the union of those $\SH_n$ which are nonmeager.
 \end{claim}

 \begin{proof}
 Let $\SK$ be the union of all of the meager $\SH_n$'s. Then $\SK$ is meager, so there is a sequence of finite sets $a_i$ ($i\in\NN$) such that no infinite union of the $a_i$'s is in $\SK$. Since $\CJ$ is a nonmeager ideal, there is an infinite $L$ such that $\bigcup_{i\in L} a_i\in\CJ$. Let $T=\bigcup_{i\in L} a_i$. Now suppose $S\in\CJ$. Then $S\cup T\in\CJ$, and hence there is $n$ such that $S\cup T\in \SH_n$. By construction $\SH_n$ is nonmeager, and since $\SH_n$ is hereditary, $S\in\SH_n$.
 \end{proof}
 We thus assume, without loss of generality, that every $\SH_n$ is nonmeager. By Lemma~\ref{lemma:hnm}, for each $n\in\NN$ there is $k_n$ such that the set $\SK_n = \set{S\subseteq\NN}{S\sm k_n\in\SH_n}$ is hereditary and everywhere nonmeager. Then, replacing $F_n$ with the function
 \[
 x\mapsto F_n(x\rs [k_n,\infty))
 \]
 defined on $X[\NN]$, we may assume that $F_n$ is an $\e$-lift of $\varphi$ on $X[S]$ for all $S\in\SK_n$.
 Since each $F_n$ is Baire-measurable, there is a comeager $Y\subseteq X[\NN]$ on which every $F_n$ is continuous. Then we may find a partition of $\NN$ into finite sets $a_i$ ($i\in\NN$), and elements $t_i$ of $X[a_i]$, such that any $x\in X$ satisfying $x\rs a_i = t_i$ for infinitely many $i$ must be in $Y$. Since $\CJ$ is nonmeager, there is an infinite $I\subseteq\NN$ such that $\bigcup_{i\in I}a_i\in\CJ$. Let $T=\bigcup_{i\in I}a_i$, and let $I = I_0\cup I_1$ be a partition into infinite sets. Set
 \[
 T_k = \bigcup_{i\in I_k} a_i \qquad \qquad t^k = \sum_{i\in I_k} t_i
 \]
 and
 \[
 G_n(x) = F_n(x\rs T_0 + t^1) - F_n(t^0) + F_n(x\rs \NN\sm T_0 + t^1) - F_n(t^1)
 \]
 It follows that each $G_n$ is continuous on $X[\NN]$. Moreover, if $S\in\CJ$, then since $S\cup T\in\CJ$, there is  $n$ such that $F_n$ is an $\e$-lift of $\varphi$ on $S\cup T$, and so $G_n$ is a $2\e$-lift of $\varphi$ on $S$.

For naturals $m$ and $n$, define
 \[
 \SL_{m,n} = \set{S\subseteq\NN}{\forall x\in X[S]\; \norm{\pi(G_m(x) - G_n(x))} \le 4\e}
 \]
 Then $\SL_{m,n}$ is hereditary, coanalytic, and contains $\SK_m\cap \SK_n$. Since $\SK_m$ and $\SK_n$ are everywhere nonmeager and hereditary, so is (from Lemma~\ref{lemma:hnm}) $\SK_m\cap \SK_n$, and hence $\SL_{m,n}$ is comeager. Define
 \[
 \SE = \bigcap_{m,n\in \NN} \SL_{m,n}
 \]
 Then $\SE$ is comeager, so we may find a partition of $\NN$ into finite sets $b_i$, along with sets $\sigma_i\subseteq b_i$, such that for any $S\subseteq\NN$, if $S\cap b_i = \sigma_i$ for infinitely many $i$, then $S\in\SE$. Since $\SE$ is hereditary, we may assume that $\sigma_i = \emptyset$ for each $i$. Let $T_0$ be the union of the even $b_i$'s, and $T_1$ the union of the odd $b_i$'s. Pick any particular $n^*\in\NN$, and define
 \[
 H(x) = G_{n^*}(x\rs T_0) + G_{n^*}(x\rs T_1)
 \]
 Since $G_{n^*}$ is continuous, so is $H$. We claim that $G$ is an $8\e$-lift of $\varphi$ on $X[S]$, for very $S\in\CJ$. So let $S\in\CJ$; then there is  $m$ such that $S\in\SH_m$. Since $S\cap T_k\in\SE$ for each $k = 0,1$, we have
 \begin{eqnarray*}
 \norm{\pi(H(x)) - \varphi(\pi_E(x))} &\le& \norm{\pi(G_{n^*}(x\rs T_0) - G_m(x\rs T_0))}\\& +& \norm{\pi(G_{n^*}(x\rs T_1) - G_m(x\rs T_1))} \\
& +& \norm{\pi(G_m(x\rs T_0)) - \varphi(\pi_E(x\rs T_0))} \\&+& \norm{\pi(G_m(x\rs T_1)) - \varphi(\pi_E(x\rs T_1))} 
  \le 12\e
 \end{eqnarray*}
 Finally, recall that $H$ is defined only on $X[\NN]$. We may extend $H$ to $E[\NN]$ using the map $\rho\colon E[\NN]\to X[\NN]$ defined above.
 \end{proof}
We are left to prove Lemma~\ref{lemma:C->asymptotic} above.
 \begin{lemma} \label{lemma:C->asymptotic2}
Working in the setting of Notation~\ref{notation2}, suppose $F \colon E[\NN] \to \SM(A)$ is a $\Cmeas$-measurable function and $\CJ\subseteq\SP(\NN)$ is a nonmeager ideal such that for every $S\in\CJ$, $F$ is a lift of $\varphi$ on $E[S]$. Then there is an asymptotically additive $\alpha$ such that for all $S\in\CJ$, $\alpha$ is a lift of $\varphi$ on $E[S]$; in fact, $\alpha$ is the sum of three block diagonal functions.
 \end{lemma}

 \begin{proof}
 As in the proof of Lemma~\ref{lemma:sigma-baire->borel2}, by the Baire-measurability of $F$ and the fact that $\CJ$ is nonmeager, we may assume that $F$ is actually continuous on $X[\NN]$. Given $x$, $y$ and $e$ in $\SM(A)$ we write
 \[
d_e(x,y) = \max\{\norm{e(x - y)}, \norm{(x - y)e}\}
 \]
 We write $X[\CJ]$ for $\bigcup_{S\in\CJ} X[S]$.

 \begin{claim}
 $X[\CJ]$ is a nonmeager subset of $X[\NN]$.
 \end{claim}

 \begin{proof}
 Suppose $X[\CJ]$ is meager. Then we may find an increasing sequence $n_i\in\NN$, and $s_i\in X[[n_i,n_{i+1})]$, such that for any $x\in X[\NN]$, if there are infinitely many $i\in\NN$ such that $x\supseteq s_i$, then $x\not\in X[\CJ]$. But since $\CJ$ is nonmeager, there is an infinite set $L\subseteq\NN$ such that
 \[
 \bigcup_{i\in L} [n_i,n_{i+1}) \in \CJ
 \]
 Now let $x = \sum_{i\in L}s_i$. (Recall that $0\in X_n$ for all $n\in\NN$, so $x\in X[\NN]$.) Then $x\in X[\CJ]$ but $x\supseteq s_i$ for all $i\in L$, a contradiction.
 \end{proof}

 \begin{claim}
 For each $n$ and $\e > 0$, there exists $k > n$ and $u\in X[[n,k)]$ such that for all $x$ and $y$ in $X[\NN]$, if $x\rs [k,\infty) = y\rs [k,\infty)$ and $x,y\supseteq u$, then $d_{1-e_k}(F(x),F(y))\le\e$.
 \end{claim}

 \begin{proof}
 Fix $n$ and $\e > 0$. For each $s\in X[[0,n)]$ and $x\in X[\NN]$, let $x(s) = s + x\rs [n,\infty)$. Define, for each $k\in\NN$,
 \[
 V_k = \set{x\in X[\NN]}{\exists s,t\in X[[0,n)]\text{ s.t. }d_{1-e_k}(F(x(s)),F(x(t))) > \e}
 \]
 By the continuity of $F$, $V_k$ is open. For each $s$ and $t$ in $X[[0,n)]$ and $x\in X[\CJ]$, there is  $k\in\NN$ such that $d_{1-e_k}(F(x(s)),F(x(t))) \le \e$, since $\pi(F(x(s))) = \varphi(\pi_E(x)) = \pi(F(x(t)))$. Since $X[[0,n)]$ is finite, it follows that for each $x\in X[\NN]$ there is  $k$ such that $x\not\in V_k$. So,
 \[
 X[\CJ]\cap \bigcap_{k=0}^\infty V_k = \emptyset.
 \]
 As $X[\CJ]$ is a nonmeager subset of $X[\NN]$, there must be $\ell$ such that $V_\ell$ is not dense in $X[\NN]$. Thus there is  $k \ge \ell$ and $u\in X[[0,k)]$ such that for all $x\in X[\NN]$, if $x\supseteq u$, then $x\not\in V_\ell$. Since for all $x\in X[\NN]$ and $s$ and $t$ in $X[[0,n)]$ we have $x(s)\in V_\ell$ if and only if $x(t)\in V_\ell$, we may take $u\in X[[n,k)]$. Finally, note that $V_k\subseteq V_\ell$, so if $x\in X[\NN]$ and $x\supseteq u$ then $x\not\in V_k$.
 \end{proof}

 \begin{claim} \label{claim:stabilizers}
 There are sequences $n_i\in\NN$, $k_i\in\NN$, and $u_i\in X[[n_i,n_{i+1})]$ such that $n_i < k_i < n_{i+1}$ and for any $x$ and $y$ in $X[\NN]$, if $x\supseteq u_i$ and $y\supseteq u_i$, then
 \begin{enumerate}
 \item\label{eq:stabilizers.tails} $x\rs [n_{i+1},\infty) = y\rs [n_{i+1},\infty)$ implies $d_{1-e_{k_i}}(F(x),F(y)) \le 2^{-i}$, and
 \item\label{eq:stabilizers.heads} $x\rs [0,n_i) = y\rs [0,n_i)$ implies $d_{e_{k_i}}(F(x),F(y))\le 2^{-i}$.
 \end{enumerate}
 \end{claim}

 \begin{proof}
 We go by induction on $i$. Set $n_0 = 0$. Given $n_i$, we first apply the previous claim with $n = n_i$ and $\e = 2^{-i}$ to find $k_i > n_i$ and $v_i\in X[[n_i,k_i)]$ such that for all $x$ and $y$ in $X[\NN]$, if $x\supseteq v_i$, $y\supseteq v_i$, and $x\rs [k_i,\infty) = y\rs [k_i,\infty)$ then $d_{1-e_{k_i}}(F(x),F(y)) \le 2^{-i}$. We then apply the continuity of $F$ to find $n_{i+1} > k_i$ and $u_i\in X[[n_i,n_{i+1})]$ such that $u_i\supseteq v_i$, and for all $x$ and $y$ in $X[\NN]$, if $x\supseteq u_i$ and $y\supseteq u_i$ and $x\rs [0,n_i) = y\rs [0,n_i)$, then $d_{1-e_{k_i}}(F(x),F(y))\le 2^{-i}$.
 \end{proof}

 Given $\zeta < 3$, we define
 \[
 T_\zeta = \bigcup\set{[n_i,n_{i+1})}{i\equiv \zeta\pmod{3}}
 \]
 and
 \[
 v_\zeta = \bigcup\set{u_i}{i\equiv \zeta\pmod{3}}.
 \]
 For each $i$, set $q_i = e_{n_{i+2}} - e_{n_{i-1}}$, and for $i\equiv\zeta\pmod{3}$, let
 \[
 \alpha_i(x) = q_i (F(x\rs [n_i,n_{i+1}) + v_{\zeta+1} + v_{\zeta+2}) - F(v_{\zeta+1}) - F(v_{\zeta+2})) q_i
 \]
 where $\zeta+1$ and $\zeta+2$ are computed modulo $3$. Set
 \[
 \beta_\zeta(x) = \sum_{i\equiv\zeta\pmod{3}} \alpha_i(x).
 \]
 Note that $q_i\perp q_j$ whenever $|i - j| \ge 3$, so the sum in the definition of $\beta_\zeta$ converges strictly. Moreover, $\beta_\zeta$ is block diagonal.

 \begin{claim}
 For each $S\in\CJ$ and $\zeta < 3$, $\beta_\zeta$ lifts $\varphi$ on $X[S\cap T_\zeta]$.
 \end{claim}

 \begin{proof}
 Clearly, the function
 \[
 G_\zeta(x) = F(x\rs T_\zeta + v_{\zeta+1} + v_{\zeta+2}) - F(v_{\zeta+1}) - F(v_{\zeta+2})
 \]
 lifts $\varphi$ on $X[S\cap T_\zeta]$. Now fix $x\in X[S\cap T_\zeta]$ and $i$ with $i\equiv \zeta\pmod{3}$.
 Let $x' = x + v_{\zeta+1} + v_{\zeta+2}$, $t = x\rs [n_i,\infty) + v_{\zeta+1} + v_{\zeta+2}$ and $h = x\rs [n_i,n_{i+1}) + v_{\zeta+1} + v_{\zeta+2}$. Then,
 \[
 G_\zeta(x) - G_\zeta(x\rs [n_i,\infty)) = F(x') - F(t)
 \]
 and
 \[
 G_\zeta(x\rs [n_i,\infty)) - G_\zeta(x\rs [n_i,n_{i+1})) = F(t) - F(h).
 \]
 Since $x',t\supseteq u_{i-2}$ and $x'\rs [n_{i-1},\infty) = t\rs [n_{i-1},\infty)$, by condition~\eqref{eq:stabilizers.tails} of Claim~\ref{claim:stabilizers}, we have
 \[
d_{1 - e_{k_{i-2}}}(F(x'),F(t)) \le 2^{-i+2}
 \]
 On the other hand, $t,h\supseteq u_{i+2}$ and $t\rs [0,n_{i+2}) = h\rs [0,n_{i+2})$, so by condition~\eqref{eq:stabilizers.heads} of Claim~\ref{claim:stabilizers}, we have
 \[
 d_{e_{k_{i+2}}}(F(t),F(h)) \le 2^{-i-2}.
 \]
 Since $q_i \le e_{n_{i+2}} \le e_{k_{i+2}}$ and $q_i \le 1 - e_{n_{i-1}} \le 1 - e_{k_{i-2}}$, it follows that
 \[
 \norm{q_i(G_\zeta(x) - G_\zeta(x\rs [n_i,\infty)))} \le 2^{-i+2}
 \]
 and
 \[
 \norm{q_i(G_\zeta(x\rs [n_i,\infty)) - G_\zeta(x\rs [n_i,n_{i+1})))} \le 2^{-i-2}.
 \]
 Hence $\norm{q_i(G_\zeta(x) - G_\zeta(x\rs [n_i,n_{i+1})))} \le 2^{-i+2} + 2^{-i-2}$ by the triangle inequality, and, using the fact that $\sum\set{q_i}{i\equiv\zeta\pmod{3}} = 1$, we have
 \[
 G_\zeta(x) - \sum_{i\equiv\zeta\pmod{3}} q_i G_\zeta(x\rs [n_i,n_{i+1})) = \sum_{i\equiv\zeta\pmod{3}} q_i (G_\zeta(x) - G_\zeta(x\rs [n_i,n_{i+1}))) \in A
 \]
 since the norms of the terms above are summable. A similar argument shows that
 \[
 \norm{(G_\zeta(x\rs [n_i,n_{i+1})) - G_\zeta(0))(1 - q_i)} \le 2^{-i+2} + 2^{-i-2}.
 \]
 and hence, as $G_\zeta(0)\in A$,
 \[
 \sum_{i\equiv\zeta\pmod{3}} q_i G_\zeta(x\rs [n_i,n_{i+1}))(1 - q_i) \in A.
 \]
 Finally, by combining the above we have
 \[
 G_\zeta(x) - \sum_{i\equiv\zeta\pmod{3}} q_i G_\zeta(x\rs [n_i,n_{i+1})) q_i \in A,
 \]
 thus $G_\zeta(x) - \beta_\zeta(x) \in A$, as required.
 \end{proof}

Since $T_0\cup T_1\cup T_2$ is cofinite, it follows that $\beta = \beta_0 + \beta_1 + \beta_2$ lifts $\varphi$ on $X[S]$, for every $S\in\CJ$.
 \end{proof}

\section{Consequences: Isomorphisms of reduced products I}\label{sec:cons.RedProd1}
Keeping in mind the definition of the Metric Approximation Property (Definition~\ref{defin:MAP}) and that of asymptotically algebraic isomorphism of reduced products (Definition~\ref{defin:trivialredprod}), we state Theorem~\ref{theoi:redprod}, the main result of this section. Later in this section, we explicit some of its consequences.
\begin{theorem}\label{thm:trivialredprod}
 Assume $\OCA+\MA_{\aleph_1}$. Let $A_n$ and $B_n$, for $n\in\NN$, be separable unital $\Cstar$-algebras with no nontrivial central projections. Suppose that each $A_n$ has the metric approximation property. Then all isomorphisms $\iso\colon \prod A_n/\bigoplus A_n\to\prod B_n/\bigoplus B_n$ are asymptotically algebraic.
\end{theorem}

\subsection{Proof of Theorem~\ref{thm:trivialredprod}}

Throughout this section $A_n$ and $B_n$ are as in the hypotheses of Theorem~\ref{thm:trivialredprod}. As in \S\ref{sec:liftstatements} and \S\ref{sec:lemma_proofs}, if $S\subseteq \NN$, we let $A[S]$ (and $B[S]$) be the subalgebras of $\prod A_n$ (and  of $\prod B_n$) of those elements whose support is contained in $S$. $p_S^A$ and $p_S^B$ denote the identity of $A[S]$ and of $B[S]$ respectively, and $\pi_A$ and $\pi_B$ are the canonical quotient maps $\pi_A\colon\prod A_n\to\prod A_n/\bigoplus A_n$ and $\pi_B\colon\prod B_n\to\prod B_n/\bigoplus B_n$. $\iso$  denotes a fixed isomorphism
\[
\iso\colon\prod A_n/\bigoplus A_n \to\prod B_n/\bigoplus B_n.
\]
The first goal is to simplify $\iso$.

 \begin{proposition}\label{prop:centralproj}
 Let $A$, and $A_n$, for $n\in\NN$, be $\Cstar$-algebras. Then
 \begin{enumerate}
 \item\label{proj1} If $p$ is a noncentral projection in $A$, then there is a contraction $a\in A$ with $\norm{pa-ap}>1/2$;
 \item\label{proj2} if each $A_n$ is a unital $\Cstar$-algebra with no nontrivial central projections, the only central projections in $\prod A_n/\bigoplus A_n$ are of the form $\pi_A(p^A_X)$, for $X\subseteq\NN$.
 \end{enumerate}
\end{proposition}

\begin{proof}
\ref{proj1}: Fix a noncentral projection $p\in A$. Since $p$ is noncentral, we can find an irreducible representation $\sigma \colon A \to \mathcal B(H)$ with $\sigma(p)\notin\{0,1\}$. Choose unit vectors $\xi$ and $\eta$ in $H$ such that $\sigma(p)\xi = \xi$ and $\sigma(p)\eta = 0$. By Kadison's Transitivity Theorem (\cite[II.6.1.13]{Blackadar.OA}) and the irreducibility of $\sigma$, we may find a contraction $a\in A$ with $\norm{\sigma(a)\eta - \xi}<\frac{1}{4}$. Then $\norm{\sigma(pa)\eta}\geq\frac{3}{4}$ and $\norm{\sigma(ap)\eta}=0$, hence then $\norm{ap - pa} \geq \norm{\sigma(ap-pa)}> 1/2$. 

\ref{proj2}: follows from the fact that the only projections in $\prod A_n / \bigoplus A_n$ are those of the form $\pi_A(p)$ for a projection $p\in\prod A_n$. This is because whenever one chooses $\varepsilon>0$ and a $\Cstar$-algebra $A$, if $q\in A$ is an almost projection (that is, $\norm{q-q^2}+\norm{q-q^*}<\varepsilon$), then one can find a projection $p\in A$ with $\norm{p-q}<4\varepsilon$ (e.g., \cite[Example 3.2.7]{bourbaki}). By \ref{proj1}, the same holds for almost central projections: if $q\in A$ is an almost central projection, there is a central projection $p\in A$ which is close to $q$. Therefore central projections in $\prod A_n / \bigoplus A_n$ lift to central projections in $\prod A_n$.
\end{proof}

By Proposition~\ref{prop:centralproj}, an isomorphism $\iso\colon \prod A_n / \bigoplus A_n \to \prod B_n / \bigoplus B_n$ induces an automorphism of the Boolean algebra $\SP(\NN)/\Fin$ by associating to $X$ the unique (modulo finite) $Y$ such that $\Lambda(\pi_A(p_X^A))=\pi_B(p_Y^B)$. In the presence of $\OCA+\MA_{\aleph_1}$, by the main results of \cite{Velickovic.OCAA} we may find an almost permutation $g$ such that for all $X\subseteq\NN$, we have $\Lambda(\pi_A(p_X^A)) = \pi_B(p_{g[X]}^B)$. Setting $A'_n=A_{g(n)}$ if $g(n)$ exists, and $A'_n=\mathbb C$ if not, we have that $g^{-1}$ induces an isomorphism between $\prod A'_n/\bigoplus A'_n$ and $\prod A_n/\bigoplus A_n$. We call such isomorphism $\iso_g$. Notice that $\iso_g$ is asymptotically algebraic, and $\iso$ is asymptotically algebraic if and only if $\iso\circ \iso_g$ is. Substituting $\iso$ with $\iso\circ \iso_g$, we may therefore assume that $\iso(\pi_A(p_X^A)) = \pi_B(p_X^B)$ for every $X\subseteq\NN$.

We fix:
\begin{itemize}
\item a countable dense subset of the unit ball of $A_n$, $\{y_i^n\}$ with $y_0^n=1_{A_n}$ for all $n$;
\item using the Metric Approximation Property of $A$ (Definition~\ref{defin:MAP}), finite dimensional operator systems $E_{n,m}$ and unital linear contractive maps $\varphi_{n,m},\psi_{n,m}$ with $\varphi_{n,m}\colon A_{n}\to E_{n,m}$, $\psi_{n,m}\colon E_{n,m}\to A_n$ with the property that for all $i\leq m$, $\norm{\psi_{n,m} \circ \varphi_{n,m} (y_i)- y_i}<2^{m}$;
\item for each pair of natural numbers $n$ and $m$, a finite $2^{-n-m}$-dense subsets $X_{n,m}\subseteq (E_{n,m})_{\leq 1}$ such that $0$ and $1$ are in $X_{n,m}$. We well order each $X_{n,m}$ so that $0$ is the minimum element of $X_{n,m}$, and define a function $\rho_{n,m}\colon (E_{n,m})_{\leq 1}\to X_{n,m}$ with the property that $\norm{x-\rho_{n,m}(x)}\leq 2^{-n-m}$ for all $x\in (E_{n,m})_{\leq 1}$. Let $\rho=\prod \rho_{n,m}$.
\end{itemize}
With these objects in mind, recall (Definition~\ref{defin:skeletal}) that a map $F\colon\prod E_{n,m}\to\prod B_n$ is skeletal if $F=F\circ\rho$ on $\prod(E_{n,m})_{\leq 1}$.

\begin{definition}\label{defin:skel2}
$\Skel$ is the set of skeletal maps from $\prod E_{n,m} \to\prod B_n$. We identify $\alpha\in\Skel$ with the countable sequence of values $\alpha_{n,m}(x)$, where $x$ ranges over $X_{n,m}$ and $n$ and $m$ over $\NN$. With this identification, $\Skel$ is a topological space with the topology obtained by viewing as a subset of $\prod_{n,m}(\bigoplus_i B_i)$.
\end{definition}

For a given function $f\colon \NN\to\NN$, a skeletal map $\alpha \colon \prod E_{n,f(n)}\to \prod B_n$ can be identified with an element of $\Skel$ by filling in with the value $0$ on any $x\in X_{n,m}$ with $m\neq f(n)$. Since each $X_{n,m}$ is finite and each $B_n$ is separable, with the above topology, $\Skel$ is a separable metrizable space.



For each $f\in\NN^{\NN\uparrow}$ let 
\[
\Phi_f=\prod\varphi_{n,f(n)}\text{ and }\Psi_f=\prod\psi_{n,f(n)}.
\]
Since each of the maps $\varphi_{n,m}$ and $\psi_{n,m}$ is contractive, $\Phi_f$ and $\Psi_f$ induce unital linear contractions
\begin{eqnarray*}
\tilde\Phi_f\colon \prod A_n/\bigoplus A_n\to\prod E_{n,f(n)}/\bigoplus E_{n,f(n)},\\
\tilde\Psi_f\colon \prod E_{n,f(n)}/\bigoplus E_{n,f(n)}\to\prod A_n/\bigoplus A_n.
\end{eqnarray*}
Given an isomorphism $\iso\colon\prod A_n/\bigoplus A_n\to\prod B_n/\bigoplus B_n$, we define
$\iso_f=\iso\circ\tilde\Psi_f$. Notice that 
 \[
 \iso_f\colon\prod E_{n,f(n)}/\bigoplus E_{n,f(n)}\to\prod B_n/\bigoplus B_n.
\]

\begin{proposition}\label{prop:liftoneverything}
Let $f\in\NN^{\NN\uparrow}$ and $\alpha^f\colon\prod E_{n,f(n)}\to \prod B_n$ be an asymptotically additive map which is a lift of $\iso_f$ on a dense and nonmeager ideal $\CI_f$. For $x\in \prod E_{n,f(n)}$, define 
\[
\tilde\alpha^f_n(x) = p_{\{n\}}^B\alpha^f_n(x)p^B_{\{n\}}
\] and let $\tilde\alpha^f = \sum\tilde\alpha^f_n$. Then $\tilde\alpha^f\circ\Phi_f$ is a lift of $\iso$ on
\[
 \set{x\in\prod A_n}{\Psi_f\circ\Phi_f(x)-x \in \bigoplus A_n}.
\]
\end{proposition}
\begin{proof}
First, note that since $\alpha^f=\sum \alpha^f_n$, then $\tilde\alpha^f$ is well defined on $\prod E_{n,f(n)}$.

\begin{claim}\label{claim:coordinatewise}
If $\Psi_f\circ\Phi_f(x)-x\in\bigoplus A_n$ then $\alpha^f(\Phi_f(x))-\tilde\alpha^f(\Phi_f(x))\in\bigoplus B_n$.
\end{claim}

\begin{proof}
Suppose not. Then there are $x$ such that $\Psi_f\circ\Phi_f(x)-x\in\bigoplus A_n$, $\varepsilon>0$, and a sequence $n_k$ such that $\norm{(\alpha^f(\Phi_f(x))-\tilde\alpha^f(\Phi_f(x)))_{n_k}}>\varepsilon$. By the definition of asymptotically additive map we can refine $\{n_k\}_k$ so that for all $k$ and $m$, if $p^B_{\{n_k\}}\alpha^f_m\neq 0$ then $p^B_{\{n_l\}}\alpha^f_m=0$ for all $l\neq k$. Again by refining $\{n_k\}_k$ we can fix $m_i$ with the property that $0=m_0<m_1<\cdots$ and if $m_i\leq j<m_{i+1}$ then $\alpha^f_jp^B_{\{n_k\}}=0$ for all $k\neq i$. By nonmeagerness of $\mathcal I$ and Proposition~\ref{prop:JT2}, there is then an infinite $L$ such that $\bigcup_{i\in L}[m_i,m_{i+1})\in\CI_f$. Let $y_n=x_n$ if $n\in [m_i,m_{i+1})$ for $i\in L$, and $0$ otherwise. Note that for all $i$ (up to finite many) we have that $n_i\in[m_i,m_{i+1})$.

Let $z=y[\{n_i\}]$. Since $\alpha^f(\Phi_f(y))-\tilde\alpha^f(\Phi_f(y))\in\bigoplus B_n$, as $\{n_i\}\in\CI_f$, we have that $w=y-z$ is such that
\[
 \pi^B(\alpha^f(\Phi_f(z))p_{\{n_i\mid i\in L\}}^B)\neq 0,
\]
a contradiction to the fact that $\iso(p^A_{\{n_i\mid i\in L\}})=p^B_{\{n_i\mid i\in L\}}$ and the choice of $\CI_f$.
\end{proof}
Fix any function $\rho\colon\prod A_n\to\prod B_n$ lifting $\iso$ and $x$ such that $\Psi_f\circ\Phi_f(x)-x\in \bigoplus A_n$. Consider
\[
 \CI_x=\{X\subseteq\NN\mid p_X^B(\alpha^f(\Phi_f(x))-\rho(x))\in \bigoplus B_n\}.
\]
This ideal contains $\CI_f$, so is nonmeager. Also, since $\rho(x), \alpha^f(\Phi_f(x))$ are fixed elements of $\prod B_n$, and $\bigoplus B_n$ is Borel in $\prod B_n$, $\CI_x$ is Borel. Since all proper, Borel, dense ideals on $\NN$ are meager (Proposition~\ref{prop:JT2}), $\CI_x=\SP(\NN)$, that is
\[
 (\alpha^f(\Phi_f(x))-\rho(x))=\alpha^f(\Phi_f(x))-\rho(x)\in \bigoplus B_n.\qedhere
\]
\end{proof}
Keeping in mind the definition of $\Skel$ (\ref{defin:skel2}), define
\[
 D_f=\{a\in\prod A_n\mid\forall n\forall m\geq f(n)(\norm{\psi_{n,m}(\varphi_{n,m}(a_n))-a_n}<2^{-n})\},
\]
and $\CX\subseteq\NN^\NN\times\Skel$ by
\[
 \CX=\{(f,\alpha)\mid \alpha\text{ is a skeletal lift of }\iso_f \text{ on }\Phi_f(D_f)\text{, and }\alpha_n=p_{\{n\}}^B\alpha_np_{\{n\}}^B\}.
\]
If $\OCA+\MA_{\aleph_1}$ holds, then by Theorem \ref{thm:lifting}, Proposition~\ref{prop:liftoneverything} and the fact that the asymptotically additive maps in Lemma~\ref{lemma:C->asymptotic} can be chosen to be skeletal (see Lemma~\ref{lemma:skeletal-replacement}), for every $f\in\NN^{\NN}$ we may find an $\alpha$ such that $(f,\alpha)\in \CX$.

\begin{proposition}\label{prop:afteragivenn}
Let $(f,\alpha), (g,\beta)\in \mathcal X$ and $\varepsilon>0$. Then there is $n$ such that for all $m\geq n$ and $x\in A[[n,m]]$ with $\norm{x}\leq 1$, if $x\in D_f\cap D_g$ we have
\[
\norm{\alpha(\Phi_{f}(x))-\beta(\Phi_{g}(x))}\leq\varepsilon.
\]
\end{proposition}

\begin{proof}
Suppose otherwise and fix $\varepsilon>0$, $n_i\leq m_i\in\NN$, with $n_i\to\infty$, and contractions $x_i\in A[[n_i,m_i]]\cap D_f\cap D_g$ with the property that for all $i$
\[
\norm{\alpha(\Phi_f(x_i))-\beta(\Phi_g(x_i))}>\varepsilon.
\]
Since the norm in $B[[n_i,m_i]]$ is given by the maximum norm over its individual coordinates, we may assume that $m_i=n_i$. Let $x=\sum x_i$. Then
\[
 \norm{\pi_B(\alpha(x)-\beta(x))}>\varepsilon,
\]
This contradicts the fact that $\alpha$ and $\beta$ both lift $\iso$ on $D_f\cap D_g$.
\end{proof}
The following bit of the proof takes inspiration from the fact that the ideal $\CI$ (from \S\ref{sec:liftstatements} and \S\ref{sec:lemma_proofs}) is a P-ideal. It is an adaptation of \cite[Lemma 3.6.8]{Farah.AQ}.

For every $\varepsilon>0$, we define a colouring $[\CX]^2=K_0^\varepsilon\cup K_1^\varepsilon$ by setting $\{(f,\alpha),(g,\beta)\}\in K_0^\varepsilon$ if and only if
\[
 \exists n\exists x\in A_n\cap D_f\cap D_g, \norm{x}\leq1, \norm{\alpha(\Phi_f(x))-\beta(\Phi_g(x))}>\varepsilon.
\]
The following proposition follows easily from the definitions and the topology of $\Skel$ (see after Definition~\ref{defin:skel2}).
\begin{lemma}\label{lemma:openinsometopology}
 $K_0^\varepsilon$ is an open subset of $[\CX]^2$ when $\CX$ is given the separable metric topology obtained from its inclusion in $\NN^\NN\times\Skel$.\qed
\end{lemma}

\begin{lemma}\label{lem:nohomogredprod}
There is no uncountable $K_0^\varepsilon$-homogeneous set for any $\varepsilon >0$.
\end{lemma}

\begin{proof}
 Suppose otherwise and let $\varepsilon>0$ and $\mathcal Y$ be a $K_0^\varepsilon$-homogeneous set of size $\aleph_1$. We refine $\mathcal Y$ to an uncountable subset several times. To avoid excessive notation, we  keep the name $\mathcal Y$ for each refinement.

Since we have $\OCA$, $\mathfrak b>\omega_1$ (\S\ref{subsec:ForcingAxiom}), hence we can find $\hat f$ such that for all $(f,\alpha)\in \mathcal Y$ we have that $f<^* \hat f$. Also, we may assume that there is a unique $\bar n$ such that for all $(f,\alpha)\in \mathcal Y$ and $m\geq\bar n$ we have $f(m)<\hat f(m)$. By refining $\mathcal Y$ again we can also assume that if $(f,\alpha)$ and $(g,\beta)$ are in $\mathcal Y$ then $f$ and $g$ agree up to $\bar n$. This is possible because the set $\{f\restriction \bar n\mid (f,\alpha)\in\mathcal Y\}$ is countable. By increasing $\hat f(i)$ for all $i\leq \bar n$, that $f<\hat f$ for all $(f,\alpha)\in \mathcal Y$. In particular $D_f\subseteq D_{\hat f}$.

 Since we are assuming $\OCA$ and $\MA_{\aleph_1}$, we can find a skeletal map $\gamma$ such that $(\hat f,\gamma)\in X$. By Proposition~\ref{prop:afteragivenn} for all $f$ with $(f,\alpha)\in \mathcal Y$ we can find $n_f$ such that if for all $m$ and contractions $x\in A[[n_f,m]]$ we have that
 \[
 \norm{\alpha(\Phi_{f}(x))-\gamma(\Phi_{\hat f}(x))}\leq\varepsilon/2.
 \]
 By refining $\mathcal Y$ we can assume that $\bar N=n_f$ is the same for all $(f,\alpha)\in \mathcal Y$. We refine $\mathcal Y$ once more asking that for all $(f,\alpha)$ and $(g,\beta)$ in $\mathcal Y$ and for all $k\le \bar{N}$, we have $f(k)=g(k)$ and
 \[
 \norm{\alpha_k-\beta_k}<\varepsilon/2
 \]
 (Recall that the space of all skeletal maps $\prod_{n\leq\bar N} E_{k,f(k)}\to\prod_{n\leq\bar N} B_n$ is separable in the operator norm topology.) 
 
 This is the last refinement we need. Pick $(f,\alpha)$ and $(g,\beta)$ in $\mathcal Y$ and $n$ and $x\in D_f\cap D_g$ with $x\in A_n$ and witnessing that $\{(f,\alpha),(g,\beta)\}\in K^\varepsilon_0$. Let $r=\norm{\alpha_n(\Phi_f(x))-\beta_n(\Phi_g(x))}$. If $n\leq\bar N$ then, as $f(n)=g(n)$, we have that $r<\varepsilon/2$, since $\Phi_f(x)=\Phi_g(x)$ and by the last refinement of $\mathcal Y$. On the other hand if $n\geq\bar N$ we have that
 \[
 r\leq \norm{\alpha_n(\Phi_f(x))-\gamma_n(\Phi_{\hat f}(x))}+\norm{\gamma_n(\Phi_{\hat f}(x))- \beta_n(\Phi_g(x))}<\varepsilon.
 \]
 This contradicts the fact that $r>\e$.
 \end{proof}
By the assumption of $\OCA$ (which implies $\OCA_\infty$), for each
$\varepsilon_k=2^{-k}$ we may find a cover $\CX=\bigcup_n \CY_{n,k}$, where each
$\CY_{n,k}$ is a $K_1^{\varepsilon_k}$-homogeneous set. Since $\leq^*$ is a
$\sigma$-directed order, as in \cite[Lemma 2.2.2]{Farah.AQ} we can find sets
$\CY_k$, for $k\in\NN$, and a sequence of natural numbers $\{n_k\}$, such that for all $k\in\NN$,  $\CY_k\supseteq\CY_{k+1}$, $[\CY_k]^2\subseteq K^{\varepsilon_k}_1$ and $P_k= \set{f}{\exists \alpha\; (f,\alpha)\in \CY_k}$ is cofinal in the order $<^{n_k}$, where
\[
 f<^{n_k} g\iff \forall m\geq n_k (f(m)<g(m)).
\]
We are in position to define the maps $\varphi_n\colon A_n\to B_n$ which witness that $\iso$ is asymptotically algebraic. For each $n < n_0$, define $\varphi_n = 0$. For each $n\ge n_0$, let $k$ be such that $n_k \le n < n_{k+1}$, and choose a sequence $(f^{i,n},\alpha^{i,n})\in\CY_k$ such that $f^{i,n}(n)\to \infty$ as $i\to\infty$. We define $\varphi_n\colon A_n\to A_n$ in stages. First, let
\[
 \varphi_n(y^n_m)=\alpha^{i,n}_n(\Phi_{f^{i,n}}(y^n_m))
\]
where $i=\min\{r\mid f^{r,n}(n)>m\}$. If $x\notin \set{y^n_m}{m\in\NN}$ and $\norm{x}\leq 1$, let
\[
 \varphi_n(x)=\varphi_n(y^n_m)
\]
where $m=\min\{r\mid \norm{x-y_r^n}<2^{-k}\}$. Finally, if $\norm{x} > 1$, let $\varphi_n(x)=\norm{x}\varphi_n(x/\norm{x})$.

\begin{claim}
 The map $\varphi = \prod \varphi_n$ lifts $\iso$ on $\prod A_n$.
\end{claim}

\begin{proof}
 Let $x\in \prod A_n$ be given. We may assume that $\norm{x} \le 1$. Moreover, we may assume that $x_n = y^n_{h(n)}$ for some $h\in\NN^\NN$, so that $x\in D_h$. Fix $k\in\NN$, and choose $(g,\alpha)\in \CY_k$ such that $h\le^* g$. By modifying $x$ on finitely many coordinates we may assume that $x\in D_g$. Thus, $\pi_B(\alpha(\Phi_g(x))) = \iso(\pi_A(x))$. Now note that for all $n\ge n_k$ and $i\in\NN$ we have
 \[
 \{(g,\alpha),(f^{i,n},\alpha^{i,n})\}\in K_1^{\varepsilon_k}
 \]
 In particular, if $i$ is the least natural number such that $f^{i,n}(n) > h(n)$, then we have
 \[
 \norm{\alpha_n(\Phi_g(x_n)) - \varphi_n(x_n)} \le \varepsilon_k,
 \]
 hence for all $k\in\NN$ we have that
 \[
\norm{ \pi_B(\alpha(\Phi_g(x))-\varphi(x))}=\limsup_n\norm{\alpha_n(\Phi_g(x_n))-\varphi_n(x_n)}\leq\e_k, 
 \]
 that is, 
 \[
 \pi_B(\varphi(x)) = \pi_B(\alpha(\Phi_g(x))) = \iso(\pi_A(x)).\qedhere
 \]
\end{proof}
This completes the proof of Theorem~\ref{thm:trivialredprod}.
\section{Consequences: Isomorphisms of reduced products II}\label{sec:cons.RedProd2}

If $X$ and $Y$ are locally compact topological spaces, a homeomorphism between $\beta X\setminus X$ and $\beta Y\setminus Y$ is \emph{trivial} if it is induced by a homeomorphism between cocompact subspaces of $X$ and $Y$. Suppose that $X=\sqcup X_i$ and $Y=\sqcup Y_i$ where each $X_i$ and each $Y_i$ is a compact space. If $\varphi$ is a homeomorphism between cocompact subspaces of $X$ and $Y$, then the graph of $\varphi$ is necessarily the union of the graphs of a sequence of homeomorphisms $\varphi_i\colon X_i\to Y_{g(i)}$, where $g$ is an almost permutation of $\NN$. (One obtains $\varphi_i$ as $\varphi\restriction X_i$.)

This definition was given in \cite{Farah-Shelah.RCQ} (see also \cite{Farah.AQ} and \cite{Farah-McKenney.ZD}), and coincides with the definition of algebraically trivial isomorphism between $C(\beta X\setminus X)$ and $C(\beta Y\setminus Y)$ as given in \cite{V.Rigidity} (see \cite[Proposition 2.7]{V.Rigidity}).

A fundamental question is the following: suppose that $A_n$ and $B_n$ are unital $\Cstar$-algebras with no nontrivial central projections, and that $\prod A_n/\bigoplus A_n$ and $\prod B_n/\bigoplus B_n$ are isomorphic. What can we say about the relations between the sequences $(A_n)_n$ and $(B_n)_n$? The following result of Ghasemi, proved in \cite{Ghasemi.FFV}, shows that not much can be said under $\CH$, as in this case it is possible to construct isomorphic reduced products from sequences which do not have anything to do with each other. Its proof uses compactness of the space of first order theories (in continuous model theory), a continuous version of the Fefferman-Vaught theorem, and countable saturation of reduced products.
\begin{theorem}\label{saeed}
Assume $\CH$, and let $A_n$ be separable unital $\Cstar$-algebras. Then there is a subsequence $\{n_k\}_{k\in\NN}$ such that if $S$ and $T$ are infinite subsets of $\NN$ then
\[
\prod_{k\in S} A_{n_k}/\bigoplus_{k\in S}A_{n_k}\cong\prod_{k\in T} A_{n_k}/\bigoplus_{k\in T}A_{n_k}.\qed
\]
\end{theorem}
Keeping in mind our results from \S\ref{sec:apmaps} (in particular Theorems~\ref{thm:apmap-known} and \ref{thm:apmapstrivial} and Corollary~\ref{cor:KirchUCT}) and the definition of algebraically trivial isomorphisms of reduced products (Definition~\ref{defin:trivialredprod}), we state the following consequences of Theorem~\ref{thm:trivialredprod}. We provide an unified proof for all of the corollaries below. (Corollary~\ref{cor:trivialredprod1} is due to Ghasemi, \cite{Ghasemi.FDD}.)
\begin{corollary}\label{cor:trivialredprod1}
The following statement is independent of $\ZFC$: given two sequences of positive natural numbers $(n_i)_i$ and $(m_i)_i$ then all isomorphisms between $\prod M_{n_i}(\ce)/\bigoplus M_{n_i}(\ce)$ and $\prod M_{m_i}(\ce)/\bigoplus M_{m_i}(\ce)$ are algebraically trivial.
\end{corollary}
\begin{corollary}\label{cor:trivialredprod2}
The following statement is independent of $\ZFC$: if $A_n$ are separable unital AF algebras with no nontrivial central projections and $B_n$ are unital separable $\Cstar$-algebras with no nontrivial central projections then $\prod A_n/\bigoplus A_n\cong\prod B_n/\bigoplus B_n$ if and only if $A_n\cong B_{g(n)}$ for some almost permutation $g$ of $\NN$ and all but finitely many $n$.
\end{corollary}
\begin{corollary}\label{cor:trivialredprod3}
The following statement is independent of $\ZFC$: if $X_i$ and $Y_i$ are compact connected metrizable spaces, $X=\sqcup X_i$ and $Y=\sqcup Y_i$, then all homeomorphisms between $\beta X\setminus X$ and $\beta Y\setminus Y$ are trivial.
\end{corollary}
\begin{corollary}\label{cor:trivialredprod4}
The following statement is independent of $\ZFC$: if $A_n$ and $B_n$ are UCT Kirchberg algebras then $\prod A_n/\bigoplus A_n\cong\prod B_n/\bigoplus B_n$ if and only if $A_n\cong B_{g(n)}$ for some almost permutation $g$ of $\NN$ and all but finitely many $n$.
\end{corollary}

\begin{proof}[Proof of Corollaries~\ref{cor:trivialredprod1}--\ref{cor:trivialredprod4}]
That all the statements are true under $\OCA+\MA_{\aleph_1}$ is Theorem~\ref{thm:trivialredprod} together with Theorem~\ref{thm:apmapstrivial}, Theorem~\ref{thm:apmap-known} and Corollary~\ref{cor:KirchUCT}. The first statement was shown to be consistent with $\ZFC$ in \cite{Ghasemi.FDD}, and to hold $\OCA+\MA_{\aleph_1}$ in \cite{McKenney.UHF}.

Theorem~\ref{saeed} shows that under $\CH$ all such statements are false, as we can apply it, for each class as above, to an infinite sequence of pairwise nonisomorphic objects.
\end{proof}

\section{Consequences: the Coskey-Farah conjecture}\label{sec:cons.Borel}
Recall that, by Definition~\ref{defin:Boreliso}, if $A$ and $B$ are separable $\Cstar$-algebras, an isomorphism $\iso\colon\SQ(A)\to\SQ(B)$ is topologically trivial if its graph
\[
 \Gamma_\iso=\{(a,b)\in \SM(A)\times\SM(B)\mid \norm{a},\norm{b}\leq 1\wedge  \iso(\pi_A(a))=\pi_B(b)\}
\]
is Borel in the product of the strict topologies. The following is a generalization of Conjecture~\ref{conj:theconj} \eqref{theconj:PFA}.

\begin{conjecture}
 Let $A$ and $B$ be separable $\Cstar$-algebras. Then $\PFA$ implies that all isomorphisms between $\SQ(A)$ and $\SQ(B)$ are topologically trivial.
\end{conjecture}
The main result proved in this section is the following.

\begin{theorem}\label{thm:Borel}
 Assume $\OCA$ and $\MA_{\aleph_1}$. Let $A$ and $B$ be separable $\Cstar$-algebras where $A$ has the metric approximation property and both $A$ and $B$ have an increasing approximate identity of projections. Then every isomorphism from $\SQ(A)$ to $\SQ(B)$ is topologically trivial.
\end{theorem}

This proof is similar to the one of Theorem~\ref{thm:trivialredprod}, being fundamentally based on applying Theorem~\ref{thm:lifting}; however, the remainder of this proof is much more technical. The reason for the increase in difficulty is that, in case $A=\bigoplus A_n$, one can find a sequence $\{p_{n}\}\subseteq A$ such that  each $a\in\SM(A)$ can be written as $a=\sum p_{n}ap_{n}$, which allows for a stratification of $\SQ(A)$ into reduced products parametrized by functions $f \colon \NN\to\NN$ controlling the degree to which $p_{n} a p_{n}$ is approximated. (This scenario can be reproduced if, in technical terms, $A$ has an approximate identity which is quasi-central for $\SM(A)$, see \cite[1.9]{Farah.Book}). In the general case, our reduced products have to be parametrized by two interleaved sequences of intervals $I_n^i\subseteq\NN$, for $n\in\NN$ and $i=0,1$,  \`a la Lemma~\ref{lem:stratification}, in addition to a function $f \colon \NN\to\NN$ controlling the degree of approximation. (See Lemma~\ref{lem:summing} for details.)


\subsection{Notation}\label{subsec:notlast}
To avoid excessive notation, we prove Theorem~\ref{thm:Borel} in the case $A=B$. As usual, we use $\OCA_\infty$ instead of $\OCA$.

$A$ denotes a separable $\Cstar$-algebra with the metric approximation property (Definition~\ref{defin:MAP}) and an increasing approximate identity of projections $(p_n)$.  $\pi$ denotes the quotient map $\pi\colon \SM(A)\to\SQ(A)$. Given $S\subseteq\NN$, define $p_S = \sum_{n\in S}(p_n - p_{n-1})$. (We set $p_{-1} = 0$.). If $\mathbb I=\{I_n\}$ is a partition of $\NN$ into consecutive intervals and $X\subseteq\NN$ we let
\[
p_X^{\mathbb I}=\sum_{n\in X} p_{I_n}.
\]
Since the $p_{I_n}$ are mutually orthogonal and uniformly bounded in norm elements of $A$, the sequence of partial sum $(\sum_{n\in X, n\leq m} p_{I_n})_m$ converges strictly in $\SM(A)$ to an element which we call $\sum_{n\in X}p_{I_n}$.

Let $(Y_n)$ be an increasing sequence of finite subsets of $A_{\leq 1}$ whose union is dense in $A_{\leq 1}$, and with the additional property that $p_S\in Y_n$ for all $S\subseteq n$. Since $A$ has the metric approximation property, we can find a finite-dimensional operator system $E_n$ and contractive linear maps $\varphi_n\colon A\to E_n$, $\psi_n\colon E_n\to A$ such that
\[
\norm{\psi_n(\varphi_n(x)) - x} \le 2^{-n}\norm{x}\text{ for all }x\in Y_n.
\]
For each $n$, fix a finite $X_n$, with a distinguished linear order, which is $2^{-n}$-dense in the unit ball of $E_n$. As before, we require that $0$ is the minimum of $X_n$. When referring to skeletal maps (Definition~\ref{defin:skeletal}) and to the space $\Skel$ (Definition~\ref{defin:skel2}), we always take the spaces $E_n$ and the sets $X_n$ as fixed. In this case, $\Skel$ is identified with a subset of $\SM(A)^{\NN}$.
\subsection{The proof}
Let $\mathrm{Part}_\NN$ be the set of all partitions of $\NN$ into consecutive finite intervals, as in \cite[\S9.7]{Farah.Book}. We turn $\mathrm{Part}_\NN$ into a poset by setting 
\[
\mathbb I\leq_1\mathbb J\text{ if and only if }\exists i_0\forall i\geq i_0 i\exists j (I_i\cup I_{i+1}\subseteq J_j\cup J_{j+1}).
\]
 (In \cite{Farah.Book} the order on $\mathrm{Part}_\NN$ is denoted by $\leq^*$. As we already have $\leq^*$ below, we decided to denote differently the two orders to avoid confusion.)

For $f\in\NN^\NN$ and $\mathbb I\in\mathrm{Part}_\NN$, we define
\[
 \Phi_{f,\mathbb I} \colon \SM(A)\to \prod_n E_{f(\max I_n)}\text{ and }\Psi_{f,\mathbb I} \colon \prod_n E_{f(\max I_n)}\to \SM(A)
 \]
 by
 \begin{align*}
 \Phi_{f,\mathbb I}(t)(n) & = \varphi_{f(\max I_n)}(p_{I_n}t p_{I_n}), \text{ and } \\
 \Psi_{f,\mathbb I}((x_n)) & = \sum)n p_{I_n} \psi_{f(\max I_n)}(x_n) p_{I_n} \text{ for } (x_n)\in\prod E_{f(\max I_n)}.
 \end{align*}
The infinite sum in the definition of $ \Psi_{f,\mathbb I}((x_n))$ is the strict limit of the partial sums. For each $(x_n)\in \prod E_{f(\max I_n)}$ this limit exists in $\SM(A)$, since the $p_{I_n}$'s are orthogonal and the $\psi_{f(\max I_n)}$ are uniformly bounded in norm as $n\to\infty$. Also, $\Phi_{f,\mathbb I}$ and $\Psi_{f,\mathbb I}$ are contractive and linear and moreover
 \[
 \Phi_{f,\mathbb I}(A)\subseteq\bigoplus_n E_{f(\max I_n)}\text{ and }\Psi_{f,\mathbb I}(\bigoplus_n E_{f(\max I_n)})\subseteq A.
\]
 For each $\mathbb I\in\mathrm{Part}_\NN$, set 
 \[
 \mathbb I^0 = \seq{I_{2n}\cup I_{2n+1}}{n\in\NN}\text{ and }\mathbb I^1 = \seq{I_{2n+1}\cup I_{2n+2}}{n\in\NN}
 \]
where $I_{-1}=\emptyset$.
 \begin{lemma}\label{lem:summing}
 Let $t\in \SM(A)$. Then there are $f\in\NN^\NN$, $\mathbb I\in\mathrm{Part}_\NN$, and $x^i\in\prod E_{f(\max I^i_n)}$, for $i=0,1$, such that
 \[
 t - (\Psi_{f,\mathbb I^0}(x^0) + \Psi_{f,\mathbb I^1}(x^1))\in A
 \]
 \end{lemma}

 \begin{proof}
 By Lemma~\ref{lem:stratification} we may find $\mathbb{I}\in\mathrm{Part}_\NN$ and $t_0$ and $t_1$ in $\SM(A)$ such that $t_i$ commutes with $p_{I^i_n}$ for each $n\in\NN$ and $i = 0,1$, and $\pi(t) = \pi(t_0 + t_1)$. Now for each $n$, choose $f(n)\in\NN$ large enough that for each $i = 0,1$,
 \[
 \norm{p_{I_n^i} t_i p_{I_n^i} - \psi_{f(n)}(\varphi_{f(n)}(p_{I_n^i} t_i p_{I_n^i}))} < 2^{-n}
 \]
 It follows that
 \[
 t_i - \Psi_{f,\mathbb I^i}(\Phi_{f,\mathbb I}(t_i))\in A
 \]
 Setting $x^i = \Phi_{f,\mathbb I}(t_i)$, we are done.
 \end{proof}
Let $D[\mathbb I]=\{x\in \SM(A)\mid x=\sum p_{I_n}xp_{I_n}\}$ and, for $f\in\NN^\NN$,
\[
D_f[\mathbb I]= \{x\in D[\mathbb I]\mid \forall n\forall m\geq f(\max I_n) (\norm{p_{I_n}xp_{I_n}-\psi_m(\varphi_m(p_{I_n}xp_{I_n}))}<2^{-n})\}.
\]

\begin{proposition}\label{prop:mainproperties}
Let $f$ and $g$ in $\NN^\NN$ and $\mathbb I$ and $\mathbb J$ in $\mathrm{Part}_\NN$. Then
\begin{enumerate}
\item\label{str1} if $f\leq^*g$ then $\pi(D_f[\mathbb I])\subseteq \pi(D_g[\mathbb I])$;
\item\label{str2} if $\mathbb I\leq_1 \mathbb J$ and $f$ is strictly increasing, then $f\leq_* g$ implies that
\[
\pi(D_f[\mathbb I^0]+D_f[\mathbb I^1])
\subseteq\pi(D_g[\mathbb J^0]+D_g[\mathbb J^1]);
\]
\item\label{str3} for every $t\in \SM(A)$ there are $f$, $\mathbb I$, $x_0$ and $x_1$ such that $t-x_0-x_1\in A$, $x_0\in D_f[\mathbb I^0]$, $x_1\in D_f[\mathbb I^1]$. Moreover, if $t$ is positive, $x_0$ and $x_1$ may be chosen to be positive.
\end{enumerate}
\end{proposition}
\begin{proof}
\eqref{str1}: Let $f$ and $g$ be such that $f\leq^* g$. Let $n_0$ be such that if $n\geq n_0$ then $f(n)\leq g(n)$. Pick $x\in D_f[\mathbb I]$, and let $y=(1-p_{\max I_{n_0}})x(1-p_{\max I_{n_0}})$. Notice that $y\in D[\mathbb I]$, and moreover if $n\geq n_0$, then $p_{I_n}xp_{I_n}=p_{I_n}yp_{I_n}$. Let $\ell\in\NN$. If $\ell\leq n_0$, then $p_{I_\ell}yp_{I_\ell}=0$ and for all $m$ we have that $\psi_m(\varphi_m(p_{I_\ell}yp_{I_\ell}))=0$. If $\ell\geq n_0$, pick $m\geq g(\max I_\ell)$. Since $\ell\geq n_0$, then $\max I_\ell\geq n_0$, hence $g(\max I_\ell)\geq f(\max I_\ell)$. Since $x\in D_f[\mathbb I]$, we have that 
\[
\norm{p_{I_\ell}yp_{I_\ell}-\psi_m(\varphi_m(p_{I_\ell}yp_{I_\ell}))}=\norm{p_{I_\ell}xp_{I_\ell}-\psi_m(\varphi_m(p_{I_\ell}xp_{I_\ell}))}<2^{-\ell}.
\]
This shows that $y\in D_g[\mathbb I]$. Since $\pi(x)=\pi(y)$, we are done.

\eqref{str2}: Without loss of generality we may assume that if $I_i\subseteq J_j$, then $i>j$. Pick $x\in D_f[\mathbb I^0]$, and fix an even $i_0$ such that if $i\geq i_0$ there is $j$ such that $I_i\cup I_{i+1}\subseteq J_j\cup J_{j+1}$. Let $y=(1-p_{I_{i_0}})x(1-p_{I_{i_0}})$, so that $\pi(x)=\pi(y)$. Let 
\[
C_0=\{i\geq i_0\mid \exists j (I^0_i\subseteq J^0_{j})\}.
\]
Define
\[
y_0=\sum_{i\in C_0}p_{I^0_i}yp_{I^0_i}, \text{ and } y_1=y-y_0.
\]
Notice that by our choice of $i_0$, we have that $y_0\in D[\mathbb J^0]$ and $y_1\in D[\mathbb J^1]$. We prove that  $y_0\in D_f[\mathbb J^0]$ and leave the rest to the reader. Write $y_0=\sum p_{J^0_n}y_0p_{J^0_n}$. For $n\in\NN$, we want to show that if $m\geq f(\max (J_n^0))$ then 
\[
\norm{p_{J^0_n}y_0p_{J^0_n}-\psi_m(\varphi_m(p_{J^0_n}y_0p_{J^0_n}))}<2^{-n}.
\]
If $n$ is such that $p_{J^0_n}y_0p_{J^0_n}=0$, there is nothing to prove, as $\varphi_m(0)=0=\psi_m(0)$ for all $n$. If not, let $D=\{i\mid I^0_i\subseteq J_n^0\}$. $D$ is nonempty, as if $p_{J^0_n}y_0p_{J^0_n}\neq 0$, then $n$ is greater than the least $j$ such that $I_{i_0}$ was a subset of $J_j$. By the definition of $y_0$, we have that 
\[
p_{J^0_n}y_0p_{J^0_n}=\sum_{i\in D}p_{I^0_i}y_0p_{I^0_i}.
\]
Let $m\geq f(\max (J_n^0))$. Since $f$ is increasing, for all $i\in D$ we have that $f(\max (I^0_i))< f(\max (J_n^0))$. Since $p_{I^0_i}y_0p_{I^0_i}=p_{I^0_i}xp_{I^0_i}$, $x\in D_f[\mathbb I^0]$, and $m\geq f(\max (I^0_i))$, then 
\[
\norm{p_{I^0_i}y_0p_{I^0_i}-\psi_m(\varphi_m(p_{I_i^0}y_0p_{I^0_i}))}=\norm{ p_{I^0_i}x_0p_{I^0_i}-\psi_m(\varphi_m(p_{I_i^0}x_0p_{I^0_i}))}<2^{-i}, 
\]
hence 
\[
\norm{p_{J^0_n}y_0p_{J^0_n}-\psi_m(\varphi_m(p_{J_n^0}y_0p_{J^0_n}))}<\sum_{i\in D}2^{-i}<2^{-\min D+1}.
\] 
Since $\min D<n$, we are done. 

\eqref{str3}: This follows from Lemma~\ref{lem:summing}.
\end{proof}
Let $\iso$ be an automorphism of $ \SQ(A)$. For each $f\in\NN^\NN$ and $\mathbb I\in\mathrm{Part}_\NN$, the map
\[
\Psi'_{f,\mathbb I}\colon \prod E_{f(\max I_n)}/\bigoplus E_{f(\max I_n)}\to \SQ(A),
\]
 defined as
\[
\Psi'_{f,\mathbb I}(\pi(x))=\pi(\Psi_{f,\mathbb I}(x)),
\]
is a well defined contractive linear map, and so is $\Lambda_{f,\mathbb I}=\Lambda\circ\Psi'_{f,\mathbb I}$.

\begin{proposition}\label{prop:onedominates}
Let $\mathbb I\in\mathrm{Part}_\NN$, and $\CI\subseteq\SP(\NN)$ be a dense nonmeager ideal. Let $q\in\SQ(A)$ be a projection. If $q\geq \pi(p_X^\mathbb I)$ for all $X\in\CI$, then $q=1$.
\end{proposition}

\begin{proof}
Let $q_X=\pi_A(p_X^{\mathbb I})$. Suppose that $q\in\SQ(A)$ is a nontrivial projection such that $qq_X=q_X$ for all $X\in\CI$. Let $r=1-q$, and suppose that $\norm{r}=1$. By Lemma~\ref{lem:stratification}, we can find $\mathbb J\in\mathrm{Part}_\NN$ and a sequence $n_i$ such that each $J_n$ is a union of finitely many $I_i$'s, and $\norm{rp_{[n_i,n_{i+1})}^{\mathbb J}}>1-2^{-i}$. Let $k_i$ such that $I_k\subseteq J_i$ for all $k$ with $k_i\leq k\leq k_{i+1}$. Then $\norm{rp_{[k_{n_i},k_{n_{i+1}})}}>1-2^{-i}$. Let $L$ be infinite such that $X=\bigcup_{i\in L}[k_{n_i},k_{n_{i+1}})\in\CI$. Then $\norm{\pi_A(rp_X^\mathbb I)}>\frac{1}{2}$, a contradiction.
\end{proof}
\begin{remark}
While in a situation of a reduced product, or in the abelian case, the density of $\CI$ is the only hypothesis one needs to get results equivalent to Proposition~\ref{prop:onedominates}, in the general case the nonmeagerness of $\CI$ is key; in fact, one can modify a result of Wofsey (see \cite[Proposition 2.2 and 2.3]{Wofsey}) to give a counterexample to the thesis of Proposition~\ref{prop:onedominates} if $\CI$ is taken to be a meager ideal. In particular, let $e_n$ be the canonical base of a separable Hilbert space $H$. Let $p_{n}$ be the projection onto $\overline\spann\{e_i\mid i\leq n\}$. Then it is possible to find $\mathbb I\in \mathrm{Part}_\NN$ and a meager dense ideal $\mathcal I\subseteq\NN$ such that there is a nonzero projection $q\in\mathcal B(H)$ with $\pi(q)\geq\pi(p_X^{\mathbb I})$ for all $X\in\mathcal I$. A version of the above where the $p_X^{\mathbb I}$ were not assumed to be projections to begin with was proved as \cite[Lemma~3.11]{V.Rigidity}.
\end{remark}
Thanks to Proposition~\ref{prop:onedominates} we can show that the ideals provided by Theorem~\ref{thm:lifting} are in fact as large as $\SP(\NN)$. This corresponds to Proposition~\ref{prop:liftoneverything}.

\begin{lemma}\label{lem:liftforeverything.Borel}
Let $\mathbb I\in\mathrm{Part}_\NN$ and $f\in\NN^\NN$ such that $f(n)\geq n$. Suppose that $\alpha_{f,\mathbb I}$ is an asymptotically additive map that is a lift of $\Lambda_{f,\mathbb I}$ on a nonmeager ideal $\CI_{f,\mathbb I}$. Then $\alpha_{f,\mathbb I}\circ\Phi_{f,\mathbb I}$ is a lift of $\Lambda_{f,\mathbb I}$ on
\[
\{x\in D[\mathbb I]\mid x-\Psi_{f,\mathbb I}(\Phi_{f,\mathbb I}(x))\in A\}.
\]
\end{lemma}

\begin{proof}
Let $\alpha=\alpha_{f,\mathbb I}$, $\Phi=\Phi_{f,\mathbb I}$, $\Psi=\Psi_{f,\mathbb I}$ and $\CI=\CI_{f,\mathbb I}$. We  first show that $\alpha(\Phi(1))-1\in A$, and then prove that this is sufficient to obtain our thesis. Since $p_{I_n}\in Y_{\max I_n}$ for all $n$, then $\Psi(\Phi(p_X^\mathbb I))-p_X^\mathbb I\in A$. Since $\alpha$ is a lift for $\Lambda_{f,\mathbb I}$, if $X\in\CI$ then
\[
\alpha(\Phi(p^\mathbb I_X))-\Lambda(\pi(p^\mathbb I_X))\in A.
\]
 \begin{claim}
For all $Y\subseteq X$ we have $\pi(\alpha(\Phi(p_Y^\mathbb I)))\leq\pi(\alpha(\Phi(p_X^\mathbb I)))$.
 \end{claim}
 \begin{proof}
Since $\alpha$ is asymptotically additive, we have that
 \[
 \pi(\alpha(\Phi(p_X^\mathbb I)))=\pi(\alpha(\Phi(p_{X\setminus Y}^\mathbb I)))+\pi(\alpha(\Phi(p_Y^\mathbb I))).
 \]
 Therefore it is enough to prove that each $\pi(\alpha(p_X))$ is positive. For this, it suffices to prove that for all $\varepsilon>0$ there is $n_0$ such that for all $n\geq n_0$ there is a positive contraction $x_n\in A$ with
\[
\norm{\alpha_n(\varphi_{f(\max I_n)}(p_{I_n}))-x_n}<\varepsilon.
\]
If so then $\alpha(\Phi(p_X^{\mathbb I}))-\sum_{n\in X}x_n\in A$. Suppose the converse and let $\varepsilon>0$ and $n_k$ be an infinite sequence contradicting the hypothesis. Since $\alpha$ is asymptotically additive we have that each $\alpha_{n_k}$ has range included in $(p_{m'_k}-p_{m_k})A(p_{m'_k}-p_{m_k})$ where $m_k\to\infty$ as $k\to\infty$. By passing to a subsequence we can assume that $\{n_k\}_k\in\CI$ and that for all $k$, $m'_{k}<m_{k+1}$.  Let $Z=\{n_k\}_k$. Since $Z\in\CI$ we have that $\pi(\alpha(\Phi(p_Z^{\mathbb I}))$ is a projection, as $\pi(\alpha(\Phi(p_Z^{\mathbb I}))\iso(\pi(p_Z^{\mathbb I}))$. Also
\[
\alpha(\Phi(p_Z^{\mathbb I}))\in\prod (p_{m'_k}-p_{m_k})A(p_{m'_k}-p_{m_k}).
\]
Since the quotient maps and the inclusions
\[
\prod (p_{m'_k}-p_{m_k})A(p_{m'_k}-p_{m_k})\subseteq\SM(A)
\]
 and
 \[
 \prod (p_{m'_k}-p_{m_k})A(p_{m'_k}-p_{m_k})/\bigoplus (p_{m'_k}-p_{m_k})A(p_{m'_k}-p_{m_k})\subseteq\SQ(A)
 \]
 commute, there is a sequence of positive contractions $(x_{n_k})$ with
 \[
 \norm{x_{n_k}-\alpha_{n_k}(\varphi_{f(\max n_k)}(p_{I_{n_k}}))}\to 0.
 \]
 \end{proof}

Since $\Phi(p^\mathbb I_X)\leq 1$, as an element of $\prod E_{f(\max I_n)}$, we have that $\pi(\alpha(1))$ dominates $\pi(\alpha(\Phi(p^\mathbb I_X)))=\iso(\pi(p_X^\mathbb I))$. Since $\iso$ is an automorphism and $\pi(\alpha(\Phi(1)))$ is positive, we can apply Proposition~\ref{prop:onedominates} to get
\[
1-\alpha(\Phi(1))\in A.
\]
Fix now $x$ as in the hypothesis, so that $\Psi(\Phi(x))-x\in A$. Let $y$ such that $\pi(y)=\iso(\pi(x))$ and define
\[
\CI_x=\{X\subseteq\NN\mid \alpha(\Phi(p^\mathbb I_X))(\alpha(\Phi(x))-y)\in A\}.\
\]
This is an ideal containing $\CI$ and so is nonmeager and contains all finite sets. We want to prove that $\CI_{x}=\SP(\NN)$. The argument is similar to the one of Proposition~\ref{prop:liftoneverything}.

Since $\alpha(\Phi(p_X^\mathbb I))$ is defined as the limit of the partial sums $\sum_{i\leq j} \alpha_i(\varphi_i(p_{X}^{\mathbb I}))$, and since $\Phi$ sends strictly convergent sequences in $\{p_X^\mathbb I\mid X\subseteq\NN\}$ to sequences which are converging in the product topology of $\prod E_{f(\max I_n)}$, $\alpha\circ\Phi$ is strictly-strictly continuous when restricted to $\{p_X^\mathbb I\mid X\subseteq\NN\}$. Since $x$, $\mathbb I$, $f$, $\Phi$ and $y$ are fixed, multiplification on the left is a strictly-strictly continuous operation, and $A$ is Borel in the strict topology of $ \SM(A)$, determining whether $\alpha(\Phi(p^\mathbb I_X))(\alpha(\Phi(x))-y)$ is in $A$ is a Borel condition. In particular,  $\CI_x$ is Borel. Since every Borel nontrivial ideal containing all finite sets is meager, we have $\CI_x=\SP(\NN)$. From this and $1-\alpha(\Phi(1))\in A$, we have
\[
\pi(\alpha(\Phi(x)))-\iso(\pi(x)))=0.\qedhere
\]
\end{proof}
 Recall the definition of skeletal map from Definition~\ref{defin:skeletal}, and define
\begin{eqnarray*}
\mathcal X=\{(f,\mathbb I,\alpha^0,\alpha^1)\mid\alpha^i\text{ is a skeletal lift of }\Lambda_{f,\mathbb I^i} \text{ on }\Phi_{f,\mathbb I^i}(D_f[\mathbb I^i])\text{ for } i=0,1\\
\text{ and } \alpha^0\restriction D_f[\mathbb I^0]\cap D_f[\mathbb I^1]=\alpha^1\restriction D_f[\mathbb I^0]\cap D_f[\mathbb I^1])\}.
\end{eqnarray*}
Note that $\mathrm{Part}_\NN$ has a natural Polish topology obtained by identifying every $\mathbb{I}\in\mathrm{Part}_\NN$ with an element of $\NN^\NN$. Recalling that we identified the set of skeletal maps $\Skel$, with its separable metrizable topology, with a subset of $\SM(A)^\NN$ (see Definition~\ref{defin:skel2}) that $\mathcal X$ carries a natural Polish topology.
\begin{lemma}
For all $f\in\NN^\NN$ and $\mathbb I\in\mathrm{Part}_\NN$, there are $\alpha^0$ and $\alpha^1$ in $\Skel$ such that $(f,\mathbb I,\alpha^0,\alpha^1)\in \mathcal X$.
\end{lemma}
\begin{proof}
By Theorem \ref{thm:lifting}, Lemma \ref{lem:liftforeverything.Borel}, and the fact that the asymptotically additive maps in Lemma~\ref{lemma:C->asymptotic} can be chosen to be skeletal (see Lemma~\ref{lemma:skeletal-replacement}), for every $\mathbb I\in\mathrm{Part}_\NN$ and $f\in\NN^{\NN}$ there are $\alpha^0$ and $\alpha^1$ such that, for $i=0,1$, $\alpha^i$ is a skeletal lift of $\Lambda_{f,\mathbb I^i}$ on $\Phi_{f,\mathbb I^i}(D_f[\mathbb I^i])$.

We are left to ensure that $\alpha^0$ and $\alpha^1$ agree on $D_f[\mathbb I^0]\cap D_f[\mathbb I^1]$.
 If $x\in D_f[\mathbb I^0]\cap D_f[\mathbb I^1]$, then $x\in D[\mathbb I]$. Write $x=\sum x_n$ with $x_n=p_{I_n}xp_{I_n}$. Modify $\alpha^1$ on $D[\mathbb I]$ by defining $\tilde\alpha^1_n(\Psi_{f,\mathbb I^1}(x_{2n+1}))=\alpha^0_{n}(\Psi_{f,\mathbb I^0}(x_{2n+1}))$ and $\tilde\alpha^1_n((\Psi_{f,\mathbb I^1}(x_{2n+2}))=\alpha^0_{n+1}(\Psi_{f,\mathbb I^0}(x_{2n+2}))$. Since $\alpha^0(\Psi_{f,\mathbb I^0}(x))$ is a lift of $\Lambda(x)$, so is $\tilde\alpha^1$. Moreover, $\tilde\alpha$ can be chosen to be skeletal by composing with the map $\rho$ as defined in  (see~\ref{subsec:notlast}).
 \end{proof}

\begin{lemma}\label{lem:afteragivenpoint}
Let $(f,\mathbb I,\alpha^0,\alpha^1)$ and $(g,\mathbb J,\beta^0,\beta^1)$ be in $\mathcal X$ and let $\varepsilon>0$. Then there is $M$ such that for all $n>M$ and $x=p_{[M,n]}xp_{[M,n]}$ with $ \norm{x}\leq 1$, if $x\in D_f[\mathbb I^i]\cap D_g[\mathbb J^j]$ we have \[\norm{\alpha^i(\Phi_{f, \mathbb I^i}(x))-\beta^j(\Phi_{g,\mathbb J^j}(x))}\leq\varepsilon.\]
\end{lemma}

\begin{proof}
We work towards a contradiction. Since for every $f$ and $\mathbb I$ there are $\alpha^0,\alpha^1$ such that $(f,\mathbb I,\alpha^0,\alpha^1)\in \mathcal X$, modifying $\mathbb I$ and $\mathbb J$ if necessary, we can assume there exist $\varepsilon>0$ and $(f,\mathbb I,\alpha^0,\alpha^1), (g, \mathbb J,\beta^0,\beta^1)$ such that there is an increasing sequence $m_1<m_2<\cdots$ and $x_i=p_{[m_i,m_{i+1})}x_ip_{[m_i,m_{i+1})}$ with $\norm{x_i}\leq 1$, $x_i\in D_f[\mathbb I^0]\cap D_g[\mathbb J^0]$ and such that for every $i$ we have that
\[
\norm{\alpha^0(\Phi_{f,\mathbb I^0}(x_i))-\beta^0(\Phi_{g,\mathbb J^0}(x_i))}>\varepsilon.
\]
Let $x=\sum x_i$. Since the image of each $\alpha^0_n$ and $\beta^0_n$ is included in a corner\footnote{Corners were define right before Lemma~\ref{lemma:aa-corners} as cutdowns by positive elements.} $(p_i-p_j)A(p_i-p_j)$, where $j\to\infty$ as $n\to\infty$, we can find an increasing sequence $n_k$ such that for every $l\neq k$ we have that
\[
\alpha^0_{n_k}\alpha^0_{n_l}=\beta^0_{n_k}\beta^0_{n_l}=\beta^0_{n_k}\alpha^0_{n_l}.
\]
Setting $Y=\bigcup [m_{n_k},m_{n_k+1})$ and $z=p_Yxp_Y$ we have that $z\in D_f[\mathbb I^0]\cap D_g[\mathbb J^0]$, $\norm{z}\leq 1$ and

\begin{eqnarray*}
\norm{\pi(\alpha^0(\Phi_{f,\mathbb I^0}(z))-\beta^0(\Phi_{g,\mathbb J^0}(z)))}&\geq&\\
\limsup\norm{\alpha^0(\Phi_{f,\mathbb I^0}(x_{n_k}))-\beta^0(\Phi_{g,\mathbb J^0}(x_{n_k}))} &\geq&\varepsilon.
\end{eqnarray*}
On the other hand since $(f,\mathbb I,\alpha^0,\alpha^1)$ and $(g,\mathbb I,\beta^0,\beta^1)$ are in $\mathcal X$ and $z\in D_f[\mathbb I^0]\cap D_g[\mathbb J^0]$ we have
\[
\pi(\alpha^0(\Phi_{f,\mathbb I}(z))=\iso(\pi(z))=\pi(\beta^0(\Phi_{g,\mathbb J}(z))),
\]
a contradiction.
\end{proof}

 For a fixed $\varepsilon>0$, define a colouring $[\mathcal X]^2=K_0^\varepsilon\cup K_1^\varepsilon$ with
\[
\{(f,\mathbb I,\alpha^0,\alpha^1),(g,\mathbb J,\beta^0,\beta^0)\}\in K_0^\varepsilon
\]
if and only if there are $i$ and $j$ in $\{0,1\}$, $n\in\NN$ and $x=p_nxp_n$ with $\norm{x}=1$ and such that
\[
x\in D_f[\mathbb I^i]\cap D_g[\mathbb J^j]\text{ and }\norm{\alpha^i(\Phi_{f,\mathbb I^i}(x))-\beta^j(\Phi_{g,\mathbb J^j}(x))}>\varepsilon.
\]
Again, the following is proved by applying the definition of the natural separable metrizable topology on $\Skel$ (see Definition~\ref{defin:skel2}).
\begin{proposition}\label{prop:SkeletalPolish}
 $K_0^\varepsilon$ is an open subset of $[\SX]^2$ when $\SX$ is given the separable metric topology obtained from its inclusion in $\NN^\NN\times\mathrm{Part}_\NN\times\Skel\times\Skel$.\qed
\end{proposition}

\begin{lemma}
There is no uncountable $K_0^\varepsilon$-homogeneous set for any $\varepsilon > 0$.
\end{lemma}

\begin{proof}
Working towards a contradiction, let $\varepsilon>0$ and $\mathcal Y$ be a $K_0^\varepsilon$-homogeneous set of size $\aleph_1$. As in Lemma~\ref{lem:nohomogredprod}, we refine $\mathcal Y$ to an uncountable subset of itself several times, but we keep the name $\mathcal Y$.

Since we have, by $\OCA$, tgat $\mathfrak b>\omega_1$ (\S\ref{subsec:ForcingAxiom}, then  \cite[Lemma~3.4]{Farah.C} implies that all sets of size $\aleph_1$ are $\leq_1$-bounded. Hence we may find $\hat f$ and $\hat{\mathbb I}$ with the property that for all $(f,\mathbb I,\alpha^0,\alpha^1)\in \mathcal Y$ we have $f<^*\hat f$ and $\mathbb I<_1 \hat{\mathbb I}$. By definition of $<^*$ and $<_1$, if $(f,\mathbb I,\alpha^0,\alpha^1)\in \mathcal Y$ there are $n_f$ and $m_{\mathbb I}$ such that for all $n\geq n_f$ and $m\geq m_{\mathbb I}$ we have $f(n)<\hat f(n)$ and that there is $k$ such that $I_{m}\cup I_{m+1}\subseteq \hat I_k\cup\hat I_{k+1}$. We refine $\mathcal Y$ so that $n_f=\overline n$ and $m_{\mathbb I}=\overline m$, whenever $(f,\mathbb I,\alpha^0,\alpha^1)\in\mathcal Y$.

Fix $\hat\alpha^0,\hat\alpha^1$ such that $(\hat f,\hat{\mathbb I},\hat \alpha^0,\hat\alpha^1)\in \mathcal X$. Thanks to Lemma~\ref{lem:afteragivenpoint} for all $(f,\mathbb I,\alpha^0,\alpha^1)\in \mathcal Y$ we can find $M_f\geq\overline m$ such for all $n\geq M_f$ and $x=p_{[M_f,n]}xp_{[M_f,n]}$ with $\norm{x}\leq 1$ then if $x\in D_f[\mathbb I^i]\cap D_{\hat f}[\hat {\mathbb I}^j]$ for $i$ and $j$ in $\{0,1\}$, then we have
\begin{equation}\label{OCAequation1}
\norm{\sum_{k}\alpha^i_k(\Phi_{f,\mathbb I^i}(p_{I_k^i}xp_{I_k^i}))-\sum_l\hat\alpha^j_k(\Phi_{g,\mathbb J^j}(p_{\hat I_l^j}xp_{\hat I_l^j}))}\leq\varepsilon/2.
\end{equation}
We can suppose $M_f=\overline M$ for all elements of $\mathcal Y$. Again using the pigeonhole principle, we can assure that for all $k\leq\max\{\overline n, \overline M+3\}$ we have that $f(k)=g(k)$ and $I_{k}=J_{k}$ if $(f,\mathbb I,\alpha^0,\alpha^1)$ and $(g,\mathbb J,\beta^0,\beta^1)$ are in $\mathcal Y$. Note that $\overline K=\max I_{\overline M+2}>\overline M$. Also, for every $k$ such that $2k\leq \overline M$ and $(f,\mathbb I,\alpha^0,\alpha^1)$ and $(g,\mathbb J,\beta^0,\beta^1)$ in $\mathcal Y$, the domains of $\alpha^0_k$ and of $\beta^0_k$ are the same, as well as the domains of $\alpha^1_k$ and $\beta^1_k$. These domains in fact depends only on $f(k)$ and $I_k$. Also, for $x=p_{\overline K}x p_{\overline K}$,
\begin{eqnarray*}
x\in D_f[\mathbb I^0]&\Rightarrow&x\in D_f[\mathbb J^0], \,\, \Phi_{f,\mathbb I^0}(x)=\Phi_{g,\mathbb J^0}(x)\text{ and }\\
x\in D_g[\mathbb I^1]&\Rightarrow& x\in D_f[\mathbb J^1], \,\, \Phi_{f,\mathbb I^1}(x)=\Phi_{g,\mathbb J^1}(x).
\end{eqnarray*}
Since the space of all skeletal maps from $\sum_{k\mid 2k\leq \overline M}E_{f(i)}$ to $A$ is separable in the uniform topology, we can refine $\mathcal Y$ to an uncountable subset such that whenever $(f,\mathbb I,\alpha^0,\alpha^1)$ and $(g,\mathbb J,\beta^0,\beta^1)$ are in $\mathcal Y$, if $2k\leq\overline M$, then

\begin{enumerate}[label=(O\arabic*)]
\item\label{OCAcond1} $\norm{\alpha^0_k-\beta^0_k}<\varepsilon/(2\overline M)$;
\item\label{OCAcond2} $\norm{\alpha^1_k-\beta^1_k}<\varepsilon/(2\overline M)$.
\end{enumerate}
This is the final refinement we need. Pick $(f,\mathbb I,\alpha^0,\alpha^1)$ and
$(g,\mathbb J,\beta^0,\beta^1)$ in $\mathcal Y$ and $x$ witnessing that
$\{(f,\mathbb I,\alpha^0,\alpha^1),(g, \mathbb J,\beta^0,\beta^1)\}\in
K_0^\varepsilon$. Then $x=p_{[n,n']}xp_{[n,n']}$ for some natural numbers $n$ and $n'$.

\begin{itemize}
\item If $n'\leq\overline K$, then, since $I_k=J_k$ for all $k$ such that $p_{I_k}xp_{I_k}\neq 0$, we have that either $x\in D_f[\mathbb I^0]\cap D_g[\mathbb J^0]$, or $x\in D_f[\mathbb I^1]\cap D[\mathbb J^1]$ (or both). If the first case applies, we have a contradiction thanks to condition~\ref{OCAcond1}, while the second case is contradicted by condition \ref{OCAcond2}. In case $x\in D_f[\mathbb I^0]\cap D_f[\mathbb I^1]\subseteq D[\mathbb I]$, and
\[
\norm{\alpha^i(\Phi_{f,\mathbb I^i}(x)-\beta^{1-i}(\Phi_{f,\mathbb J^{1-i}}(x))}>\varepsilon,
\]
then the contradiction comes from that $\alpha^0(\Phi_{f,\mathbb I^0}(x))=\alpha^1(\Phi_{f,\mathbb I^1}(x))$, and conditions \ref{OCAcond1} and \ref{OCAcond2}.
\item If $n>\overline M$, then (\ref{OCAequation1}) and the triangle inequality lead to a contradiction.
\item if $n\leq\overline M<\overline K<n'$, since $\overline M<\max I_{\overline M}<\max I_{\overline M+1}<\overline K$, we can split $x=y+z$ where $y=p_{k}xp_{k}$ and $z=p_{(k,n']}xp_{(k,n']}$ for some $k$ with $\overline M<k\leq\overline K$. But then, if $x\in D_f[\mathbb I^0]$ then $\alpha^0(\Phi_{f,\mathbb I^0}(x))=\alpha^0(\Phi_{f,\mathbb I^0}(y))+\alpha^0(\Phi_{f,\mathbb I^0}(z))$, and we reach a contradiction by the triangle inequality. 
\end{itemize}
The case of $x\in D_f[\mathbb I^1]$ is treated similarly.
\end{proof}

Fix $\varepsilon_k=2^{-k}$ and write $\mathcal X=\bigcup_n \SX_{n,k}$ where each $\SX_{n,k}$ is $K_1^{\varepsilon_k}$-homogeneous, thanks to $\OCA_\infty$. Since the product order on $\NN^\NN\times\mathrm{Part}_{\NN}$ is $\sigma$-directed, we can inductively construct, by repeatedly using  \cite[Lemma 2.2.2]{Farah.AQ}, sets $\mathcal D_k$ and $\SY_k$, for $k\in\NN$, such that
\begin{itemize}
\item $\SY_{k+1}\subseteq\SY_k$;
\item $\mathcal D_k$ is a countable dense subset of $\SY_k$;
\item $\SY_k$ is $K^{\varepsilon_k}_1$-homogeneous;
\item $\SY_k$ is cofinal when $\NN^\NN\times\mathrm{Part}_\NN$ is considered with the product order.
\end{itemize}

\begin{lemma}\label{lem:almostthere}
Suppose that $x_0$ and $x_1$ in $\SM(A)$ are such that there are, for $\ell\in\NN$, natural numbers $n_\ell^0$ and $n_\ell^1$, and $(f_\ell,\mathbb I_\ell, \alpha^0_\ell,\alpha^1_\ell)\in\mathcal D_k$ such that for every $\ell$:

\begin{enumerate}[label=(\arabic*)]
\item\label{lem:cond:Borelgraph,c1} there is $i$ such that
\[
\max (I_\ell)_{2i}=n_\ell^0\text{ and }\max (I_\ell)_{2i+1}=n_\ell^1;
\]
\item\label{lem:cond:Borelgraph,c2} if $\ell<\ell'$, then for every $i$ such that $\max (I_\ell)_i\leq\max\{n_\ell^0,n_\ell^1\}$ we have $(I_{\ell})_i=(I_{\ell'})_i$ and $f_\ell(i)=f_{\ell'}(i)$; and
\item\label{lem:cond:Borelgraph,c3} for every $p_{n_\ell^0}x_0p_{n_\ell^0}\in D_{f_\ell}[\mathbb I_\ell^0]$, $p_{n_\ell^1}x_1p_{n_\ell^1}\in D_{f_\ell}[\mathbb I_\ell^1]$.
\end{enumerate}
Then
\[
\norm{\pi(\lim_\ell\sum_{j<\ell}\alpha_\ell^0(\Phi_{f_\ell,\mathbb I_\ell^0}(p_{(n_j^0,n_{j+1}^0]}x_0p_{(n_j^0,n_{j+1}^0]})))-\Lambda(\pi(x_0))}<10\varepsilon_k
\]
and
\[
\norm{\pi(\lim_\ell\sum_{j<\ell}\alpha_\ell^1(\Phi_{f_\ell,\mathbb I_\ell^1}(p_{(n_j^1,n_{j+1}^1]}x_1p_{(n_j^1,n_{j+1}^1]})))-\Lambda(\pi(x_1))}<10\varepsilon_k.
\]
\end{lemma}

\begin{proof}
We prove the statement for $x_0$. The proof for $x_1$ is equivalent. Let two sequences $\{n_\ell^0\}_{\ell\in\NN}$ and $\{(f_\ell,\mathbb I_\ell,\alpha^0_\ell,\alpha^1_\ell)\}_{\ell\in\NN}$ be as in the hypothesis. Define $\hat{\mathbb I}$, an element of $\mathrm{Part}_\NN$, by $\hat I_n=(I_\ell)_n$ if $\max (I_\ell)_n\leq n_\ell^0$. By condition \ref{lem:cond:Borelgraph,c2}, $\hat{\mathbb I}$ is well defined. Define
\[
x_{0,m}=p_{\hat I_{2m}\cup \hat I_{2m+1}}x_0p_{\hat I_{2m}\cup \hat I_{2m+1}}.
\]
By condition~\ref{lem:cond:Borelgraph,c3}, we have that $x_0=\sum_m x_{0,m}$. Pick $f$ large enough such that $x_0\in D_f[\hat{\mathbb I}^0]$. By cofinality of $\SY_k$, there is $(g,\mathbb J,\alpha^0,\alpha^1)\in \SY_k$ such that $f\leq^* g$ and $\hat{\mathbb I}\leq_1 \mathbb J$. By definition of $\leq_1$, we have that for every $n$ large enough there is $m$ such that $\hat I_{2n}\cup\hat I_{2n+1}\subseteq J_m\cup J_{m+1}$. Let
\[
z_0=\sum_{m\mid x_{0,m}\in D_g[\mathbb J^0]}x_{0,m}\text{ and } z_1=\sum_{m\mid x_{0,m}\in D_g[\mathbb J^1]\setminus D_g[\mathbb J^0]}x_{0,m}.
\]
Then $x_0-z_0-z_1\in A$, and $z_0\in D_g[\mathbb J^0]$ and $z_1\in D_g[\mathbb J^1]$. In particular, $\pi(\alpha^0(\Phi_{g,\mathbb J^0}(z_0)))=\Lambda(\pi(z_0))$ and $\alpha^1(\Phi_{g,\mathbb J^1}(z_1))=\Lambda(\pi(z_1))$. On the other hand, since for every $\ell$ we have
\[
\{(f_\ell,\mathbb{I}_\ell,\alpha^0_\ell,\alpha^1_\ell),(g,\mathbb J,\alpha^0,\alpha^1)\}\in K_1^{\varepsilon_k}
\]
by homogeneity of $\SY_k$, if $m\leq n_\ell^0$ we have that
\[
\norm{\alpha^0_\ell(\Phi_{f_\ell,\mathbb I_\ell^0}(x_{0,m}))-\alpha^0(\Phi_{g,\mathbb J^0}(x_{0,m}))}\leq\varepsilon_k\text{ if }x_{0,m}\in D_g[\mathbb J^0]
\]
 and
\[
\norm{\alpha^0_\ell(\varphi_{f_\ell,\mathbb I^0_\ell}(x_{0,m}))-\alpha^1(\Phi_{g,\mathbb J^1}(x_{0,m})}\leq\varepsilon_k\text{ if }x_{0,m}\in D_g[\mathbb J^1].
\]
Passing to strict limits of partial sums, we have the thesis.
\end{proof}

Recall that $\Gamma_\Lambda$ is the graph of $\Lambda$. The next lemma  gives an analytic and a coanalytic definition of $\Gamma_\Lambda$.
\begin{lemma}\label{lemma:almostdone}
Assume $\OCA$ and $\MA_{\aleph_1}$. For positive contractions $a$ and $b$ in $\SM(A)$, the following conditions are equivalent:

\begin{enumerate}[label=(\roman*)]
\item\label{thm:Borelgraph,c1} $(a,b)\in\Gamma_\Lambda$;
\item\label{thm:Borelgraph,c2} For every $k\in\NN$, there are positive contraction $x_0$, $x_1$, $y_0$ and $y_1$ in $\SM(A)$ such that $\pi(a)=\pi(x_0+x_1)$, $\pi(b)=\pi(y_0+y_1)$ with the property that there are sequences of natural numbers $n_\ell^0$ and $n_\ell^1$, for $\ell\in\NN$, and a sequence $(f_\ell,\mathbb I_\ell,\alpha^0_\ell,\alpha^1_\ell)$ of elements of $D_k$ satisfying conditions~\ref{lem:cond:Borelgraph,c1}--\ref{lem:cond:Borelgraph,c3} of Lemma~\ref{lem:almostthere},

\begin{enumerate}[label=(\arabic*)]\setcounter{enumii}{3}
\item\label{lem:cond:Borelgraph,c4}
\[
\norm{\lim_i\sum_{j<i}\alpha_i^0(\Phi_{f_i,\mathbb{I}_i^0}(p_{(n_j^0,n_{j+1}^0]}x_0p_{(n_j^0,n_{j+1}^0]}))-y_0}<5\varepsilon_k,
\]
and
\item\label{lem:cond:Borelgraph,c5}
\[
\norm{\lim_i\sum_{j<i}\alpha_i^1(\Phi_{f_i,\mathbb I_i^1}(p_{(n_j^1,n_{j+1}^1]}x_1p_{(n_j^1,n_{j+1}^1]}))-y_1}<5\varepsilon_k
\]
where both limits are strict limits.
\end{enumerate}

\item\label{thm:Borelgraph,c3} For all positive contractions $x_0$, $x_1$, $y_0$ and $y_1$ in $\SM(A)$, if $\pi(x^0+x^1)=\pi(a)$ and for every $k\in\NN$  there are sequences of natural numbers $n_\ell^0$ and $n_\ell^1$, for $\ell\in\NN$, and a sequence $(f_\ell,\mathbb I_\ell,\alpha^0_\ell,\alpha^1_\ell)$ of elements of $D_k$ satisfying conditions \ref{lem:cond:Borelgraph,c1}--\ref{lem:cond:Borelgraph,c3} of Lemma \ref{lem:almostthere}, \ref{lem:cond:Borelgraph,c4} and \ref{lem:cond:Borelgraph,c5}, then $\pi(y_0+y_1)=\pi(b)$.
\end{enumerate}
\end{lemma}

\begin{proof}
\ref{thm:Borelgraph,c1}$\Rightarrow$\ref{thm:Borelgraph,c2}: By cofinality of $\SY_k$ and thanks to Lemma \ref{lem:summing} and Proposition \ref{prop:mainproperties}, there are $(f,\mathbb I,\alpha^0,\alpha^1)\in \SY_k$ and positive contractions $x_0$ and $x_1$  with $\pi(a)=\pi(x_0+x_1)$, $x_0\in D_f[\mathbb I^0]$ and $x_1\in D_f[\mathbb I^1]$. Let
\[
y_0=\alpha^0(\Phi_{f,\mathbb I^0}(x_0))\text{ and }y_1=\alpha^1(\Phi_{f,\mathbb I^1}(x_1)).
\]
Since $(f,\mathbb I,\alpha^0,\alpha^1)\in \mathcal X$, $\pi(y_0+y_1)=\pi(b)$.

Let $n_{-1}^0=n_{-1}^1=0$ and suppose that $n_\ell^0$, $n_\ell^1$ and $(f_\ell,\mathbb I_\ell,\alpha^0_\ell,\alpha^1_\ell)\in \mathcal D_k$ are constructed. By density of $\mathcal D_k$ we can find $n_{\ell+1}^0>n_\ell^0$, $n_{\ell+1}^0>n_\ell^1$ and $(f_{\ell+1},\mathbb I_{\ell+1},\alpha^0_{\ell+1},\alpha^1_{\ell+1})\in \mathcal D_k$ with the property that
\begin{itemize}
\item $(I_{\ell+1})_i=I_i$ for all $i$ such that $\max I_i\leq\max{n_{\ell+1}^0,n_{\ell+1}^1}$,
\item there is $i$ such that $\max I_{2i}=n_{\ell+1}^0$ and $j$ such that $\max I_{2j+1}=n_{\ell+1}^1$
\item if $\max I_i\leq \max{n_{\ell+1}^0,n_{\ell+1}^1}$ then $f_{\ell+1}(i)=f(i)$.
\end{itemize}
Such a construction ensures that conditions \ref{lem:cond:Borelgraph,c1}-\ref{lem:cond:Borelgraph,c3} of Lemma~\ref{lem:almostthere} are satisfied. Since for each $\ell$ we have that $\{(f,\mathbb I,\alpha^0,\alpha^1),(f_\ell,\mathbb I_\ell,\alpha^0_\ell,\alpha^1_\ell)\}\in K_1^{\varepsilon_k}$, then for all $j\in\NN$ we have
\[
\norm{\alpha_i^0(\Phi_{f_i,\mathbb I_i^0}(p_{(n_j^0,n_{j+1}^0]}x_0p_{(n_j^0,n_{j+1}^0]}))-\alpha^0(\Phi_{f,\mathbb I^0}(p_{(n_j^0,n_{j+1}^0]}x_0p_{(n_j^0,n_{j+1}^0]}))}<\varepsilon_k,
\]
and so
\[\norm{\lim_i\sum_{j\leq i}\alpha^0(\Phi_{f,\mathbb I^0}(p_{(n_j^0,n_{j+1}^0]}x_0p_{(n_j^0,n_{j+1}^0]}))-\alpha^0(\Phi_{f,\mathbb I^0}(x_0))}<2\e_k.
\]
Since
\[
y_0=\alpha^0(\Phi_{f,\mathbb I^0}(x_0))=\lim_i\sum_{j\leq i}\alpha^0(\Phi_{f,\mathbb I^0}(p_{(n_j^0,n_{j+1}^0]}x_0p_{(n_j^0,n_{j+1}^0]})),
\]
 \ref{lem:cond:Borelgraph,c4} follows from the triangle inequality. The same calculation gives \ref{lem:cond:Borelgraph,c5}.

Assume now \ref{thm:Borelgraph,c2}. We should note that conditions~\ref{lem:cond:Borelgraph,c1}--\ref{lem:cond:Borelgraph,c5} imply that $\norm{\Lambda(\pi(x_0))-\pi(y_0)}\leq \varepsilon_k$ and $\norm{\Lambda(\pi(x_1))-\pi(y_1)}\leq\varepsilon_k$, therefore \ref{thm:Borelgraph,c1} follows. For this reason, we also have that \ref{thm:Borelgraph,c1} implies \ref{thm:Borelgraph,c3}. Similarly pick positive contractions $a$ and $b$ in $\SM(A)$. If there are $x_0$, $x_1$, $y_0$ and $y_1$ satisfying that for every $k$ there are $n_\ell^0$, $n_\ell^1$ and $(f_\ell,\mathbb I_\ell,\alpha^0_\ell,\alpha^1_\ell)\in \mathcal D_k$, for $\ell\in\NN$, satisfying conditions \ref{lem:cond:Borelgraph,c1}--\ref{lem:cond:Borelgraph,c5}, and such that $\pi(x_0+x_1)=\pi(a)$, then we have that $\pi(y_0+y_1)=\Lambda(\pi(x_0+x_1))$. If \ref{thm:Borelgraph,c3} holds, the left hand side is equal to $\pi(b)$, hence $(a,b)\in\Gamma_\Lambda$, proving \ref{thm:Borelgraph,c1}.
\end{proof}

\begin{proof}[Proof of Theorem \ref{thm:Borel}]
Condition \ref{thm:Borelgraph,c2} from Lemma~\ref{lemma:almostdone} gives that $\Gamma_\Lambda^{1,+}$, the restriction of the graph $\Gamma_\Lambda$ to pairs of positive contractions in $\SM(A)$, is analytic, while \ref{thm:Borelgraph,c3} ensures that $\Gamma_{\Lambda}^{1,+}$ is coanalytic. Consequently $\Gamma_{\Lambda}^{1,+}$ is Borel. 

Since the absolute value, when restricted to self-adjoint operators, is a strictly continuous operation and, for if $a$ and $b$ are self-adjoint contractions in $\SM(A)$ we have that $(a,b)\in\Gamma_\Lambda$ if and only if $(|a|+a,|b|+b)$ and $(|a|-a,|b|-b)$ are in $\Gamma_\Lambda^{1,+}$, then the restriction of the graph to pairs of self-adjoints is Borel. Equally, since addition and taking adjoints are strictly continuous operations and for all contractions  $a$ and $b$ in $\SM(A)$ we have that $(a,b)\in\Gamma_\Lambda$ if and only if $(a+a^*,b+b^*)$ and $(a-a^*,b-b^*)$ are in $\Gamma_{\Lambda}$, we can conclude that $\Gamma_\Lambda$ is Borel.
\end{proof}

\section{Consequences: Embeddings of Reduced Products}\label{sec:cons.NonembLift}
If $A_n$, for $n\in\NN$, are $\Cstar$-algebras and $\CI\subseteq\SP(\NN)$ is an ideal, let
\[
 \bigoplus_{\CI} A_n =\{(a_n)\in\prod A_n\colon \forall\varepsilon>0 \{n\mid \norm{a_n}>\varepsilon\}\in\CI\}.
\]
This algebras, called the $\CI$-reduced product of the $A_n$'s, were considered in \cite{Ghasemi.FDD} and \cite[\S2.5]{Farah-Shelah.RCQ} (see also \cite[\S 16]{Farah.Book}).

In case $A_n=\ce$ for each $n$, $\bigoplus_{\CI}\ce$ is denoted by $c_{\CI}$. If $\CI$ is not countably generated, $\bigoplus_{\CI}A_n$ is not separable. If $\CI$ contains the finite sets, then $\bigoplus_{\CI} A_{n}$ is an essential ideal of $\prod A_n$, and if each $A_n$ is unital, $\SM(\bigoplus_{\CI}A_n)=\prod A_n$. 

We study when coronas of the form $\prod A_n/\bigoplus_{\CI}A_n$ can or cannot embed into the corona of a separable $\Cstar$ algebra. By the main result of \cite{FHV.Calkin}, under CH all $\Cstar$-algebras of density $\mathfrak c$, and therefore all reduced products of separable $\Cstar$-algebras over all ideals, embed in the Calkin algebra. In case of Forcing Axioms, such embeddings cannot exist.

\begin{theorem}\label{thm:noinjection}
Assume $\OCA$ and $\MA_{\aleph_1}$. Let $\CI\subseteq\SP(\NN)$ be a dense, meager ideal containing $\Fin$. Let $A_n$, for $n\in\NN$, be nonzero $\Cstar$-algebras, and $B$ be a separable $\Cstar$-algebras. Then $\prod A_n/\bigoplus_{\mathcal I}A_n$ does not embed into $\SQ(B)$.
\end{theorem}

\begin{proof}
We work towards a contradiction, and fix nonzero $\Cstar$-algebras $A_n$, for $n\in\NN$, a $\Cstar$-algebra $B$, and a meager dense $\mathcal I\subseteq\SP(\NN)$ containing all finite sets such that there is an embedding $\varphi\colon \prod A_n/\bigoplus_\CI A_n\to \SQ(B)$. Denote by $\pi_{\CI}$ the quotient map $\prod A_n\to\prod A_n/\bigoplus_{\CI}A_n$. Since the $A_n$'s are nonzero, we can pick positive elements $a_n,b_n\in A_n$ with $a_nb_n=a_n$ and $\norm{a_n}=\norm{b_n}=1$. (In case $A_n$ is unital, we can pick $a_n=b_n=1$). Consider $\psi\colon \ell_\infty/c_0\to\SQ(B)$ given by 
\[
\psi(\pi_0((\lambda_n)_n)=\varphi(\pi_\CI(\sum_{n}\lambda_na_n)),\,\,\, \text{ for } (\lambda_n)_n\in\ell_\infty,
\]
where $\pi_0\colon\ell_\infty\to c_0$ is the canonical quotient map. Since $\Fin\subseteq\CI$, $\psi$ is a well defined map which sends positive elements to positive elements. For $S\in\SP(\NN)$, let $p_S\in\SM(B)$ be any positive contraction such that  $\pi_B(p_S)=\varphi(\pi_\CI(\sum_{n\in S}b_n))$. Since $\varphi$ is a $^*$-homomorphism, the elements $p_S$, $S\subseteq\NN$ witness that $\psi$ preserves the coordinate structure. By $\OCA+\MA_{\aleph_1}$ and Theorem~\ref{thm:lifting}, we can find a nonmeager ideal $\CJ$ and an asymptotically additive $\alpha$ such that $\alpha$ lifts $\psi$ on $\CJ$. Since $\CJ$ is nonmeager, we can find an infinite $X\in\CJ\setminus \CI$. Since $\norm{a_n}=1$ for all $n$, then 
\[
\norm{\psi(\pi_0(\chi_X))}=\norm{\varphi(\pi_\CI(\sum_{n\in X}a_n))}=\norm{\pi_\CI(\sum_{n\in X}a_n))}=1,
\]
 hence $\limsup_{n\in X}\norm{\alpha_n(1)}=1$. Let $Y\subseteq X$ be the set of all $n$ with $\norm{\alpha_{n}(1)}\geq 1/2$. By passing to a subsequence we can assume that for all $n\in Y$ the $\alpha_n(1)$'s are orthogonal to each other. Let $Z$ be an infinite subset of $Y$ which is in $Z\CI\cap\CJ$. Such a $Z$ exists by density of $\CI$. Since $\alpha$ is a lift on $\CJ$ we have that a contradiction by
\[
0=\norm{\psi(\pi_0(\chi_Z))}=\norm{\pi_{A}(\alpha(\chi_Z))}=\limsup_{n\in Z}\norm{\alpha_n(1)}\geq 1/2.\qedhere
\]
\end{proof}
The meagerness of $\CI$ in the assumption of Theorem~\ref{thm:noinjection} is key. For example, if $\mathcal U$ is an ultrafilter, then $\ce^{\mathcal U}\cong\ce$, and therefore such a reduced product embeds into all coronas. 

\begin{corollary}\label{cor.FA.EmbCalkin}
Whether all $\Cstar$-algebras of density $2^{\aleph_0}$ embed into $\SQ(H)$ is independent of $\ZFC$. The same can be said for $\SQ(A\otimes\mathcal K(H))$, where $A$ is a unital separable $\Cstar$-algebra.
\end{corollary}

\begin{proof}
Theorem~\ref{thm:noinjection} gives a model of set theory and $\Cstar$-algebras of density continuum that do not embed into $\SQ(H)$. On the other hand, by the main result of \cite{FHV.Calkin}, in a model of $\CH$ all $\Cstar$-algebras of density continuum embed into $\SQ(H)$.
\end{proof}
More work on which $\Cstar$-algebras can or cannot embed into the Calkin algebra, and under what conditions, has been done in the recent \cite{FarKatVac:Emb}, where it was shown that for a $\Cstar$-algebra $A$ it is possible to find a ccc forcing $\mathbb P_A$ which embeds it into $\SQ(H)$, and \cite{Vaccaro}, which deals with endomorphisms of $\SQ(H)$.

We now focus on what we can say if $\CI=\Fin$, and on whether we can lift a $^*$-homomorphisms (from a reduced product to a corona $\Cstar$-algebra) to $^*$-homomorphisms from the product to the multiplier algebra. First, we isolate a lemma.
\begin{lemma}\label{lemma:strictcont}
Let $A_n$, for $n\in\NN$, be unital $\Cstar$-algebras, and let $A$ be a nonunital separable $\Cstar$-algebra. Suppose that $\alpha_n\colon A_n\to A$ are injective $^*$-homomorphisms such that $\alpha_i\alpha_j=0$ whenever $i\neq j$ and $\alpha=\sum\alpha_n$ is a well defined $^*$-homomorphism $\prod A_n\to\SM(A)$ lifts an embedding $\prod A_n/\bigoplus A_n\to\SQ(A)$. Then $\alpha$ is strictly continuous.
\end{lemma}
\begin{proof}
The strict topology on $\prod A_n$ is the uniform topology that is, a sequence $x_n$, where each $x_n\in\prod A_n$, converges to $x$ if and only if for all $k\in\NN$ we have that $x_n(k)\to x(k)$ as $n\to\infty$. 

Let $x_n$ and $x$ be elements of $\prod A_n$ such that $x_n$ strictly converges to $x$. Since the ranges of $\alpha_i$ are orthogonal, and $\alpha$ is defined as the limit of the partial sums of the $\alpha_i$'s, we have that $\alpha(x_n)\to\alpha(x)$ if and only if for all $i$ we have that $\alpha_i(x_n(i))\to\alpha_i(x(i))$.
On the other hand, since each $\alpha_i$ is norm-continuous, and the norm topology and the strict topology agree on $A$, we have that each $\alpha_i$ is strictly continuous, therefore $\alpha_i(x_n(i))\to\alpha_i(x(i))$ for all $i$. This concludes the proof.
\end{proof}

\begin{theorem}\label{thm:redprodfindim}
Assume $\OCA+\MA_{\aleph_1}$. Let $k(n)$ be a sequence of natural numbers and $A$ be a separable $\Cstar$-algebra with an increasing  approximate identity of projections. Suppose that there is a unital embedding $\varphi\colon\prod M_{k(n)}(\ce)/\bigoplus M_{k(n)}(\ce)\to\SQ(A)$. Then
\begin{enumerate}
\item\label{thesis:redprodfindim1} there are a projection $q\in\SM(A)$, a strictly continuous unital $^*$-homomorphism $\psi\colon\prod M_{k(n)}(\ce)\to q\SM(A)q$ and a nonmeager ideal $\CI\subseteq\SP(\NN)$ such that $\psi$ lifts $\varphi$ on $\CI$.
\item\label{thesis:redprodfindim2} if $\varphi$ is moreover surjective, then there is a projection $q\in\SM(A)$ and $\ell\in\NN$ such that
\[
 1-q\in A,\,\, qAq\cong\bigoplus_{n\geq \ell} M_{k(n)}(\ce) \text{ and
 }q\SM(A)q\cong\prod_{n\geq \ell} M_{k(n)}(\ce).
\]
\end{enumerate}
\end{theorem}
\begin{proof}
\eqref{thesis:redprodfindim1}:
To simplify notation we assume that $k(n)=n$. For $S\subseteq\NN$, let $M[S]=\prod_{n\in S} M_n(\ce)\subseteq\prod M_n(\ce)$ and $1_S$ be its unit.

Let $\{e_n\}$ be an increasing approximate identity of positive contractions for $A$ such that $e_{n+1}e_n=e_n$ for all $n$. By Theorem~\ref{thm:lifting} we can find an asymptotically additive $\alpha=\sum\alpha_n\colon \prod M_n(\ce)\to\SM(A)$ and a ccc/Fin ideal $\CI$ on which $\alpha$ lifts $\varphi$. Let $\alpha_S=\sum_{n\in S}\alpha_n=\alpha\restriction M[S]$. The following is a generalization of Lemma~\ref{lemma:seqofappmaps}.
\begin{claim}
For every $\varepsilon>0$ there is $n_\varepsilon$ such that whenever $F\subseteq\NN$ is finite with $\min F>n_\varepsilon$ then $\alpha_F$ is an $\varepsilon$-$^*$-homomorphism.
\end{claim}
\begin{proof}
We only prove $\varepsilon$-linearity and leave the rest to the reader. Working towards a contradiction, suppose that we can find finite disjoint intervals $F_1,\ldots,F_n$ with $\max F_i<\min F_{i+1}$ and contractions $x_i,y_i\in M[F_i]$ with
\[
\norm{\alpha_{F_i}(x_i+y_i)-\alpha_{F_i}(x_i)-\alpha_{F_i}(y_i)}>\varepsilon.
\]
 By nonmeagerness of $\CI$ and Proposition~\ref{prop:JT2}, there is an infinite $L\subseteq\NN$ such that $\bigcup_{i\in L}[\min F_i,\min F_{i+1})\in\CI$. With $x=\sum_{i\in L}x_i$ and $y=\sum_{i\in L}y_i$ we get that
\[
\varepsilon\leq \norm{\pi(\alpha(x+y)-\alpha(x)-\alpha(y))}=\norm{\varphi(x+y)-\varphi(x)-\varphi(y)}=0,
\]
a contradiction.
\end{proof}
By the claim we can construct an increasing sequence $n_i$ such that, with $J_i=[n_i,n_{i+1})$, we have that
\begin{itemize}
\item $\alpha_{J_i}$ is a $2^{-i}$-$^*$-homomorphism,
\item if $x\in M[J_i]$ and $y\in M[J_{i+1}]$ then $\norm{\alpha_{J_i}(x)\alpha_{J_{i+1}}(y)}<2^{-i}$
\item if $|i-i'|\geq 2$ then $\alpha_{J_i}\alpha_{J_{i'}}=0$.
\end{itemize}
Since $\alpha$ is asymptotically additive we can further assume that
\begin{itemize}
\item $\alpha_{J_i}[M[J_i]]\subseteq (e_{k_i}-e_{j_i})A(e_{k_i}-e_{j_i})$ where $k_i<j_{i+2}$ for all $i$,
\item each $R_i=\alpha_{J_i}(1_{J_i})$ is a projection, and that
\item $\alpha_{J_i}(x)\in R_iAR_i\subseteq (e_{k_i}-e_{j_i})A(e_{k_i}-e_{j_i})$ for all $x\in M[J_i]$.
\end{itemize}
Note that $R_iR_{i'}=0$ if $|i-i'|\geq 2$. Moreover, since the range of $\alpha_{J_i}$ is contained in $(e_{k_i}-e_{j_i})A(e_{k_i}-e_{j_i})$, and each $e_k$ is a projection, then for all $x\in M[J_i]$ we have that and that $(e_{k_i}-e_{j_i})\alpha_{J_i}(x)=\alpha_{J_i}(x)$.

By Theorem~\ref{thm:apmap-known} we can find a sequence of $^*$-homomorphisms $\beta_{2i}\colon M[J_{2i}]\to R_{2i}AR_{2i}$ such that $\beta_{2i}-\alpha_{2i}\to 0$ as $i\to\infty$. Let $Q_{2i}=\beta_{2i}(1_{J_i})$, and note that $\norm{R_{2i+1}Q_{2i}},\norm{R_{2i-1}Q_{2i}}\to 0$ as $i\to\infty$. We can now find a projection $Q_{2i+1}$ with $\norm{Q_{2i+1}-R_{2i+1}}\to 0$ as $i\to\infty$ such that $Q_{2i+1}=(e_{k_{2i+1}}-e_{j_{2i+1}})Q_{2i+1}$ and $Q_{2i+1}Q_{2i}=Q_{2i+1}Q_{2i+2}=0$. Note that by the choice of $K_i$ and $j_i$ we have that $Q_iQ_{i'}=0$ whenever $i\neq i'$. Let $\beta_{2i+1}\colon M[J_{2i+1}]\to Q_{2i+1}AQ_{2i+1}$ be a $^*$-homomorphism with $\norm{\beta_{2i+1}-\alpha_{J_{2i+1}}}\to 0$ as $i\to\infty$, whose existence is ensured again by Theorem~\ref{thm:apmap-known}. With $\beta=\sum\beta_i$ and $\tilde Q=\beta(1)$ we have the thesis. Notice that by setting $\gamma_j=\beta_i\restriction M_j$ if $j\in J_i$, we have that $\beta=\sum\gamma_i$. Therefore, by Lemma~\ref{lemma:strictcont}, $\beta$ is strictly-strictly continuous.

\eqref{thesis:redprodfindim2}:
Let $\psi$ be the strictly-strictly continuous $^*$-homomorphism constructed in \eqref{thesis:redprodfindim1}, $q=\psi(1)$, and $\pi\colon\prod M_n(\ce)\to\prod M_n(\ce)/\bigoplus M_n(\ce)$ be the canonical quotient map. We argue as in Proposition~\ref{prop:liftoneverything} to get that $\CI=\SP(\NN)$. We first prove that $q-1\in A$. For this, note that $\pi_A(q) \geq\pi_A(\psi(p_X))=\varphi(\pi(p_X))$ for every $X\subseteq\NN$ with $X\in\CI$. Since $\varphi$ is an isomorphism, $\varphi^{-1}(\pi_A(q))=1$, and therefore $q-1\in A$. Fix $x\in\prod M_n(\ce)$ and $y\in\SM(A)$ such that $\varphi(\pi(x))=\pi_A(y)$, and let
\[
\CI_x=\{X\subseteq\NN\mid \psi(p_X)(\psi(x)-y)\in A\}.
\]
Since $\psi$ is strictly-strictly continuous, $x$ and $y$ are fixed, and $A$ is Borel in the strict topology of $\SM(A)$, $\CI_x$ is Borel. On the other hand, $\CI\subseteq\CI_x$ and so $\CI_x=\SP(\NN)$, as all proper dense Borel ideals are meager (Proposition~\ref{prop:JT2}). As $\NN\in\CI_x$ we have
\[
\pi_A(\psi(x))=\varphi(\pi(x)).
\]
Let $q_n=\psi(1_n)$.By continuity of $\psi$, $\{\sum_{n\leq m}q_n\}_m$ is an increasing sequence of projections converging (strictly) to $q$, and so is an approximate identity of $qAq$. Also, whenever $X\subseteq\NN$, the projection $q_X=\psi(p_X)$ is such that if $a\in\SM(A)$ then
\[
q_Xa-aq_X\in A,
\]
 as $\pi(p_X)$ is central in $\SQ(A)$. Using Proposition~\ref{prop:centralproj}, it is easy to show that there is $n_0$ such that $q_n$ is central in $(q-q_{[0,n_0]})A(q-q_{[0,n_0]})$ for all $n\geq n_0$. In particular
\[
(q-q_{[0,n_0]})A(q-q_{[0,n_0]})\cong \bigoplus_{n>n_0}q_nAq_n.
\]
\begin{claim}
There is $k\geq n_0$ such that if $n>k$ then $\psi\restriction M_n(\ce)\colon M_n(\ce)\to q_nAq_n$ is an isomorphism.
\end{claim}
\begin{proof}
From that $\psi$ is a lift of $\varphi$ and Proposition~\ref{prop:FA.ApproxStruct}, it follows that there is $k$ such that if $n\geq k$ then $\psi\restriction M_n(\ce)$ is injective. We want to show that $\varphi\restriction M_n(\ce)$ is eventually surjective. If not, there is an infinite sequence $k_i$ such that $\varphi\restriction M_{k_i}(\ce)$ is not surjective. In this case, the vector space $\varphi\restriction M_{k_i}(\ce)$ is properly contained in the vector space $q_{k_i}Aq_{k_i}$, and so there is $a_i\in q_{k_i}Aq_{k_i}$ with $d(a_i,\varphi\restriction M_{k_i}(\ce))=1$ and $\norm{a_i}=1$. Fix $\bar a=\sum a_i\in\SM(A)\setminus A$. We can find $\bar b\in\prod M_n(\ce)$ such that
\[
\varphi^{-1}(\pi_A(\bar a))=\pi(\bar b).
\]
Since $\psi$ is a lifting of $\varphi$, we have that $\norm{q_n(\psi(\bar b)-\bar a)q_n}\to 0$, a contradiction.
\end{proof}
By setting $\tilde q=q-q_{[0,n_0]}$ we have the thesis.
\end{proof}

\begin{remark}\label{rem:CH}
It is possible to find a unital infinite dimensional $\Cstar$-algebra $A$ such that, under $\CH$, we have that $\prod M_n(\ce)/\bigoplus M_n(\ce)\cong\ell_\infty(A)/c_0(A)$. As we have seen in Theorem~\ref{thm:trivialredprod}, such an isomorphism cannot exist in $\ZFC$. We sketch this argument: in \cite[\S3]{Ghasemi.FFV} it was shown that if $A_n$'s are $\Cstar$-algebras whose theory (in continuous model theory, see \cite{bourbaki}) converges (in the weak $^*$-topology) to the theory of $B$, then $\prod A_n/\bigoplus A_n$ and $B$ are elementary equivalent. Let $B=\prod M_n(\ce)/\bigoplus M_n(\ce)$, and let $A$ be a separable $\Cstar$-algebra elementary equivalent to $B$. Then $B$ and $\ell_\infty(A)/A$ are elementary equivalent, and, under $\CH$, of size $\aleph_1$. Since reduced products of sequences of separable $\Cstar$-algebras are countably saturated (\cite{Farah-Shelah.RCQ}), and all countably saturated elementary equivalent $\Cstar$-algebras of size $\aleph_1$ must be isomorphic, we have the thesis. We thank the referee for this remark.
\end{remark}

\section{Open questions}\label{sec:conclusion}
The conclusion of Theorem~\ref{thm:trivialredprod} motivates several questions concerning approximate isomorphisms, and how much structure it is preserved by such maps. By Theorem~\ref{thm:apmapstrivial}\eqref{ulam1}, the consistency of the statement `all isomorphisms of reduced products of element in $\mathcal C$ are algebraically trivial', amounts to whether $\mathcal C$ is Ulam stable (see Definition~\ref{defin:ulam}). The only classes of $\Cstar$-algebras known to be Ulam stable are the one of finite-dimensional algebras (\cite[\S5]{Farah.C}) and of abelian algebras (\cite{Semrl.USAbel}), and we would like to enlarge such list.

\begin{question}\label{ques5}
Which classes of $\Cstar$-algebras are Ulam stable?
\end{question}

A starting approach would be to try to combine techniques of \cite{Farah.C,MKAV.UC} and \cite{Semrl.USAbel} to approach Ulam stability questions for the class of separable subhomogeneous $\Cstar$-algebras (or even continuous trace). 

Even though we cannot perturb uniformly approximate maps to homomorphisms, we can still ask whether almost isomorphic algebras can be represented as close objects on the same Hilbert space. (Kadison-Kastler proximity was defined in \S\ref{sec:apmaps}). The reason we ask the following is that for well behaved algebras Kadison-Kastler proximity implies isomorphism: this is the case of nuclear $\Cstar$-algebras, see \cite[Theorem A]{CSSWW}.

\begin{question}\label{ques2}
Let $\mathcal C$ be a class of $\Cstar$-algebras. Does, for every $\varepsilon>0$, there exist $\delta>0$ such that whenever $A\in \mathcal C$ and $B\in \mathcal C$ are $\delta$-isomorphic then there are faithful representations $\rho_A \colon A\to \mathcal{B}(H)$ and $\rho_B \colon B\to\mathcal{B}(H)$ such that $d_{KK}(\rho_A[A],\rho_B[B]) <\varepsilon$?
 \end{question}

A positive answer to Question~\ref{ques2} for the class of separable nuclear $\Cstar$-algebras would imply that the class of separable nuclear $\Cstar$-algebras is uniformly classifiable by approximate isomorphisms.

 Lastly, we want to study which properties are stable under approximate isomorphisms. Natural properties such as nuclearity or simplicity are preserved by Kadison-Kastler proximity. 
\begin{question}\label{ques1}
 Suppose $A$ and $B$ are $\Cstar$-algebras which are $\varepsilon$-isomorphic, for a small $\varepsilon$. What structure in $A$ must occur in $B$? In particular, are `being nuclear' and `being simple' stable under approximate isomorphisms?
\end{question}

As a consequence of an unpublished result of Kirchberg (see \cite[Proposition 3.14.1]{bourbaki}), if a nuclear $A$ is not subhomogeneous, then $A$ is elementary equivalent to a nonexact $\Cstar$-algebra $B$, which can be chosen to be separable. This applies, for example, to all infinite-dimensional nuclear simple algebras. Therefore, by the argument of Remark~\ref{rem:CH} (see also Theorem~\ref{saeed}), if $A_n$ are nuclear simple unital separable $\Cstar$-algebras, one can find separable unital $\Cstar$-algebras $B_n$ such that all the $B_n$'s are nonexact and nonsimple, but under CH,  there is an isomorphism between $\prod A_n/\bigoplus A_n$ and $\prod B_n/\bigoplus B_n$. 
If such an isomorphism exists in $\ZFC$, then it must exists in a model of $\OCA+\MA_{\aleph_1}$, and therefore it is asymptotically algebraic. The same argument as in the proof of Theorem~\ref{thm:apmapstrivial}(2) would give that the class of nuclear separable simple $\Cstar$-algebras is not stable under approximate isomorphisms.  Viceversa, if the class of nuclear separable simple $\Cstar$-algebras is not stable under approximate isomorphisms one can construct an asymptotically algebraic isomorphism between two reduced products $\prod A_n/\bigoplus A_n$ and $\prod B_n/\bigoplus B_n$ where each $A_n$ is a unital, nuclear, separable and simple $\Cstar$-algebra and each $B_n$ is a unital separable $\Cstar$-algebra which is not nuclear and simple but does not have any central projection.

We conclude by proposing a conjecture on coronas of stabilizations of unital $\Cstar$-algebras. If $A$ and $B$ are $\Cstar$-algebras, where $A\subseteq B$, $A$ is hereditary in $B$ if for all positive $a\in A$ and $b\in B$, $b\leq a$ implies $b\in A$. $A$ is full in $B$ if the ideal generated by $A$ equals $B$.  If $A$ and $B$ are unital separable $\Cstar$-algebras, then $A\otimes\mathcal K(H)\cong B\otimes\mathcal K(H)$ if and only if $A$ is isomorphic to a full hereditary $\Cstar$-subalgebra of $B\otimes\mathcal K(H)$. Fix two separable unital $\Cstar$-algebras $A$ and $B$, where $A$ has the metric approximation property. Suppose that $\SQ(A\otimes\mathcal K(H))\cong\SQ(B\otimes\mathcal K(H))$ are isomorphic in a model of $\OCA+\MA_{\aleph_1}$. A modification of the proof of Theorem~\ref{thm:trivialredprod} provides a descending sequence $\varepsilon_n\to 0$ and a sequence of natural numbers $k_n$ such that there are $\varepsilon_n$-$^*$-embeddings $\varphi_n\colon A\to M_{k_n}(B)$. It is unclear whether we can perturb these to on-the-nose embeddings, or whether we can find, from the existence of such an embedding, a full hereditary copy of $A$ inside $M_{k_n}(B)$, for some $k_n$. This motivates the following:

\begin{conjecture}
 Assume $\OCA+\MA_{\aleph_1}$. Let $A$ and $B$ be unital, separable $\Cstar$-algebras. Then $\SQ(A\otimes\mathcal K(H))\cong\SQ(B\otimes\mathcal K(H))$ if and only if $A\otimes\mathcal K(H)\cong B\otimes\mathcal K(H)$.
\end{conjecture}
\bibliographystyle{plain}
\bibliography{library}
\end{document}